\documentclass[reqno]{amsart}
\usepackage{graphicx}
\usepackage{xcolor,bbm,easybmat,bm}
\usepackage{amsmath,amsfonts,amsthm,amscd,amssymb,mathrsfs,esint}
\usepackage{enumerate}
\usepackage[shortlabels]{enumitem}

\DeclareMathOperator{\supp}{supp}
\DeclareMathOperator{\loc}{loc\ }
\DeclareMathOperator{\divg}{div}
\DeclareMathOperator{\dist}{dist}
\DeclareMathOperator{\BMO}{BMO}
\newtheorem{thm}{Theorem}[section]
\newtheorem{cor}{Corollary}[section]
\newtheorem{lem}[thm]{Lemma}
\newtheorem{prop}{Proposition}[section]
\newtheorem{defn}{Definition}[section]

\newtheorem{re}{Remark}[section]

\theoremstyle{remark}

\numberwithin{equation}{section}
\newcommand{\norm}[1]{\left\Vert#1\right\Vert}
\newcommand{\abs}[1]{\left\vert#1\right\vert}
\newcommand{\br}[1]{\left(#1\right)}
\newcommand{\set}[1]{\left\{#1\right\}}
\newcommand{\Real}{\mathbb R}
\newcommand{\Rn}{\mathbb R^n}
\newcommand{\Rpl}{\mathbb R_+^{n+1}}
\newcommand{\eps}{\varepsilon}

\newcommand{\vp}{\varphi}
\newcommand{\wvp}{\widetilde{\varphi}}
\newcommand{\wt}{\widetilde}
\newcommand{\bdy}{\partial}
\newcommand{\Dt}{\partial_t}
\newcommand{\Di}{\partial_i}
\newcommand{\Dj}{\partial_j}
\newcommand{\act}[1]{\langle#1\rangle_{\widetilde{W}^{-1,2},W^{1,2}}}

\newcommand{\bd}[1]{\mathbf{#1}}
\newcommand{\mP}{\mathcal{P}}
\newcommand{\1}{\mathbbm{1}}

\begin{document}
\title[Elliptic operators having a BMO anti-symmetric part]{The Dirichlet problem for elliptic operators having a BMO anti-symmetric part}
\author{Steve Hofmann}
\author{Linhan Li}
\author{Svitlana Mayboroda}
\author{Jill Pipher}
\thanks{S. Hofmann acknowledges support of the National Science Foundation (grant number DMS-1664047, DMS-2000048).
S. Mayboroda is supported in part by the
NSF RAISE-TAQS grant DMS-1839077 and the Simons foundation grant 563916, SM.}

\newcommand{\Addresses}{{
  \bigskip
  \footnotesize

  Steve Hofmann, \textsc{Department of Mathematics, University of Missouri, Columbia, MO 65211, USA}\par\nopagebreak
  \textit{E-mail address}: \texttt{hofmanns@missouri.edu}

  \medskip

  Linhan Li, \textsc{Department of Mathematics, Brown University, Providence, RI 02912 USA}
  \emph{Present address: School of Mathematics, University of Minnesota, Minneapolis, MN 55455 USA}\par\nopagebreak
  \textit{E-mail address}: \texttt{linhan\_li@alumni.brown.edu}

  \medskip

  Svitlana Mayboroda, \textsc{School of Mathematics, University of Minnesota, Minneapolis, MN 55455, USA}\par\nopagebreak
  \textit{E-mail address}: \texttt{svitlana@math.umn.edu}
  
  \medskip
  
   Jill Pipher, \textsc{Department of Mathematics, Brown University, Providence, RI 02912 USA}\par\nopagebreak
  \textit{E-mail address}: \texttt{jpipher@math.brown.edu}

}}

\maketitle

\begin{abstract} The present paper establishes the first result on the absolute continuity of elliptic measure with respect to the Lebesgue measure for a divergence form elliptic operator with non-smooth coefficients that have a $\BMO$ anti-symmetric part. In particular, the coefficients are not necessarily bounded. We prove that the Dirichlet problem for elliptic equation ${\rm div}(A\nabla u)=0$ in the upper half-space $(x,t)\in\mathbb{R}^{n+1}_+$ is uniquely solvable when $n\ge2$ and the boundary data is in $L^p(\mathbb{R}^n,dx)$ for some $p\in (1,\infty)$. 
This result is equivalent to saying that the elliptic measure associated to $L$ belongs to the $A_\infty$ class with respect to the Lebesgue measure $dx$, a quantitative version of absolute continuity.
\end{abstract}

\tableofcontents

\section{Introduction and statement of main results}

Motivated by questions about the behavior of solutions of elliptic and parabolic equations with low regularity drift terms, Seregin, Silvestre, {\v{S}}ver{\'a}k, and Zlato{\v{s}} (\cite{seregin2012divergence}) investigated equations such as $-\Delta u+\bd c\cdot\nabla u=0$ and $\Dt u+\bd c\cdot\nabla u-\Delta u=0$, where $c$ is a divergence-free vector field in $\Rn$. They discovered that the divergence-free condition can be utilized to relax the regularity assumptions on $\bd c$ under which one can obtain the Harnack inequality and other regularity results for solutions. It turns out that the interior regularity theory of De Giorgi, Nash, and Moser can be carried over to elliptic equations with $\bd c\in \BMO^{-1}$, and to parabolic equations with $\bd c\in L^{\infty}(\BMO^{-1})$. Generalizing to elliptic or parabolic equations in divergence form, this condition is equivalent to assuming that the coefficient matrix $A$ of the operator $L=-\divg(A\nabla)$ can be decomposed into an $L^{\infty}$ elliptic symmetric part and an unbounded anti-symmetric part in a certain function space. To be precise, the anti-symmetric part should belong to the John-Nirenberg space $\BMO$ (bounded mean oscillation) in the elliptic case, and to $L^{\infty}(\BMO)$ in the parabolic case.
The space $\BMO$ plays a key role in two ways. First, this space has the right scaling properties which arise naturally in the iterative arguments of De Giorgi-Nash-Moser. Secondly, the $\BMO$ condition on the anti-symmetric part of the matrix enables one to properly define weak solutions. This latter fact follows essentially from the div-curl lemma appearing in the theory of compensated compactness (\cite{coifman1993compensated}), and the details can be found in \cite{seregin2012divergence} or \cite{li2019boundary}.

The interior regularity results of Seregin, Silvestre, {\v{S}}ver{\'a}k, and Zlato{\v{s}} lead naturally to questions about
boundary regularity. In \cite{li2019boundary}, the second and the fourth authors studied the boundary behavior of weak solutions. 
It turns out that many results for elliptic operators with bounded, measurable coefficients can be extended to this setting. For example, they proved the boundary H\"older regularity of the solution,  established the existence of the elliptic measure $\omega$ associated to these operators, and offered multiple characterizations of the mutual absolute continuity of the elliptic measure and the surface measure in Lipschitz domains. This work laid out the background necessary to 
launch the investigation into boundary value problems for elliptic operators having a $\BMO$ anti-symmetric part. 

In the present paper we establish the first result pertaining to absolute continuity of the elliptic measure for operators with $\BMO$ anti-symmetric part and well-posedness of the Dirichlet boundary value problem with $L^p$ data. 

In order to frame our results in the context of the currently existing elliptic theory, let us review some historical milestones. In the middle of the 20th century the theory of boundary value problems mainly concentrated on the case when coefficients of the underlying equations and domains exhibit some amount of smoothness. The past 30-40 years have brought great developments in the study of elliptic measure and boundary value problems for operators with non-smooth bounded measurable coefficients. The background theory of weak solutions, Green function estimates, maximum principle, and similar results were extended to all divergence form elliptic operators with bounded measurable coefficients. It turned out, however, that the question of absolute continuity of the resulting elliptic measure with respect to the Lebesgue measure on the boundary, or, equivalently, of well-posedness of the Dirichlet boundary value problem with boundary data in $L^p$, is much more delicate. First of all, examples have been found that show such results can not be expected for all elliptic operators and some regularity of the coefficients in the transversal direction to the boundary is, in fact, necessary \cite{caffarelli1981completely}, \cite{modica1980construction}. In light of these examples, the initial efforts concentrated on the study of operators whose coefficients are constant in the transverse direction to the boundary. Later results have extended the theory to the optimal regularity of the coefficients, expressed in terms of a {\it Carleson measure condition}. In this survey, and in this paper, we shall concentrate on the fundamental case where the domain is the upper half-space $\Rpl=\set{(x,t)\in\Rn\times(0,\infty)}$ and the coefficients of the operator are independent of the transverse direction, that is, $t$-independent. The first breakthrough in this direction was the 1981 paper of Jerison and Kenig \cite{jerison1981dirichlet} which established well-posedness of the Dirichlet problem and the absolute continuity of the elliptic measure for operators with symmetric bounded measurable $t$-independent coefficients on $\Real^{n+1}_+$ and, by a change of variables, above a graph of a Lipschitz function. A seemingly innocent assumption of symmetry turned out to be critical and it took 20 years to extend these results to non-symmetric operators in dimension 2 \cite{kenig2000new} and more than 30 years to non-symmetric operators in any dimension \cite{hofmann2015square}. The 1981 work of Jerison and Kenig relied on the beautiful and powerful Rellich identity which roughly speaking says that the $L^2$ norms of the normal and tangential derivatives of solutions on the boundary are comparable. It is proved by an integration by parts argument invoking the symmetry of the coefficients. However, not only the method of the proof of the Rellich identity, but the $L^2$ equivalence of the norms of the normal and tangential trace of the solution itself fails when the coefficients are not necessarily symmetric. This has been demonstrated in \cite{kenig2000new}, where the authors established extremely useful characterizations of solvability of the Dirichlet problem in $L^p$  in terms of the square function/non-tangential maximal function estimates (in any dimension), a method that made possible many later developments including the present paper, and resolved the question of absolute continuity of elliptic measure with respect to the Lebesgue measure for $t$-independent non-symmetric operators in dimension 2. Unfortunately, many ingredients in the argument in \cite{kenig2000new} rely heavily on the space being 2 dimensional. For example, the 2-d case relies on a change of variable argument that does not carry forward to higher dimensions. Only 15 years later these results have been finally extended to multidimensional setting. In \cite{hofmann2015square} the authors established the square function/non-tangential maximal function estimates for solutions to the $t$-independent, non necessarily symmetric, operators on $\Real^{n+1}_+$ for all $n$, and as a result, absolute continuity of the elliptic measure with respect to the Lebesgue measure and well-posedness of the Dirichlet boundary value problem in $L^p$. The method involved a new pull-back/push-forward sequence based on the Hodge decomposition of the coefficients, the celebrated solution to the Kato problem \cite{auscher2002solution}, the square function/non-tangential maximal function estimates for the heat semigroup, and many other elements. The method has later been streamlined in \cite{Auscher2017parabolic} to avoid an explicit pull-back/push-forward on Lipschitz domain -- an important development in our context. 

As we mentioned above, all of these results as well as many elements of the surrounding elliptic theory have been restricted to the context of {\it bounded} measurable coefficients. The present paper pioneers the consideration of the $\BMO$ anti-symmetric part, an optimal structural assumption on the coefficients. The lack of boundedness invalidates many of the arguments that we have described above. We shall discuss all the new difficulties and some critical junctures of our proof in Section~\ref{overview sec}, after the statement of Theorem~\ref{main thm-carl}. These new difficulties include a new Hodge decomposition beyond $L^2$, and new estimates for the Riesz transforms, square functions and non-tangential maximal functions associated to the heat semigroup. Changes of variables and other techniques that preserved the boundedness properties of coefficients are lost in the presence of $\BMO$ coefficients. There are many other issues which require a more technical discussion and we refer an interested reader to Section~\ref{overview sec}.

We now rigorously state our results. Let $A=A(x)$ be an $(n+1)\times(n+1)$ matrix of real, t-independent coefficients such that 
\begin{enumerate}
    \item The symmetric part $A^s=\frac{1}{2}(A+A^\intercal)=\br{A_{ij}^s(x)}$ is $L^{\infty}(\Rn)$, and satisfies the ellipticity condition
    \begin{equation}\label{ellipticity def}
    \begin{aligned}
     \lambda_0\abs{\xi}^2\le \langle A^s(x)\xi,\xi\rangle&=\sum_{i,j=1}^{n+1}A^s_{ij}(x)\xi_i\xi_j\quad\text{for all }\xi\in\Real^{n+1}, x\in\Rn,\\  \text{and}\quad\norm{A^s}_{L^\infty(\Rn)}&\le\lambda_0^{-1}, 
    \end{aligned}
     \end{equation}
     for some $0<\lambda_0<1$.
    \item The anti-symmetric part $A^a=\frac{1}{2}(A-A^{\intercal})=(A^a_{ij}(x))$ is in the space $\BMO(\Rn)$, with
\begin{equation}\label{BMO def}
    \norm{A^a_{ij}}_{\BMO}:=\sup_{Q\subset\Rn}\fint_{Q}\abs{A^a_{ij}-(A^a_{ij})_{Q}}dx\le\Lambda_0
\end{equation}
for some $\Lambda_0>0$. Here $(f)_Q$ denotes the average $\frac{1}{\abs{Q}}\int_Qf(x)dx$.
\end{enumerate}

We define in $\Real^{n+1}$ a second order divergence form operator 
\begin{equation}\label{L defn}
    L=-\divg_{x,t} (A(x)\nabla_{x,t}),
\end{equation}
which is interpreted in the sense of maximal accretive operators via sesquiliniear form. 
We say that $u$ is a weak solution to the equation $Lu=0$ in $\Rpl$ if $u\in W_{\loc}^{1,2}(\Rpl)$ and 
\begin{equation}\label{weaksol def}
    \iint_{\Rpl} A\nabla u\cdot \nabla v=0
\end{equation}
for all $v\in C_0^{\infty}(\Rpl)$.

We consider the $L^p$ Dirichlet problem $(D)_p$ for the equation $\divg (A\nabla u)=0$ in the upper half-space $\Rpl$ when $n\ge 2$. We shall denote by $\mu$ the Lebesgue measure in $\Rn$. Sometimes we simply denote it by $dx$, and the meaning should be clear from context. For $p\in(1,\infty)$, we say the Dirichlet problem for $L^p(\Rn,d\mu)$ data is solvable if for each $f\in L^p(\Rn, d\mu)$, there is a solution $u\in W_{\loc}^{1,2}(\Rpl)$ such that
\begin{equation*}(D)_p\quad 
    \begin{cases}
    Lu=0 \quad\text{in } \Rpl,\\
    u\to f\in L^p(\Rn,d\mu) \text{ non-tangentially } \mu \text{-a.e. on } \Rn\\
    Nu\in L^p(\Rn, d\mu).
  \end{cases}
\end{equation*}
Here, $N(u)$ denotes the non-tangential maximal function of $u$:
\begin{equation}\label{nt def}
    N(u)(x):= \sup_{(y,t):\abs{x-y}< t}\abs{u(y,t)},
\end{equation}
and $u$ converges to $f$ non-tangentially means
\[
\lim_{(y,t)\to (x,0), (y,t)\in\Gamma(x)}u(y,t)=f(x),
\]
where $\Gamma(x)=\set{(y,t)\in \Rn\times \Real_+: \abs{y-x}<t}$.

The main result of this paper is that the $L^p$ Dirichlet problem for $L$ in $\Rpl$ is uniquely solvable for some $p\in (1,\infty)$ sufficiently large:
\begin{thm}\label{main thm}
Let $A$ be a matrix of real, t-independent coefficients satisfying \eqref{ellipticity def} and \eqref{BMO def}. Then for some $p\in (1,\infty)$, for each $f\in L^p(\Rn,\mu)$, there exists a unique $u$ that solves $(D)_p$ for $L=-\divg(A\nabla)$ in the upper half-space $\Rpl$ when $n\ge2$. 
\end{thm}

This result is equivalent to quantitative absolute continuity of elliptic measure with respect to the Lebesgue measure, the $A_\infty$ property - see the next Section.

For the uniqueness part of the statement, we actually prove the following Fatou-type result.
\begin{thm}\label{Fatou thm}
Let $A$ be an $(n+1)\times(n+1)$ matrix of real coefficients. Assume that the symmetric part of $A$ is bounded and elliptic, and that the anti-symmetric part is in the space $\BMO(\Rpl)$.  Assume that $(D)_p$ is solvable for $L=-\divg(A\nabla)$ in $\Rpl$ for some $p\in(1,\infty)$. Suppose that $u$ satisfies
\[
\begin{cases}
Lu=0 \quad \text{in}\quad\Real_+^{n+1},\\
Nu\in L^p(\Rn,d\mu).
\end{cases}
\]
Then, the non-tangential limit of $u$ exists a.e. in $\Rn$ (and is denoted by $u|_{\partial\Rpl}^{N.T.}$), $u|_{\partial\Rpl}^{N.T.}\in L^p(\Rn,d\mu)$, and
\[
u(X)=\int_{\Rn}u|_{\partial\Rpl}^{N.T.}(y)k(X,y)d\mu(y),
\]
where $k(X,y)$ is defined in \eqref{def k}.
\end{thm}
One can see that this result is stronger than uniqueness. Notice that in this theorem, we do not assume that $A$ is $t$-independent. Moreover, for $u$, in contrast to a solution to $(D)_p$, we do not assume a priori that it converges non-tangentially.

\section{An overview of the proof of Theorem \ref{main thm}}\label{overview sec}

As mentioned in the introduction, it is shown in \cite{kenig2016square} that some Carleson measure estimate implies some quantitative mutual absolute continuity, namely, the $A_\infty$ condition, between the elliptic measure associated to an elliptic operator with real, $L^\infty$ coefficients and the Lebesgue measure. In \cite{li2019boundary}, it is verified that this result also holds for elliptic operators having a $\BMO$ anti-symmetric part. To understand the precise statement and its connection to Theorem \ref{main thm}, we first need some notations and definitions.

For a set $E\subset\Rn$, we denote its Lebesgue measure $\mu(E)$ by $\abs{E}$.
For any cube $Q\subset\Rn$, let $x_Q$ and $l(Q)$ be the center and side length of $Q$, respectively. Let $X_Q:= (x_Q,l(Q))$ denote the corkscrew point in $\Rpl$ relative to $Q$. For $x\in\Rn$ and $r>0$, we use $T(x,r):=\set{Y\in\Real^{n+1}_+: \abs{Y-(x,0)}<r}$ to denote half balls in $\Real^{n+1}_+$, and $\Delta(x,r):=\set{y\in\Rn:\abs{y-x}<r}$ to denote balls in $\Rn$. We shall simply write $T_R$ and $\Delta_R$ for $T(0,R)$ and $\Delta(0,R)$, respectively. We use $C_0(\Omega)$ to denote the set of continuous functions with compact support on $\Omega$. $W_0^{1,2}(\Omega)$ is defined to be the closure of $C_0^\infty(\Omega)$ in $W^{1,2}(\Omega)$. 

In a bounded Lipschitz domain $\Omega$, for each $X\in\Omega$, the elliptic measure $\omega^X$ is constructed in \cite{li2019boundary} to be the measure on $\bdy\Omega$, such that $u(X)=\int_{\bdy\Omega}hd\omega^X$ solves the Dirichlet problem for continuous boundary data $h\in C(\bdy\Omega)$ in the sense that $\divg(A\nabla u)=0$ in $\Omega$, with $u\in W_{loc}^{1,2}(\Omega)\cap C(\overline{\Omega})$, and $u=h$ on $\bdy\Omega$. 

The elliptic measure on $\Rn$ can be defined as follows. Let $f\in C_0(\Rn)$ with $\supp f\subset \Delta_{R_0}$ for some $R_0>0$. We define an extension of $f$ (still denoted by $f$) which is equal to $0$ on $\Real^{n+1}_+\setminus T_{R_0}$. Then for all $R\ge R_0$, $f^\pm\in C(\bdy T_R)$, where $f^\pm:=\max\set{\pm f,0}$. For each $X\in T_R$, let $\omega_R^X$ be the elliptic measure on $\bdy T_R$. Then $u_R^\pm(X):=\int_{\bdy T_R}f^\pm d\omega_R^X$ solves the Dirichlet problem in $T_R$ with boundary data $f^\pm$ for all $R\ge R_0$. For any $R_0\le R_1\le R_2$ and $X\in T_{R_1}$, the maximum principle (\cite{li2019boundary} Lemma 4.7) implies that $u_{R_1}^\pm(X)\le u_{R_2}^\pm(X)\le \norm{f^\pm}_{L^\infty(\Rn)}$. Therefore, we can define $u$ as the monotone limit
\begin{equation*}
    u(X):=\lim_{R\to \infty}\br{u_R^+(X)-u_R^-(X)} \qquad\forall\, X\in \Real^{n+1}_+.
\end{equation*}
And we have 
\begin{equation}
\norm{u}_{L^\infty(\Real^{n+1}_+)}\le \norm{f}_{L^\infty(\Rn)}.    
\end{equation} 
The mapping $f \mapsto u(X)$ is a positive bounded linear functional on $C_0(\Rn)$, and thus can be extended to a positive bounded linear functional on the set of all continuous functions on $\Rn$ that converge to $0$ uniformly at infinity. The Riesz Representation Theorem implies that there exists a regular Borel measure $\omega^X$ on $\Rn$ such that $u(X)=\int_{\Rn}fd\omega^X$. This $\omega^X$ is defined to be the elliptic measure on $\Rn$. One can show, by H\"older continuity of solutions and Caccioppoli's inequality, that $u\in W_{loc}^{1,2}(\Real^{n+1}_+)$ and solves the Dirichlet problem in $\Real^{n+1}_+$ with boundary data $f\in C_0(\Rn)$.

For any $X$, $X_0\in \Rpl$, the Harnack principle implies that $\omega^X$ and $\omega^{X_0}$ are mutually absolute continuous. Define the kernel function $K(X_0,X,y)$ to be the Radon-Nikodym derivative  $K(X_0,X,y):=\frac{d\omega^X}{d\omega^{X_0}}(y)$. And define
\begin{equation}\label{def k}
    k(X,y):=\frac{d\omega^X}{d\mu}(y), \qquad\text{for }y\in\Rn.
\end{equation} 
Note that 
\[
k(X,y)=K(X_0,X,y)k(X_0,y)\qquad \text{for any } X, X_0\in\Rpl,\, y\in\Rn.
\]

\begin{defn}[$A_{\infty}$]
 For any cube $Q_0\subset\Rn$, we say that a non-negative Borel measure $\omega$ belongs to $A_{\infty}(Q_0)$ (or $A_{\infty}(d\mu)$) with respect to the Lebesgue measure $d\mu$, if there are positive constants $C$ and $\theta$ such that for every cube $Q\subset Q_0$ (or $Q\subset\Rn$, respectively), and every Borel set $E\subset Q$, 
 \[
 \omega(E)\le C\br{\frac{\abs{E}}{\abs{Q}}}^{\theta}\omega(Q),
 \]
 where $C$ and $\theta$ are independent of $E$ and $Q$.
\end{defn}
We note that in the sequel, we shall actually establish
this local $A_\infty$ property in a scale-invariant way, that is, with constants that are independent of $Q_0$ (see Theorem \ref{main thm-carl}).

\begin{lem}[\cite{kenig2016square} Corollary 3.2, \cite{li2019boundary} Theorem 8.5]
Assume that $A$ satisfies  \eqref{ellipticity def} and \eqref{BMO def} in $\Rpl$, and define $L$ as in \eqref{L defn}. Assume that there is some uniform constant $C<\infty$ such that for all Borel sets $H\subset\Rn$, the weak solution $u$ to the Dirichlet problem 
  $$
  \begin{cases}
  Lu=0 \quad\text{in }\Rpl\\
  u=\chi_H \quad\text{on }\bdy\Rpl
  \end{cases}
  $$
  satisfies the following Carleson bound
  \begin{equation}\label{Carl}
     \sup_{Q\subset\Rn}\frac1{\abs{Q}}\int_0^{l(Q)}\int_Q\abs{\nabla u(x,t)}^2t\,dx\le C. 
  \end{equation}
  Here $l(Q)$ denotes the side length of the cube $Q$.
 Then for any cube $Q_0\subset\Rn$, $\omega^{X_{Q_0}}\in A_{\infty}(Q_0)$.
\end{lem}

It is well-known from the general theory of weights that the $A_\infty$ condition $\omega^{X_{Q_0}}\in A_{\infty}(Q_0)$ implies that there is some $q\in (1,\infty)$ such that $k(X_{Q_0},\cdot)$ satisfies the following reverse H\"older inequality: for any $\Delta\subset Q_0$,
\begin{equation}\label{RHq}
  \br{\frac{1}{\abs{\Delta}}\int_{\Delta}k(X_{Q_0},y)^qd\mu(y)}^{1/q}\lesssim \frac{1}{\abs{\Delta}}\int_{\Delta}k(X_{Q_0},y)d\mu(y),
 \end{equation}
 where the implicit constant depends only on $\lambda$, $\Lambda$ and $n$.
Moreover, by estimates for the kernel function $K$, one can show that
\begin{equation}\label{k in Lq}
    \text{for any } X=(x,t)\in\Rpl, k(X,\cdot)\in L^q(\Rn,d\mu),
\end{equation}
where $q$ is the same as in \eqref{RHq}.
The proof can be found in \cite{hofmann2018carleson}, where these results are proved for degenerate elliptic operators in the upper half-space. The argument of \cite{hofmann2018carleson} applies to the operators under discussion. We also remark that for bounded (Lipschitz) domains, the kernel function estimates used to prove \eqref{k in Lq} for operators with $L^\infty$ coefficients can be found in \cite{kenig1994harmonic} and \cite{caffarelli1981boundary}, while for elliptic operators with $\BMO$ anti-symmetric part these are  verified in \cite{li2019boundary}.

It is known that \eqref{RHq} yields the solvability of $L^p$ Dirichlet problem, with $p\ge q':=\frac{q}{q-1}$. See e.g. \cite{kenig1994harmonic} Theorem 1.7.3, or \cite{hofmann2018carleson} for this argument. Therefore, to prove the existence part of Theorem \ref{main thm}, it suffices to show the Carleson measure estimate \eqref{Carl}. Indeed, we prove the following:

\begin{thm}\label{main thm-carl}
Let $A$ be a matrix of real, t-independent coefficients satisfying \eqref{ellipticity def} and \eqref{BMO def}. Let $L$ be defined as \eqref{L defn}. Then  any bounded weak solution $u$ to $L$ in $\Rpl$ with $\norm{u}_{L^{\infty}}\le1$  satisfies the estimate
\begin{equation}\label{Carl'}
    \int_0^{l(Q)}\int_Q\abs{\nabla u(x,t)}^2t\,dx\lesssim\abs{Q},
\end{equation} for any cube $Q\subset\Rn$, and the implicit constant depends only on the ellipticity constants and the $\BMO$ semi-norm. 
And thus for any cube $Q_0\subset\Rn$, the elliptic measure $\omega^{X_{Q_0}}\in A_\infty(Q_0)$ with constants depending only on the dimension, the ellipticity constant and the $\BMO$ semi-norm.  
\end{thm}


There are many difficulties when the coefficients are not $L^\infty$. We illustrate them by first taking a closer look at the structure of the matrix $A$. We write  
$$
A=\begin{bmatrix}
\begin{BMAT}{c.c}{c.c}
A_{||} & \mathbf{b} \\
\mathbf{c} & d
\end{BMAT}
\end{bmatrix},
$$
where $A_{||}$ denotes the $n\times n$ submatrix of $A$ with entries $(A_{||})_{i,j}$, $1\le j\le n$, $\bd b$ denotes the column vector $(A_{i,n+1})_{1\le i\le n}$, $\bd c$ denotes the row vector $(A_{n+1,j})_{1\le j\le n}$, and $d= A_{n+1,n+1}$. 

We observe that if the coefficients are in $L^\infty$, and in addition, $\divg_x \bd c=0$, then the Carleson measure estimate \eqref{Carl'} follows simply from an integration by parts argument. But even in this case, when having BMO coefficients, difficulties arise. For example, when the coefficients could be in BMO, we cannot bound the integrals
$\iint_{\Rpl}A\nabla u\cdot \nabla\Psi (u\Psi\, t)dxdt$ and $\iint_{\Rpl}\bd c\cdot\nabla_x\Psi u^2\Psi$, which appear from integration by parts. Here, $\Psi$ is a cutoff function that is supported in the box $2Q\times (\epsilon,l(Q))$ and equals to 1 in $Q\times (\epsilon,l(Q))$. To deal with this issue, we shall work with weak solutions to the operator $L_0=-\divg A_0\nabla$, where $A_0$ is defined in \eqref{A0}. We observe in Lemma \ref{L to L0 lem} that a weak solution of $L$ is also a weak solution of $L_0$. This observation enables us to work with the equation $L_0u=0$ in $\Rpl$, for which we can control the BMO coefficients by the John-Nirenberg inequality.

When $\divg_x\bd c\neq 0$, the situation is more complicated, even when coefficients are in $L^\infty$. We define an $n$-dimensional divergence form operator $L_{||}:= \divg A_{||}\nabla$, and its adjoint $L^*_{||}:=-\divg A^*_{||}\nabla$. We highlight three ingredients in the proof of the $A_{\infty}$ condition for elliptic measure associated to operators with $L^{\infty}$, $t$-independent coefficients in \cite{hofmann2015square}:
\begin{enumerate}
    \item An adapted Hodge decomposition of $\bd c$ and $\bd b$.
    \item $L^p$ estimates for square functions involving the ``ellipticized" heat semigroup $\mP_t:= e^{-t^2L_{||}}$ associated to $L_{||}$, and $\mP_t:= e^{-t^2L_{||}^*}$. Some of these estimates reply heavily on the solution to the Kato problem.
    \item $L^p$ estimates for the non-tangential maximal function involving $\mP_t$ and $\mP_t^*$, which enables one to construct a set $F$ with desired properties.
\end{enumerate}
None of these ingredients comes for free when we move to the elliptic operators having a BMO anti-symmetric part. But fortunately, in a recent paper (\cite{HLMPLp}), we were able to obtain the desired $L^p$ estimates for square functions involving $\mathcal{P}_t$ and $\mathcal{P}_t^*$. The arguments for the $L^p$ estimates rely on the $L^p$ estimate for the square root operator $\sqrt{L}$, which is also derived in \cite{HLMPLp}. We note here that in \cite{escauriaza2018kato}, the Kato problem, or the $L^2$ estimate for $\sqrt L$, was solved for elliptic operators having a BMO anti-symmetric part. Previously, the Kato conjecture was proved for operators having the Gaussian property (\cite{hofmann2002solution}) and for elliptic operators in divergence form with complex, bounded coefficients (\cite{auscher2002solution}). 

In Section \ref{Hodge subsec}, we deal with the Hodge decomposition. We point out that we need a $W^{1,2+\epsilon}$ Hodge decomposition because the BMO coefficients require higher integrability, while for $L^\infty$ coefficients, a $W^{1,2}$ Hodge decomposition suffices (see \cite{hofmann2018carleson}). The $L^p$ estimates for the non-tangential maximal function involving $\mP_t$ and $\mP_t^*$ are presented in Section \ref{Lp nt max subsec}.

\subsection{Further reductions of the statement}

Recall that our goal is to derive the Carleson measure estimate \eqref{Carl'}.
Note that this formulation allows us to assume that $A$ is smooth as long as the bounds do not depend on the regularity of the coefficients.

It turns out that we do not need to show \eqref{Carl'} holds for integral over all of the cube $Q$, but only on a subset $F$ of $Q$ that has a big portion of the measure of $Q$. To be precise, we have the following lemma.
\begin{lem}\label{reduc Carl lem}
Let $u$ be a weak solution to $L$ in $\Rpl$. Assume that  there is a uniform constant $c$, and, for each cube, $Q\subset\Rn$ there is a Borel set $F\subset Q$, with $\abs{F}\ge c\abs{Q}$, such that
\begin{equation}\label{Carl F}
    \int_0^{l(Q)}\int_F\abs{\Dt u(x,t)}^2t\,dx\lesssim\abs{Q}, 
\end{equation}
with the implicit constant depending on $c$, $n$, $\norm{u}_{L^{\infty}}$, the ellipticity constants and the BMO semi-norm only, in particular, independent of $Q$ and $F$.

Then $u$ satisfies the Carleson measure estimate \eqref{Carl'}.\end{lem}

The proof of Lemma \ref{reduc Carl lem} requires two steps of reduction. First, one can show by integration by parts and the Caccioppoli inequality on Whitney cubes that
\begin{equation}\label{reduc Carl Dt}
  \int_0^{l(Q)}\int_Q\abs{\nabla u(x,t)}^2t\,dxdt\lesssim\int_0^{2l(Q)}\int_{2Q}\abs{\Dt u(x,t)}^2t\,dxdt+\abs{Q}.  
\end{equation}
The details can be found in \cite{hofmann2015square}. 

Secondly, since the coefficients are independent of $t$, $\Dt u$ is also a weak solution of $L$ (see Appendix \ref{weak sol Appen}, Remark \ref{Dt sol Re}), and thus $\Dt u$ satisfies Harnack Principle and interior H\"older estimates (see \cite{li2019boundary}). This allows us to apply a well-known result for Carleson measures (see, e.g., \cite{auscher2001extrapolation} Lemma 2.14), to deduce from \eqref{Carl F} an apparently stronger bound
\begin{equation*}
    \int_0^{l(Q)}\int_Q\abs{\Dt u(x,t)}^2t\,dx\lesssim\abs{Q}.
\end{equation*}
Combining this with \eqref{reduc Carl Dt}, Lemma \ref{reduc Carl lem} follows. This lemma gives us the freedom to choose the set $F$.

The construction of the set $F$ is presented in Section \ref{set F subsection}. Basically, we will construct $F$ such that on the set $F$, the non-tangential maximal function involving $\mP_t=e^{-t^2L_{||}}$ and $\mP^*_t= e^{-t^2L^*_{||}}$, as well as some other maximal functions are small (see \eqref{F def}). We will exploit this property of the set $F$ in the proof of the Carleson measure estimate. Namely, as long as a term can be bounded by maximal functions showing up in the definition of $F$, then there is hope to control that term with desired bounds.

It turns out that to prove the Carleson measure estimate \eqref{Carl F}, it suffices to prove the following main lemma (see Section \ref{proof of Carl sec}).
\begin{lem}[Main Lemma]\label{main lem}
Let $\sigma$, $\eta\in(0,1)$. Then there exists some finite constant $c=c(\lambda_0,\Lambda_0,n)>0$, and some finite constant $\tilde{c}=\tilde{c}(\sigma,\eta,\lambda_0,\Lambda_0,n)>0$, such that
\begin{equation*}
    J_{\eta,\epsilon}\le (\sigma+c\eta)J_{\eta,\epsilon}+\tilde{c}\abs{Q}.
\end{equation*}
\end{lem}
Here,\begin{equation*}
    J_{\eta,\epsilon}:=\iint_{\Real_+^{n+1}}A_0\nabla u\cdot \nabla u\,\Psi^2 t\,dxdt
\end{equation*}
where $u$ is a bounded weak solution to $Lu=0$ (and thus also a weak solution to $L_0u=0$) in $\Rpl$ with $\norm{u}_{L^{\infty}}\le 1$, and $\Psi=\Psi_{\eta,\epsilon}$ is a cut-off function defined in Section \ref{cutoff section}.

The main lemma is proved in Section \ref{proof of Carl sec}. In the proof, a typical way to deal with the BMO coefficients is to use the anti-symmetry, H\"older's inequality, and John-Nirenberg's inequality. This method inevitably increases the exponent of the integrand, and thus requires some $L^{2+\epsilon}$ estimates. Besides the $W^{1,2+\epsilon}$ Hodge decomposition we mentioned earlier, it is crucial to have an $L^p$ estimate for the cut-off function $\Psi$ (see Lemma \ref{cutoff est lem}), and $L^p$ estimates for the non-tangential maximal functions and square functions that involve semigroups, for $p>2$.

\section{Technical tools}\label{tech tools sect}

\subsection{Some useful results in PDE}
We shall frequently use two results from \cite{giaquinta1983multiple}. We include them here for reader's convenience. 

The first one is useful in proving reverse H\"older type inequalities. 
\begin{lem}[\cite{giaquinta1983multiple} Chapter V Proposition 1.1]\label{GiaProp1.1}
Let $Q$ be a cube in $\Rn$. Let $g\in L^q(Q)$, $q>1$, and $f\in L^s(Q)$, $s>q$, be two nonnegative functions. Suppose
\[
\fint_{Q_R(x_0)}g^qdx\le b\br{\fint_{Q_{2R}(x_0)}gdx}^q+\fint_{Q_{2R}(x_0)}f^qdx+\theta\fint_{Q_{2R}(x_0)}g^qdx
\]
for each $x_0\in Q$ and each $R<\min\set{\frac{1}{2}\dist (x_0,\bdy Q),R_0}$, where $R_0$, $b$, $\theta$ are constants with $b>1$, $R_0>0$, $0\le\theta<1$. Then $g\in L_{\loc}^p(Q)$ for $p\in[q,q+\epsilon)$ and
\[
\br{\fint_{Q_R}g^pdx}^{1/p}\le c\left(\Bigl(\fint_{Q_{2R}}g^qdx\Bigr)^{1/q}+\Bigl(\fint_{Q_{2R}}f^pdx\Bigr)^{1/p}\right)
\]
for $Q_{2R}\subset Q$, $R<R_0$, where $c$ and $\epsilon$ are positive constants depending only on $b$, $\theta$, $q$, $n$ (and $s$).

\end{lem}

In applications, if one can show that
\begin{multline*}
     \fint_{Q_{R}(x_0)}\abs{\nabla u}^2dx\\
     \le b\br{\fint_{Q_{2R}(x_0)}\abs{\nabla u}^{2r}dx}^{1/r}+\fint_{Q_{2R}(x_0)}\abs{f}^2dx+\theta\fint_{Q_{2R}(x_0)}\abs{\nabla u}^2dx
\end{multline*}
for each $x_0\in Q$ and each $R<\min\set{\frac{1}{2}\dist (x_0,\bdy Q),R_0}$, where $b>1$, $r\in(0,1)$ and $\theta\in[0,1)$ are some constants, then by letting $g=\abs{\nabla u}^{2r}$, $q=\frac{1}{r}$ and $f$ be $\abs{f}^{2r}$ in Lemma \ref{GiaProp1.1}, one obtains that
$\abs{\nabla u}\in L_{\loc}^p(Q)$ for $p\in[2,2+\epsilon)$ and
\[
\br{\fint_{Q_R}\abs{\nabla u}^pdx}^{1/p}\le c\left(\br{\fint_{Q_{2R}}\abs{\nabla u}^2dx}^{1/2}+\br{\fint_{Q_{2R}}\abs{f}^pdx}^{1/p}\right)
\]
for $Q_{2R}\subset Q$, $R<R_0$, where $c$ and $\epsilon$ are positive constants depending only on $b$, $\theta$, $r$ and $n$.

\begin{lem}[\cite{giaquinta1983multiple} Chapter V Lemma 3.1]\label{GiaLemma3.1}
Let $f(t)$ be a nonnegative bounded function defined in $[r_0,r_1]$, $r_0\ge 0$. Suppose that for $r_0\le t<s\le r_1$ we have
\[
f(t)\le \br{A(s-t)^{-\alpha}+B}+\theta f(s)
\]
where $A,B,\alpha,\theta$ are nonnegative constants with $0\le\theta<1$. Then for all $r_0\le\rho<R\le r_1$ we have
\[
f(\rho)\le c\br{A(R-\rho)^{-\alpha}+B}
\]
where $c$ is a constant depending on $\alpha$ and $\theta$.
\end{lem}

\subsection{Hardy Norms}\label{Hardy norms subsec}
\begin{defn}
 We say $f\in L^1(\Rn)$ is in the real Hardy space $\mathcal{H}^1(\Rn)$ if 
 \[
 \norm{f}_{\mathcal{H}^1(\Rn)}:=\norm{\sup_{t>0}\abs{h_t*f}}_{L^1(\Rn)}<\infty,
 \]
 where $h_t(x)=\frac{1}{t^n}h\br{\frac{x}{t}}$, and $h$ is any smooth non-negative function on $\Rn$, with $\supp h\subset B_1(0)$ such that $\int_{\Rn}h(x)dx=1$. 
 \end{defn}

\begin{prop}\label{HardyProp1}
Let $1<p<\infty$. Let $u\in\dot{W}^{1,p}(\Rn)$, $v\in \dot{W}^{1,p'}(\Rn)$. Then $\Dj u\Di v-\Di u\Dj v\in\mathcal{H}^1(\Rn)$ for any $1\le i,j\le n$, and
\begin{equation}
  \norm{\Dj u\Di v-\Di u\Dj v}_{\mathcal{H}^1(\Rn)}\lesssim \norm{\nabla u}_{L^p}\norm{\nabla v}_{L^{p'}},
\end{equation}
where the implicit constant depends only on $p$ and dimension.
\end{prop}
We refer to \cite{li2019boundary} and \cite{seregin2012divergence} for its proof.

\begin{prop}\label{HardyProp3}
Let $1<p<\infty$. Let $u\in\dot{W}^{1,p}(\Rn)$, $v\in \dot{W}^{1,p'}(\Rn)$. Then  $\Di(uv)\in\mathcal{H}^1(\Rn)$ for any $1\le i\le n$ with
\begin{equation}
  \norm{\Di(uv)}_{\mathcal{H}^1(\Rn)}\lesssim \norm{u}_{L^p}\norm{\nabla v}_{L^{p'}}+\norm{\nabla u}_{L^p}\norm{v}_{L^{p'}},
\end{equation}
where the implicit constant depends only on $p$ and dimension.
\end{prop}

\begin{prop}\label{HardyProp2}
Let $u$, $v\in W^{1,2}(\Rn)$, and $\vp$ be a Lipschitz function in $\Rn$. Then  $\Dj(uv)\Di\vp-\Di(uv)\Dj\vp\in\mathcal{H}^1(\Rn)$ for any $1\le i,j\le n$, and
\begin{equation*}
  \norm{\Dj(uv)\Di\vp-\Di(uv)\Dj\vp}_{\mathcal{H}^1(\Rn)}\lesssim \norm{u\abs{\nabla \vp}}_{L^2}\norm{\nabla v}_{L^2}+\norm{v}_{L^2}\norm{\abs{\nabla u}\abs{\nabla\vp}}_{L^2},
\end{equation*}
or
\begin{multline*}
    \norm{\Dj(uv)\Di\vp-\Di(uv)\Dj\vp}_{\mathcal{H}^1(\Rn)}\\
    \lesssim \norm{\nabla \vp}_{L^{\infty}(\Rn)}\Big(\norm{u}_{L^2}\norm{\nabla v}_{L^2}+\norm{v}_{L^2}\norm{\nabla u}_{L^2}\Big),   
\end{multline*}
where the implicit constant depends only on dimension. 

\end{prop}
The proofs for Proposition \ref{HardyProp3} and \ref{HardyProp2} can be found in \cite{HLMPLp}.

\subsection{Hodge Decomposition}\label{Hodge subsec}

Recall that we write the matrix $A=A(x)$ as follows
$$
A=\begin{bmatrix}
\begin{BMAT}{c.c}{c.c}
A_{||} & \mathbf{b} \\
\mathbf{c} & d
\end{BMAT}
\end{bmatrix},
$$
where $A_{||}$ is the $n\times n$ submatrix of $A$, $\bd b$ is a $n\times 1$ vector, $\bd c$ is a $1\times n$ vector, $d$ is a scalar function. We consider the symmetric part $A^s$ and anti-symmetric part $A^a$ of $A$:
$$
A=A^s+A^a:=
\begin{bmatrix}
\begin{BMAT}{c.c}{c.c}
A^s_{||} & \mathbf{b}^s \\
\mathbf{c}^s & d
\end{BMAT}
\end{bmatrix}
+
\begin{bmatrix}
\begin{BMAT}{c.c}{c.c}
A^a_{||} & \mathbf{b}^a \\
\mathbf{c}^a & 0
\end{BMAT}
\end{bmatrix}.
$$
We assume that $A^s$ is $L^{\infty}$ and elliptic, with the ellipticity constant $\lambda_0$ and $\norm{A^s}_{\infty}\le\lambda_0^{-1}$, and that $A^a$ is in $\BMO(\Rn)$, with the $\BMO$ semi-norm 
\begin{equation*}
    \norm{a^a_{ij}}_{\BMO}:=\sup_{Q\subset\Rn}\fint_{Q}\abs{a^a_{ij}-(a^a_{ij})_{Q}}dx\le\Lambda_0.
\end{equation*}

\begin{prop}\label{Hodgedecp_prop}
For any cube $Q\subset\Rn$, there exist $\vp,\widetilde{\vp}\in W_0^{1,2}(5Q)$ that solve
\begin{align}
    &-\divg_x(A_{||}^*\nabla_x\vp)=\divg_x(\mathbf{c}\mathbbm{1}_{5Q}-(\mathbf{c}^a)_{2Q}),\label{cHodege}\\
    &\quad\divg_x(A_{||}\nabla_x\widetilde{\vp})=\divg_x(\bd b\mathbbm{1}_{5Q}-(\bd b^a)_{2Q}),\label{bHodge}
\end{align}
respectively. Moreover, there exists some $\epsilon_0=\epsilon_0(n,\lambda_0,\Lambda_0)>0$ and $C=C(n,\lambda_0,\Lambda_0)>0$ such that for all $p\in[2,2+\epsilon_0]$,
\begin{align}\label{nabla vp_Lp}
    \fint_{5Q}\abs{\nabla \vp(x)}^pdx\le C,\qquad
    \fint_{5Q}\abs{\nabla \widetilde{\vp}(x)}^pdx\le C.
\end{align}
\end{prop}

\begin{proof}
We only prove $\fint_{5Q}\abs{\nabla \widetilde{\vp}}^p\le C$, as the estimate for $\nabla \vp$ can be derived similarly. We will identify $\widetilde{\vp}$ with its zero extension outside of $5Q$.

Let $Q_{R_0}$ be a cube in $\Rn$ with $Q_{R_0}\cap 5Q\neq\emptyset$. For any $x\in Q_{R_0}$ and $0<R<\frac{1}{2}\dist(x,\bdy Q_{R_0})$, we have three possibilities:
\begin{enumerate}[(i)]
    \item $Q_{\frac{3}{2}R}(x)\cap 5Q=\emptyset$,
    \item $Q_{\frac{3}{2}R}(x)\cap(Q_{R_0}\setminus 5Q)=\emptyset$,
    \item $Q_{\frac{3}{2}R}(x)\cap 5Q\neq\emptyset$ and $Q_{\frac{3}{2}R}(x)\cap (Q_{R_0}\setminus 5Q)\neq\emptyset$.
\end{enumerate}

In case (ii), $Q_{\frac{3}{2}R}(x)\subset 5Q$, by the interior Caccioppoli inequality and Poincar\'e-Sobolev inequality, we have
\begin{align}
    \int_{Q_R(x)}\abs{\nabla \widetilde{\vp}}^2dy&\le CR^{-2}\int_{Q_{\frac{3}{2}R}(x)}\abs{\widetilde{\vp}-(\widetilde{\vp})_{Q_{3/2R}(x)}}^2dy+CR^n\label{int_Caccio}\\
    &\le C\br{\int_{Q_{\frac{3}{2}R}(x)}\abs{\nabla \widetilde{\vp}}^r}^{2/r}+CR^n,\nonumber
\end{align}
where $r=\frac{2n}{n+2}$.

In case (iii), we also have
\begin{equation*}
   \int_{Q_R(x)}\abs{\nabla \widetilde{\vp}}^2dy\le C\br{\int_{Q_{\frac{3}{2}R}(x)}\abs{\nabla \widetilde{\vp}}^r}^{2/r}+CR^n,
\end{equation*}
which follows from the boundary Caccioppoli inequality, 
\begin{equation}
    \int_{Q_R(x)}\abs{\nabla \widetilde{\vp}}^2dy\le CR^{-2}\int_{Q_{\frac{3}{2}R}(x)\cap 5Q}\abs{\widetilde{\vp}}^2dy+CR^n\label{bdy_Caccio},\\
\end{equation}
and a Sobolev-Poincar\'e theorem. The proof for \eqref{bdy_Caccio} is postponed until the end.

Now we can apply Lemma \ref{GiaProp1.1} to get
\[
\fint_{Q_{\frac{R_0}{2}}\cap 5Q}\abs{\nabla \widetilde{\vp}}^p\le C\br{\fint_{Q_{R_0}\cap 5Q}\abs{\nabla \widetilde{\vp}}^2}^{p/2}+C\fint_{Q_{R_0}}1.
\]
Choose $Q_{\frac{R_0}{2}}\supseteq 5Q$ then
\begin{equation*}
    \fint_{5Q}\abs{\nabla \widetilde{\vp}}^p\le C\br{\fint_{5Q}\abs{\nabla \widetilde{\vp}}^2}^{p/2}+C_n.
\end{equation*}
We claim that
\begin{equation}\label{nablav_L2}
  \fint_{5Q}\abs{\nabla \widetilde{\vp}}^2\le C(n,\lambda_0,\Lambda_0),
\end{equation}
which would imply the desired bound for $\fint_{5Q}\abs{\nabla \widetilde{\vp}}^p$.
In fact, taking $\widetilde{\vp}\in W_0^{1,2}(5Q)$ as a test function, equation \eqref{bHodge} and ellipticity of $A^s$ imply
\[
\lambda_0\int_{5Q}\abs{\nabla \widetilde{\vp}}^2\le\int_{5Q}A_{||}^s\nabla \widetilde{\vp}\cdot\nabla \widetilde{\vp}=\int_{5Q}A_{||}\nabla \widetilde{\vp}\cdot\nabla \widetilde{\vp}=\int_{5Q}\mathbf{b}\cdot\nabla \widetilde{\vp}.
\]
We have
\begin{multline*}
    \abs{\int_{5Q}\mathbf{b}\cdot\nabla \widetilde{\vp}}=\abs{\int_{5Q}\br{b_j^s+b_j^a-(b_j^a)_{5Q}}\Dj \widetilde{\vp}}\\
    \le \frac{\lambda_0}{2}\int_{5Q}\abs{\nabla \widetilde{\vp}}^2+C\int_{5Q}\abs{b_j^a-(b_j^a)_{5Q}}^2+C\int_{5Q}1.
\end{multline*}
Then \eqref{nablav_L2} follows from the John-Nirenberg inequality.

It remains to prove \eqref{int_Caccio} and \eqref{bdy_Caccio}.

\underline{Proof of \eqref{bdy_Caccio}}

For any $R\le t<s\le\frac{3}{2}R$, define $\xi\in C_0^2(Q_{\frac{t+s}{2}}(x))$ and $\eta\in C_0^2(Q_s(x))$ such that
$0\le\xi,\eta\le1$, $\xi=1$ in $Q_t(x)$, $\eta=1$ in $Q_{\frac{t+s}{2}}(x)$, and $\abs{\nabla\xi},\abs{\nabla\eta}\lesssim\frac{1}{s-t}$.

Choose $\widetilde{\vp}\xi^2\in W_0^{1,2}(5Q)$ as a test function, then \eqref{bHodge} gives
\begin{equation}\label{bdyCaccio_test}
\int_{Q_{\frac{t+s}{2}}(x)}A_{||}\nabla \widetilde{\vp}\cdot\nabla(\widetilde{\vp}\xi^2)=\int_{Q_{\frac{t+s}{2}}(x)}\mathbf{b}\mathbbm{1}_{5Q}\cdot\nabla(\widetilde{\vp}\xi^2).
\end{equation}
To estimate the left-hand side of \eqref{bdyCaccio_test}, we split the matrix into the symmetric and  anti-symmetric part. For the first one we have
\[
\int_{Q_{\frac{t+s}{2}}(x)}A^s_{||}\nabla \widetilde{\vp}\cdot\nabla(\widetilde{\vp}\xi^2)\ge \frac{\lambda_0}{2}\int_{Q_{\frac{t+s}{2}}(x)}\abs{\xi\nabla \widetilde{\vp}}^2-\frac{C}{(s-t)^2}\int_{Q_{\frac{t+s}{2}}(x)}\widetilde{\vp}^2.
\]
For the second one, we can write
\begin{multline*}
  \quad\int_{Q_{\frac{t+s}{2}}(x)}A_{||}^a\nabla \widetilde{\vp}\cdot\nabla(\widetilde{\vp}\xi^2)
  =\frac{1}{2}\int_{Q_{\frac{t+s}{2}}(x)}a_{ij}^a\br{\Dj \widetilde{\vp}\Di (\widetilde{\vp}\xi^2)-\Di \widetilde{\vp}\Dj(\widetilde{\vp}\xi^2)}\\
  =\frac{1}{4} \int_{Q_{\frac{t+s}{2}}(x)}a_{ij}^a\br{\Dj(\widetilde{\vp}^2)\Di(\xi^2)-\Di(\widetilde{\vp}^2)\Dj(\xi^2)}\\
  =\frac{1}{4} \int_{Q_{\frac{t+s}{2}}(x)}a_{ij}^a\br{\Dj(\widetilde{\vp}\eta)^2\Di(\xi^2)-\Di(\widetilde{\vp}\eta)^2\Dj(\xi^2)}.\\
\end{multline*}
By Proposition \ref{HardyProp2}, the absolute value of this quantity is bounded by
\[
\frac{C}{s-t}\norm{\widetilde{\vp}\eta}_{L^2}\norm{\nabla(\widetilde{\vp}\eta)}_{L^2}\le\frac{C_{\theta}}{(s-t)^2}\int_{Q_s(x)}\abs{\widetilde{\vp}}^2+\theta\int_{Q_s(x)}\abs{\nabla \widetilde{\vp}}^2
\]
for any $0<\theta<1$. 

As for the right-hand side of \eqref{bdyCaccio_test}, we have
\[
\abs{\int_{Q_{\frac{t+s}{2}}(x)}\mathbf{b}^s\mathbbm{1}_{5Q}\cdot\nabla(\widetilde{\vp}\xi^2)}\le\frac{\lambda_0}{8}\int_{Q_{\frac{t+s}{2}}(x)}\abs{\xi\nabla \widetilde{\vp}}^2+\frac{C}{(s-t)^2}\int_{Q_{\frac{t+s}{2}}(x)}\abs{\widetilde{\vp}}^2+Cs^n.
\]
Then, by Proposition \ref{HardyProp3},
\begin{multline*}
     \abs{\int_{Q_{\frac{t+s}{2}}(x)}\mathbf{b}^a\mathbbm{1}_{5Q}\cdot\nabla(\widetilde{\vp}\xi^2)}\\
  \le\frac{C}{s-t}\br{\int_{Q_{\frac{t+s}{2}}(x)}\abs{\widetilde{\vp}\xi}^2}^{1/2}s^{n/2}+\norm{\xi}_{L^2}\norm{\nabla(\widetilde{\vp}\xi)}_{L^2}\\
  \le\frac{\lambda_0}{8}\int_{Q_{\frac{t+s}{2}}(x)}\abs{\xi\nabla \widetilde{\vp}}^2+\frac{C}{(s-t)^2}\int_{Q_{\frac{t+s}{2}}(x)}\abs{\widetilde{\vp}}^2+Cs^n.
\end{multline*}
Combining these estimates with \eqref{bdyCaccio_test}, we fix $0<\theta<1$ to be sufficiently small and obtain
\begin{align*}
    \int_{Q_t(x)}\abs{\nabla \widetilde{\vp}}^2&\le\int_{Q_{\frac{t+s}{2}}(x)}\abs{\xi\nabla \widetilde{\vp}}^2
    \le \frac{C_{\theta}}{(s-t)^2}\int_{Q_s(x)}\abs{\widetilde{\vp}}^2+C\theta\int_{Q_s(x)}\abs{\nabla \widetilde{\vp}}^2+Cs^n\\
    &\le\frac{C_{\theta}}{(s-t)^2}\int_{Q_{\frac{3}{2}R}(x)}\abs{\widetilde{\vp}}^2+\frac{1}{2}\int_{Q_s(x)}\abs{\nabla \widetilde{\vp}}^2+CR^n.
\end{align*}
Then \eqref{bdy_Caccio} follows from Lemma \ref{GiaLemma3.1}.

The interior Caccioppoli \eqref{int_Caccio} can be shown in the same manner if one chooses $\br{\widetilde{\vp}-(\widetilde{\vp})_{Q_{\frac{3}{2}R}(x)}}\xi^2$ as a test function in the beginning.
\end{proof}
\begin{re}
Note that one can replace $(\mathbf{c}^a)_{2Q}$ and $(\mathbf{b}^a)_{2Q}$ in the right-hand side of \eqref{cHodege} and \eqref{bHodge}, respectively, by any constant vector $\mathbf{C}$ without changing the result. This follows from the simple fact that $\int_{5Q}\mathbf{C}\cdot\nabla v=0$ for any test function $v\in W^{1,2}_0(5Q)$.
\end{re}

Moser-type interior estimates for the weak solution to the homogeneous equation $-\divg_x A_{||}\nabla_x u=0$ have been shown in \cite{li2019boundary}, or \cite{seregin2012divergence} for the parabolic equations. We show that similar estimates hold for weak solutions to the nonhomogeneous equations.  

\begin{prop}\label{HodgeDecompMoser_prop}
Let $\vp$ and $\widetilde{\vp}$ be as in Proposition \ref{Hodgedecp_prop}. Let $B_{2R}=B_{2R}(x_0)\subset 5Q$. Then for any $p>1$,
\begin{equation}\label{HodgeDecompMoser}
    \sup_{B_R}\abs{\wvp-c_0}\le C\br{\fint_{B_{2R}}\abs{\wvp-c_0}^p}^{1/p}+CR(\norm{\mathbf{b}^s}_{L^{\infty}}+\norm{\mathbf{b}^a}_{\BMO}),
\end{equation}
where $c_0$ is any constant, and $C=C(n,\lambda_0,\Lambda_0,p)$. Moreover, a similar estimate holds for $\vp$: 
\begin{equation}\label{HodgeDecompMoser_vp}
    \sup_{B_R}\abs{\vp-c_0}\le C\br{\fint_{B_{2R}}\abs{\vp-c_0}^p}^{1/p}+CR(\norm{\mathbf{c}^s}_{L^{\infty}}+\norm{\mathbf{c}^a}_{\BMO}).
\end{equation}
\end{prop}

\begin{proof}
Fix any $p>1$ and $\frac{1}{2}<k_0<\frac{p}{2}$. Let $\frac1{2}<k_1<\min\{1,k_0\}$ and $k\ge k_0$. Let $\alpha=2$ when $n\ge 3$, and let $\alpha\in (1,2)$ when $n=2$. Choose $q\in (2,\frac{n\alpha}{n-\alpha})$. Set $s_0=\frac{2q}{q+2}$. Note that $1<s_0<\frac{n}{n-2}$ when $n\ge 3$ and $1<s_0<\alpha$ when $n=2$.

Define as in \cite{li2019boundary} Lemma 3.4, for any $\delta>0$, $N>>1$ and $\beta\ge k_0$, 
    $$
    H_{\delta,N}(t)=
    \left\{\begin{array}{l l}
    t^\beta,&\quad t\in[\delta,N],\\
    N^\beta+\frac{\beta}{k_1} N^{\beta-k_1}(t^{k_1}-N^{k_1}), &\quad t>N.
    \end{array}\right.
   $$
Then
   $$
    H'_{\delta,N}(t)=
    \left\{\begin{array}{rl}
   \beta t^{\beta-1},&\quad t\in (\delta,N),\\
   \beta N^{\beta-k_1}t^{k_1-1}, &\quad t>N.
   \end{array}\right.
   $$

  Define, furthermore, 
   $$
   G_{\delta,N}(w)=\int_\delta^w|H'_{\delta,N}(t)|^2dt,\ w\ge \delta.
   $$
Then for $w\ge \delta$, 
     \begin{equation}\label{H-G1}
       H(w)\le w^{\beta},
     \end{equation}
     \begin{equation} \label{H-G2}
       wH'(w)\le\beta w^{\beta},
     \end{equation}
      and
     \begin{equation}\label{H-G3}
      G(w)\le \frac1{2k_1-1}wG'(w).
     \end{equation}
    Here and in the sequel we omit the subscripts in $G_{\delta,N}$ and $H_{\delta,N}$.
    
    Let $\delta=R(\norm{\mathbf{b}^s}_{L^{\infty}}+\norm{\mathbf{b}^a}_{\BMO})$, and define $\Psi=\abs{\wvp-c_0}+\delta$, where $c_0$ is an arbitrary constant. Then $\Psi$ is a subsolution to the equation $\divg_x(A_{||}\nabla_x\widetilde{\vp})=\divg_x(\bd b\mathbbm{1}_{5Q}-(\bd b^a)_{2Q})$. Also, since $\Psi\ge\delta$, one can define $H(\Psi)$, $G(\Psi)$ etc.
    
    For any $R\le r'<r\le 2R$, let $\eta\in C_0^2(B_r)$ with $\eta=1$ in $B_{r'}$ and $\abs{\nabla\eta}\lesssim(r-r')^{-1}$. Choose $v=G(\Psi)\eta^2>0$ as a test function. Then since $\Psi$ is a subsolution, one has
    \begin{equation}\label{Psi_subsol}
     \int_{B_r}A_{||}\nabla\Psi\cdot\nabla v\le\int_{B_r}\mathbf{b}\cdot\nabla v.   
    \end{equation}
    
    For the left-hand side of \eqref{Psi_subsol}, we have (see the proof of Lemma 3.4 of \cite{li2019boundary})
   \[
        \int_{B_r}A_{||}^s\nabla\Psi\cdot\nabla v\ge\frac{\lambda_0}{2}\int_{B_r}\abs{\nabla H(\Psi)}^2\eta^2-\frac{C(n,\lambda_0,k_0)}{(2k_0-1)^2}\frac{\beta^2r^n}{(r-r')^2}\br{\fint_{B_r}\Psi^{\beta q}}^{2/q},\]
        and 
        \begin{multline*}
           \int_{B_r}A_{||}^a\nabla\Psi\cdot\nabla v\le\frac{\lambda_0}{8}\int_{B_r}\abs{\nabla H(\Psi)}^2\eta^2\\
       +\frac{C(n,\lambda_0,\Lambda_0,q,k_0)}{(2k_0-1)^2}\frac{\beta^2r^n}{(r-r')^2}\br{\fint_{B_r}\Psi^{\beta q}}^{2/q}. 
        \end{multline*}

   The right-hand side of \eqref{Psi_subsol} equals
   \begin{multline*}
        \int_{B_r}\mathbf{b}^s\cdot\nabla\br{G(\Psi)\eta^2}+\int_{B_r}\mathbf{b}^a\cdot\nabla\br{G(\Psi)\eta^2}\\
       =\int_{B_r}\mathbf{b}^s\cdot\nabla H(\Psi)\abs{H'(\Psi)}\eta^2+2\int_{B_r}\mathbf{b}^s\cdot\nabla\eta G(\Psi)\eta\\
       +\int_{B_r}\br{\mathbf{b}^a-(\mathbf{b}^a)_{B_{r}}}\cdot\nabla H(\Psi)\abs{H'(\Psi)}\eta^2
       +2\int_{B_r}\br{\mathbf{b}^a-(\mathbf{b}^a)_{B_{r}}}\cdot \nabla\eta G(\Psi)\eta\\
       =: I_1+I_2+I_3+I_4.
   \end{multline*}
 Using Cauchy-Schwartz inequality, \eqref{H-G2}, as well as Young's inequality, we obtain
   \begin{align*}
       \abs{I_1}\le \frac{\lambda_0}{8}\int\abs{\nabla H(\Psi)}^2\eta^2+C(n,\lambda_0)\norm{\mathbf{b}^s}_{L^{\infty}}^2\beta^2\int\Psi^{2\beta-2}\eta^2.
   \end{align*}
   Recall, in addition, that $\Psi\ge\delta=R(\norm{\mathbf{b}^s}_{L^{\infty}}+\norm{\mathbf{b}^a}_{\BMO})$ and $2<q<\frac{2n}{n-2}$. Then  $\abs{I_1}$ is bounded by
  \begin{multline}\label{vpMoser_I1}
     \frac{\lambda_0}{8}\int\abs{\nabla H(\Psi)}^2\eta^2+C(n,\lambda_0)\beta^2R^{-2}\int\Psi^{2\beta}\eta^2\\
    \le \frac{\lambda_0}{8}\int\abs{\nabla H(\Psi)}^2\eta^2+C(n,\lambda_0)\beta^2R^{-2}r^n\br{\fint_{B_r}\Psi^{q\beta}}^{2/q}.
  \end{multline}
   For $I_2$, we use \eqref{H-G3} and obtain
   \begin{align}
       \abs{I_2}&\le\frac{\norm{\mathbf{b}^s}_{L^{\infty}}}{r-r'}\frac{\beta^2}{2k_1-1}\int\Psi^{2\beta-1}\abs{\eta}
       \le \frac{C(n,k_0)\beta^2}{(r-r')(2k_0-1)R}\int_{B_r}\Psi^{2\beta}\nonumber\\
       &\le \frac{C(n,k_0)\beta^2r^n}{(r-r')(2k_0-1)R}\br{\fint_{B_r}\Psi^{q\beta}}^{2/q}.
   \end{align}
Turning to $I_3$, we estimate 
   \begin{align}
       \abs{I_3}&\le \br{\int_{B_r}\abs{\mathbf{b}^a-(\mathbf{b}^a)_{B_r}}^{s_0'}}^{\frac{1}{s_0'}}\br{\int\abs{\nabla H(\Psi)}^2\eta^2}^{1/2}\br{\int\abs{H'(\Psi)}^q\eta^q}^{1/q}\nonumber\\
       &\le C(n,q)\norm{\mathbf{b}^a}_{\BMO}r^{\frac{n}{s_0'}}\br{\int\abs{\nabla H(\Psi)}^2\eta^2}^{1/2}\beta\br{\int\Psi^{q\beta-q}\eta^q}^{1/q}\nonumber\\
       &\le \frac{\lambda_0}{8}\int\abs{\nabla H(\Psi)}^2\eta^2+\frac{C(n,\lambda_0,q)\beta^2r^n}{R^2}\br{\fint_{B_r}\Psi^{q\beta}}^{2/q},
   \end{align}
   where $s_0'=\frac{s_0}{s_0-1}$.
   Finally, for $I_4$, we have
   \begin{align}\label{vpMoser_I4}
       \abs{I_4}&\le\br{\int_{B_r}\abs{\mathbf{b}^a-(\mathbf{b}^a)_{B_r}}^{\frac{q}{q-2}}}^{\frac{q-2}{q}}\br{\int\abs{G(\Psi)}^{q/2}\abs{\nabla\eta}^{q/2}}^{2/q}\nonumber\\
       &\le C(n,q)\norm{\mathbf{b}^a}_{\BMO}r^{\frac{(q-2)n}{q}}\frac{1}{2k_1-1}\frac{\beta^2}{r-r'}\br{\int_{B_r}\Psi^{q\beta-\frac{q}{2}}}^{2/q}\nonumber\\
       &\le C(n,q,k_0)\frac{\beta^2r^n}{(2k_0-1)(r-r')R}\br{\fint_{B_r}\Psi^{q\beta}}^{2/q}.
   \end{align}
   
  Combining \eqref{Psi_subsol}--\eqref{vpMoser_I4}, we get
  \begin{equation}\label{gradHPsi}
      \frac{\lambda_0}{8}\fint_{B_{r'}}\abs{\nabla H(\Psi)}^2\le C\beta^2\br{(r-r')^{-2}+(r-r')^{-1}R^{-1}+R^{-2}}\br{\fint_{B_r}\Psi^{q\beta}}^{2/q}.
  \end{equation}
Furthermore, since $\alpha=2$ when $n\ge 3$ and $\alpha\in(1,2)$ when $n=2$, by
Sobolev embedding 
   \[
   \br{\fint_{B_{r'}}H(\Psi)^{\frac{n\alpha}{n-\alpha}}}^{\frac{n-\alpha}{n\alpha}}\lesssim \br{\fint_{B_{r'}}H(\Psi)^2}^{\frac1{2}}+r'\br{\fint_{B_{r'}}\abs{\nabla H(\Psi)}^2}^{\frac1{2}}.
   \]
   Now by \eqref{gradHPsi}, \eqref{H-G1}, and letting $N$ go to infinity, we obtain 
   \begin{multline*}
       \br{\fint_{B_{r'}}\Psi^{\beta\frac{n\alpha}{n-\alpha}}}^{\frac{n-\alpha}{n\alpha}}
      \le \br{\fint_{B_{r'}}\Psi^{2\beta}}^{1/2}\\
      + C\beta r'\br{(r-r')^{-2}+(r-r')^{-1}R^{-1}+R^{-2}}^{1/2}\br{\fint_{B_r}\Psi^{q\beta}}^{1/q}\\
      \le C\br{1+\beta \br{\frac{r'}{r-r'}+\frac{r'}{\sqrt{(r-r')R}}+\frac{r'}{R}}}\br{\fint_{B_r}\Psi^{q\beta}}^{1/q}. 
   \end{multline*}

Letting $l=\frac{n\alpha}{(n-\alpha)q}>1$, $\beta=\beta_i=kl^i$, $r=r_i=R+\frac{R}{2^i}$ and $r'=r_{i+1}$ for $i=0,1,2,\dots$, one finds
 \begin{align*}
     \br{\fint_{B_{r_{i+1}}}\Psi^{kl^{i+1}q}}^{\frac{1}{kl^{i+1}q}}&\le (Ckl^i)^{\frac{1}{kl^i}}\br{\fint_{B_{r_i}}\Psi^{kl^iq}}^{\frac{1}{kl^iq}}\le\dots\\
     &\le (Ck)^{\sum_{j=0}^i\frac{1}{kl^j}}l^{\sum_{j=0}^i\frac{j}{kl^j}}
     \br{\fint_{B_{2R}}\Psi^{kq}}^{\frac{1}{kq}}.
 \end{align*}
 Letting $i\to\infty$, we have $\sup_{B_R}\Psi\le C(n,\lambda_0,\Lambda_0,q,k_0)\br{\fint_{B_{2R}}\Psi^{kq}}^{\frac{1}{kq}}$, and thus
 \[
 \sup_{B_R}\abs{\wvp-c_0}\le C\br{\fint_{B_{2R}}\abs{\wvp-c_0}^{kq}}^{\frac{1}{kq}}+ C\delta,
 \]
 where $C=C(n,\lambda_0,\Lambda_0,q,k_0)$. Choosing $k$ and $q$ such that $kq=p$ yields \eqref{HodgeDecompMoser}. The proof of \eqref{HodgeDecompMoser_vp} is similar and thus omitted.
\end{proof}

\subsection{Weak Solutions of the Parabolic Equation}

We introduce $\mathcal{P}_t:= e^{-t^2L_{||}}$ and $\mP^*_t:= e^{-t^2L_{||}^*}$, the ``ellipticized" heat semigroup associated to $L_{||}=-\divg A_{||}\nabla$ and to its adjoint $L_{||}^*$, respectively. In this subsection, we shall derive Moser-type estimates for $\Dt\mP_{\eta t} f$ (and $\Dt\mP_{\eta t}^* f$), as well as reverse H\"older estimate for $\nabla_x\mP_{\eta t} f$ (and $\nabla_x\mP_{\eta t}^* f$).

{\bf Notation.}
In the rest of this section, since we only work with the $n$-dimensional operator $L_{||}$ and its adjoint $L_{||}^*$ instead of the operator $L$ defined in $\Rpl$, we shall simply write $L$ for $L_{||}$, and the same for its adjoint. For the same reason, we shall write $\nabla$ for $\nabla_x$, and $\divg$ for $\divg_x$.
We denote by $\wt{W}^{-1,2}$ the space of bounded semilinear functionals on $W^{1,2}(\Rn)$.

Let $u(x,t)=e^{-tL}(f)(x)$, for some $f\in L^2(\Rn)$. Then by Proposition \ref{wellpose_IVP}, $u(x,t)$ is the weak solution to the initial value problem
\begin{equation*}
   \begin{cases}
\Dt u-\divg(A\nabla u)=0 \quad\text{in }\Rn\times(0,\infty),\\
u(x,0)=f(x).
\end{cases}
\end{equation*}
That is, $u(x,t)\in L_{\loc}^2\br{(0,\infty),W^{1,2}(\Rn)}\cap C\br{[0,\infty),L^2}$, and satisfies 
\begin{align*}
    \int_{\Rn}u(x,T)\overline{\vp(x,T)}dx&+\int_0^T\int_{\Rn}A\nabla u\cdot\overline{\nabla\vp}dxdt\\
    &=\int_{\Rn}u(x,0)\overline{\vp(x,0)}dx+\int_0^T\overline{\langle \Dt\vp,u\rangle}_{\wt{W}^{-1,2},W^{1,2}}
\end{align*}
for any $T>0$, any $\vp\in L^2\br{[0,T],W^{1,2}(\Rn)}$ with $\Dt\vp\in L^2\br{[0,T],\wt{W}^{-1,2}(\Rn)}$.

Moreover, since $A$ depends only on $x\in\Rn$, $\Dt u$ is a weak solution to $\Dt v-\divg(A\nabla v)=0$ in $\Rn\times(0,\infty)$ (see the remark after Proposition \ref{wellpose_IVP}).  By \cite{HLMPLp}, Theorem 4.9 and its remark, $\Dt u\in L^2_{\loc}\br{(0,\infty),L^2(\Rn)}$ and $\Dt\nabla u\in L^2_{\loc}\br{(0,\infty),L^2(\Rn)}$. Finally, by the Gaussian estimate for the kernel of $\Dt e^{-tL}$ (see \cite{HLMPLp} Theorem 4.8), one can show that 
\begin{equation}\label{Dtu_Linfty}
    \Dt u\in L^{\infty}\br{[\delta_0,\infty)\times\Rn}\quad\forall\,\delta_0>0.
\end{equation} 

These facts enable us to prove the following estimate for $\Dt u$ using Moser iteration. 
\begin{prop}\label{sup_Dtu_Moser Prop}
Let $Q\subset\Rn$ be a cube with $l(Q)=R_0$. Then
\begin{equation}\label{sup_Dtu_Moser}
    \sup_{Q\times(R_0^2,(2R_0)^2]}\abs{\Dt u(x,t)}
    \le CR_0^{-\frac{n+2}{2}}\br{\int_{\frac{3}{2}Q}\int_{\frac{R_0^2}{2}}^{(2R_0)^2}\abs{\Dt u(x,t)}^2dtdx}^{1/2},
\end{equation}
for some $C=C(n,\lambda_0,\Lambda_0)$.
\end{prop}

\begin{proof}
Let $v(x,t)=\Dt u(x,t)$. Then by the definition of weak solution and Lemma \ref{Evanslem} (ii), we have
\[
\int_0^T\int_{\Rn}\Dt v(x,t)\vp(x,t)dxdt+\int_0^T\int_{\Rn}A\nabla v\cdot\nabla\vp=0,
\]
for all $\vp\in L^2\br{[0,T],W^{1,2}(\Rn)}$ with $\supp\vp\subset \Rn\times(0,T]$. By considering $v^{\pm}$ we can assume  $v\ge0$, and that
\begin{equation}\label{Dtu_subsol}
  \int_0^T\int_{\Rn}\Dt v(x,t)\vp(x,t)dxdt+\int_0^T\int_{\Rn}A\nabla v\cdot\nabla\vp\le0,  
\end{equation}
for all $\vp\in L^2\br{[0,T],W^{1,2}(\Rn)}$ with $\supp\vp\subset \Rn\times(0,T]$ and $\vp\ge0$ a.e.

Now for any $0\le s\le 1$, define 
$$Q_s=(1+s)Q, \quad I_s=((1-s)R_0^2,(2R_0)^2], \quad \mbox{and}\quad C_s=Q_s\times I_s.$$ Fix $l\in\mathbb N$, define $q_l=2k_0^l$, where $k_0=\frac{n+2}{n}$. Note that $q_0=2$. 
Furthermore, for any fixed 
$$\frac{4}{3}\frac{1}{2^{l+2}}\le s_0<s_1\le\frac{3}{2}\frac{1}{2^{l+2}},$$ choose $\Psi_{s_0,s_1}\in C_0^2(C_{\frac{s_0+s_1}{2}})$ and $\wt{\Psi}_{s_0,s_1}\in C_0^2(C_{s_1})$ such that $\Psi_{s_0,s_1}=1$ in $C_{s_0}$, $\wt{\Psi}_{s_0,s_1}=1$ in $C_{\frac{s_0+s_1}{2}}$, 
$0\le \Psi_{s_0,s_1},\wt{\Psi}_{s_0,s_1}\le 1$, and
\[
\abs{\nabla\Psi_{s_0,s_1}}^2+\abs{\Dt \Psi_{s_0,s_1}}+\abs{\nabla\wt{\Psi}_{s_0,s_1}}^2+\abs{\Dt \wt{\Psi}_{s_0,s_1}}\lesssim \frac{R_0^{-2}}{(s_1-s_0)^2}.
\]
We omit the subscript $s_0,s_1$ in $\Psi_{s_0,s_1}$ and $\wt{\Psi}_{s_0,s_1}$ from now on.

Let $t\in I_{s_0}$. Recalling \eqref{Dtu_Linfty}, one can take $\vp=v^{q_l-1}\Psi^2$ as a test function. Then \eqref{Dtu_subsol} gives 
\begin{equation}\label{Dtu_test0}
 \int_0^t\int_{\Rn}\Dt v v^{q_l-1}\Psi^2+\int_0^t\int_{\Rn}A\nabla v\cdot\nabla(v^{q_l-1}\Psi^2)\le0.   
\end{equation}
For the first term, integration by parts gives
\begin{align*}
   \int_0^t\int_{\Rn}\Dt v v^{q_l-1}\Psi^2&=\frac{1}{q_l}\int_{\Rn}v^{q_l}(x,t)\Psi^2(x,t)dx-\frac{1}{q_l}\int_0^t\int_{\Rn}v^{q_l}\Dt(\Psi^2)\\
   &\ge \frac{1}{q_l}\int_{Q_{s_0}}v^{q_l}(x,t)dx-\frac{CR_0^{-2}}{q_l(s_1-s_0)^2}\int_{C_{\frac{1}{2^{l+1}}}}v^{q_l}.
\end{align*}

The second term in \eqref{Dtu_test0} is split, as usual, corresponding to the symmetric and antisymmetric part of $A$. Working with $A^s$, we estimate
\begin{multline*}
  \quad\int_0^t\int_{\Rn}A^s\nabla v\cdot\nabla(v^{q_l-1}\Psi^2)\\
  =\frac{4(q_l-1)}{q_l^2}\int_0^t\int_{\Rn}A^s\nabla(v^{\frac{q_l}{2}})\cdot\nabla(v^{\frac{q_l}{2}})\Psi^2
  \\+\frac{4}{q_l}\int_0^t\int_{\Rn}A^s\nabla(v^{\frac{q_l}{2}})\cdot\nabla\Psi v^{\frac{q_l}{2}}\Psi\\
  \ge \frac{2\lambda_0(q_l-1)}{q_l^2}\int_0^t\int_{Q_{s_0}}\abs{\nabla v^{\frac{q_l}{2}}}^2dxdt-\frac{C(n,\lambda_0)R_0^{-2}}{q_l(s_1-s_0)^2}\int_{C_{\frac{1}{2^{l+1}}}}v^{q_l}.
\end{multline*}
Turning to $A^a$, note that $A^a\nabla v\cdot\nabla v \,\Psi^2=0$ due to anti-symmetry, so that
\begin{multline*}
     \int_0^t\int_{\Rn}A^a\nabla v\cdot\nabla(v^{q_l-1}\Psi^2)=\int_0^t\int_{\Rn}A^a\nabla v\cdot\nabla(\Psi^2)v^{q_l-1}\\
    =\frac{1}{q_l}\int_0^t\int_{\Rn}A^s\nabla(v^{q_l})\cdot\nabla(\Psi^2)
    =\frac{1}{q_l}\int_0^t\int_{\Rn}A^s\nabla(v^{\frac{q_l}{2}}\wt{\Psi}v^{\frac{q_l}{2}}\wt{\Psi})\cdot\nabla(\Psi^2).
\end{multline*}
By Proposition \ref{HardyProp2}, 
\begin{multline*}
    \abs{\int_0^t\int_{\Rn}A^a\nabla v\cdot\nabla(v^{q_l-1}\Psi^2)}\\
    \le \frac{C_n\Lambda_0}{q_l}\int_0^t\norm{\nabla\Psi}_{L^{\infty}(\Rn)}\norm{v^{\frac{q_l}{2}}\wt{\Psi}}_{L^2(\Rn)}\norm{\nabla\br{v^{\frac{q_l}{2}}\wt{\Psi}}}_{L^2(\Rn)}dt\\
    \le\frac{C_{\theta}\Lambda_0}{q_l}\frac{R_0^{-2}}{(s_1-s_0)^2}\int_{C_{\frac{1}{2^{l+1}}}}v^{q_l}dxdt+\frac{\theta}{q_l}\int_{C_{s_1}}\abs{\nabla v^{\frac{q_l}{2}}}^2.
\end{multline*}
 Combining these estimates with \eqref{Dtu_test0}, we have
 \begin{multline*}
    \int_{Q_{s_0}}v^{q_l}(x,t)dx+\int_0^t\int_{Q_{s_0}}\abs{\nabla v^{\frac{q_l}{2}}}^2\\
 \le\frac{CR_0^{-2}}{(s_1-s_0)^2}\int_{C_{\frac{1}{2^{l+1}}}}v^{q_l}dxdt+C\theta\int_{C_{s_1}}\abs{\nabla\br{v^{\frac{q_l}{2}}}}^2dxdt, 
\end{multline*}
 where $C=C(n,\lambda_0,\Lambda_0,\theta)$.
 
 Choosing $\theta$ to be sufficiently small, and then taking supremum in $t\in I_{s_0}$, we obtain 
 \begin{multline*}
     \sup_{t\in I_{s_0}}\int_{Q_{s_0}}v^{q_l}(x,t)dx+\int_{C_{s_0}}\abs{\nabla v^{\frac{q_l}{2}}}^2dxdt\\
 \le\frac{CR_0^{-2}}{(s_1-s_0)^2}\int_{C_{\frac{1}{2^{l+1}}}}v^{q_l}dxdt+\frac{1}{2}\int_{C_{s_1}}\abs{\nabla\br{v^{\frac{q_l}{2}}}}^2dxdt,
 \end{multline*}
  which implies
 \begin{multline}\label{sup_Dtu1}
     \sup_{t\in I_{\frac{4}{3}\frac{1}{2^{l+2}}}}\int_{Q_{\frac{4}{3}\frac{1}{2^{l+2}}}}v^{q_l}(x,t)dx+\int_{C_{\frac{4}{3}\frac{1}{2^{l+2}}}}\abs{\nabla v^{\frac{q_l}{2}}}^2dxdt\\
 \le C(n,\lambda_0,\Lambda_0)R_0^{-2}4^l\int_{C_{\frac{1}{2^{l+1}}}}v^{q_l}dxdt
 \end{multline}
 by Lemma \ref{GiaLemma3.1}.
 
Let us insert a cut-off function $\Psi_l(x,t)\in C_0^2(C_{\frac{4}{3}\frac{1}{2^{l+2}}})$ into \eqref{sup_Dtu1} so that we can use an embedding theorem. As usual, $\Psi_l$ satisfies $0\le\Psi_l\le1$, $\Psi_l=1$ in $C_{\frac{1}{2^{l+2}}}$, and
 \[
 \abs{\nabla\Psi_l}^2+\abs{\Dt\Psi_l}\lesssim R_0^{-2}4^l.
 \]
 Then we have
 \begin{multline*}
     \sup_{t\in I_{\frac{4}{3}\frac{1}{2^{l+2}}}}\int_{Q_{\frac{4}{3}\frac{1}{2^{l+2}}}}v^{q_l}(x,t)\Psi_l(x,t)dx
 +\int_{C_{\frac{4}{3}\frac{1}{2^{l+2}}}}\abs{\nabla (v^{\frac{q_l}{2}}\Psi_l)}^2dxdt\\
 \le CR_0^{-2}4^l\int_{C_{\frac{1}{2^{l+1}}}}v^{q_l}dxdt.
 \end{multline*}
Now, by a well-known embedding (see e.g. \cite{lieberman1996second} Theorem 6.9), we have
\begin{multline*}
    \int_{C_{\frac{1}{2^{l+2}}}}v^{q_lk_0}\le\int_{C_{\frac{4}{3}\frac{1}{2^{l+2}}}}(v^{\frac{q_l}{2}}\Psi_l)^{2k_0}\\
    \le\sup_{t\in I_{\frac{4}{3}\frac{1}{2^{l+2}}}}\br{\int_{Q_{\frac{4}{3}\frac{1}{2^{l+2}}}}(v^{\frac{q_l}{2}}\Psi_l)^2(x,t)dx}^{2/n}\int_{C_{\frac{4}{3}\frac{1}{2^{l+2}}}}\abs{\nabla(v^{\frac{q_l}{2}}\Psi_l)}^2dxdt\\
    \le\br{\sup_{t\in I_{\frac{4}{3}\frac{1}{2^{l+2}}}}\int_{Q_{\frac{4}{3}\frac{1}{2^{l+2}}}}(v^{\frac{q_l}{2}}\Psi_l)^2(x,t)dx
    +\int_{C_{\frac{4}{3}\frac{1}{2^{l+2}}}}\abs{\nabla (v^{\frac{q_l}{2}}\Psi_l)}^2dxdt}^{k_0}\\
  \le C\br{R_0^{-2}4^l}^{k_0}\br{\int_{C_{\frac{1}{2^{l+1}}}}v^{q_l}dxdt}^{k_0}.
\end{multline*}
Therefore, for all $l\in\mathbb N$,
\[
\br{\int_{C_{\frac{1}{2^{l+2}}}}v^{q_{l+1}}dxdt}^{\frac{1}{q_{l+1}}}\le C^{\frac{1}{q_{l+1}}}(R_0^{-2}4^l)^{\frac{1}{q_l}}\br{\int_{C_{\frac{1}{2^{l+1}}}}v^{q_l}dxdt}^{\frac{1}{q_l}}.
\]
Then \eqref{sup_Dtu_Moser} follows from iteration and letting $l$ go to infinity.
\end{proof}

\begin{prop}\label{nablau_RH prop}
Let $Q\subset\Rn$ be a cube with $l(Q)=R_0$. Then for any $t>0$,
\begin{multline}\label{nablau_RH}
   \br{\fint_{Q}\abs{\nabla u(x,t)}^pdx}^{1/p}\\
   \le C\br{\fint_{2Q}\abs{\nabla u(x,t)}^2dx}^{1/2}+R_0\br{\fint_{2Q}\abs{\Dt u(x,t)}^pdx}^{1/p} 
\end{multline}
for all $p\in[2,2+\epsilon)$, where $C=C(n,\lambda_0,\Lambda_0)$ and $\epsilon=\epsilon(n,\lambda_0,\Lambda_0)$ are positive constants.
\end{prop}

\begin{proof}
Let $x_0\in 4Q$ and $0<R<\min\set{\frac{1}{2}\dist(x_0,4Q),2R_0}$. Choose two cut-off functions. First,  $\Psi\in C_0^1(Q_{\frac{3}{2}R}(x_0))$, with $\Psi=1$ on $Q_R(x_0)$ and $\abs{\nabla \Psi}\lesssim R^{-1}$, and secondly, $\wt\Psi\in C_0^1(Q_{2R}(x_0))$, with $\Psi=1$ on $Q_{\frac{3}{2}R}(x_0)$ and $\abs{\nabla \wt{\Psi}}\lesssim R^{-1}$.

Fix $t>0$ and define $\bar{u}=\fint_{Q_{2R}(x_0)}u(x,t)dx$. Take $(u(x,t)-\bar{u})\Psi^2(x)$ as a test function. Then $\Dt u-\divg(A\nabla u)=0$ implies that
\begin{multline}\label{nablau_RH_wksol}
  \int_{\Rn}A\nabla u(x,t)\cdot\nabla\br{(u(x,t)-\bar{u})\Psi^2(x)}dx\\ =-\int_{\Rn}\Dt u(x,t) (u(x,t)-\bar{u})\Psi^2(x)dx.
\end{multline}
For the integral involving the symmetric part of $A$, we have
\begin{multline*}
    \int_{\Rn}A^s\nabla u(x,t)\cdot\nabla\br{(u(x,t)-\bar{u})\Psi^2(x)}dx\\
   \ge\frac{\lambda_0}{2}\int_{Q_R(x_0)}\abs{\nabla u}^2dx-\frac{C(n,\lambda_0)}{R^2}\int_{Q_{\frac{3}{2}}R(x_0)}(u-\bar{u})^2.  
\end{multline*}
For the integral involving the anti-symmetric part of $A$, we insert $\wt\Psi$ and apply Proposition \ref{HardyProp2}:
\begin{multline*}
     \abs{\int_{\Rn}A^a\nabla u(x,t)\cdot\nabla\br{(u(x,t)-\bar{u})\Psi^2(x)}dx}\\
    =\abs{\int_{\Rn}A^a\nabla(u-\bar{u})^2\cdot\nabla(\Psi^2)}
    =\abs{\int_{\Rn}A^a\nabla\br{(u-\bar{u})^2\wt{\Psi}^2}\cdot\nabla(\Psi^2)}\\
    \le\frac{C_n\Lambda_0}{R}\norm{(u-\bar{u})\wt{\Psi}}_{L^2(\Rn)}\norm{\nabla\br{(u-\bar{u})\wt{\Psi}}}_{L^2(\Rn)}\\
    \le C_n\theta\int_{Q_{2R}(x_0)}\abs{\nabla u}^2dx+\frac{C(n,\Lambda_0,\theta)}{R^2}\int_{Q_{2R}(x_0)}(u-\bar{u})^2dx.
\end{multline*}
Finally, we estimate the right-hand side of \eqref{nablau_RH_wksol} by Cauchy-Schwartz:
\begin{multline*}
   \abs{\int_{\Rn}\Dt u(x,t)(u(x,t)-\bar{u})\Psi^2(x)dx}\\
   \le C_n R^2\int_{Q_{\frac{3}{2}R}(x_0)}\abs{\Dt u}^2dx+\frac{C_n}{R^2}\int_{Q_{\frac{3}{2}}R(x_0)}(u(x,t)-\bar{u})^2dx. 
\end{multline*}

To summarize, 
\begin{multline*}
   \int_{Q_R(x_0)}\abs{\nabla u}^2dx
   \lesssim R^{-2}\int_{Q_{2R}(x_0)}(u-\bar{u})^2dx\\
+R^2\int_{Q_{\frac{3}{2}R}(x_0)}\abs{\Dt u}^2dx+\theta\int_{Q_{2R}(x_0)}\abs{\nabla u}^2dx. 
\end{multline*}
Choosing $\theta$ to be sufficiently small and using Sobolev inequality, we obtain
\begin{multline*}
   \fint_{Q_R(x_0)}\abs{\nabla u}^2dx
   \le C\br{\fint_{Q_{2R}(x_0)}\abs{\nabla u}^{\frac{2n}{n+2}}dx}^{\frac{n+2}{n}}
\\+CR_0^2\fint_{Q_{2R}(x_0)}\abs{\Dt u}^2dx+\frac{1}{2}\fint_{Q_{2R}(x_0)}\abs{\nabla u}^2dx.
\end{multline*}
Then \eqref{nablau_RH} follows from Lemma \ref{GiaProp1.1}. 
\end{proof}

Let $w(x,t)=\mathcal{P}_{\eta t}f(x)=e^{-{\eta t}^2L_{||}}(f)(x)$ for some $\eta>0$. Then $\Dt w(x,t)=2\eta^2t\partial_{\tau}u(x,(\eta t)^2)$. Using this relationship one easily gets 
\begin{cor}\label{sup dtw cor}
Let $k\in\mathbb Z$, and $Q\subset\Rn$ be a cube with $l(Q)\approx 2^{-k}\eta$. Then 
\begin{multline*}
    \sup_{Q\times(2^{-k},2^{-k+1}]}\abs{\Dt w(x,t)}\\
    \le C(2^{-k}\eta^2)^{1/2}(2^{-k}\eta)^{-\frac{n+2}{2}}\br{\int_{\frac{3}{2}Q}\int_{2^{-k-\frac{1}{2}}}^{2^{-k+1}}\abs{\Dt w(x,t)}^2dtdx}^{1/2},
\end{multline*}
for some $C=C(n,\lambda_0,\Lambda_0)$. Equivalently,
\[
\sup_{Q\times(2^{-k},2^{-k+1}]}\abs{\Dt \mathcal{P}_{\eta t}f(x)}^2
\le C(n,\lambda_0,\Lambda_0)\frac{\eta}{\abs{Q}}\int_{\frac{3}{2}Q}\int_{2^{-k-\frac{1}{2}}}^{2^{-k+1}}\abs{\Dt \mathcal{P}_{\eta t}f(x)}^2\frac{dt}{t}dx, 
\]
for all $f\in L^2(\Rn)$. The estimate also holds for $\Dt \mathcal{P}^*_{\eta t}f(x)$. 

\end{cor}

\begin{cor}\label{nablaw RHp cor}
Let $Q\subset\Rn$ be a cube with $l(Q)\approx 2^{-k}\eta$. Then 
\begin{equation}
   \br{\fint_{Q}\abs{\nabla w(x,t)}^p}^{1/p}\le C\br{\fint_{2Q}\abs{\nabla w(x,t)}^2}^{1/2}+\eta^{-1}\br{\fint_{2Q}\abs{\Dt w(x,t)}^p}^{1/p} 
\end{equation}
for any $t\in(2^{-k},2^{-k+1})$, $p\in[2,2+\epsilon)$. Here, $C=C(n,\lambda_0,\Lambda_0)$ and $\epsilon=\epsilon(n,\lambda_0,\Lambda_0)$ are positive constants.
\end{cor}

\subsection{$L^p$ estimates for square functions}\label{Lp est for square functions subsect}
The following results are obtained in \cite{HLMPLp} and we include them here for reader's convenience. The operator $L$ should be  thought of as the operator $L_{||}$ or $L_{||}^*$ in our setting.

\begin{prop}[\cite{HLMPLp} Proposition 6.2]\label{Lp_G1Prop}
For all $1<p<\infty$, and\\
$F\in W^{1,2}(\Rn)\cap W^{1,p}(\Rn)$,
\begin{equation}\label{Lp_G1}
\norm{\Big(\int_0^{\infty}\abs{tL e^{-t^2L}F}^2\frac{dt}{t}\Big)^{1/2}}_{L^p(\Rn)}\le C_p\norm{\nabla F}_{L^p(\Rn)}.
\end{equation}
Or equivalently, 
\begin{equation}\label{Lp_G1Dt}
\norm{\Big(\int_0^{\infty}\abs{\Dt e^{-t^2L}F}^2\frac{dt}{t}\Big)^{1/2}}_{L^p(\Rn)}\le C_p\norm{\nabla F}_{L^p(\Rn)}.
\end{equation}
\end{prop}

\begin{prop}[\cite{HLMPLp} Proposition 6.3]\label{Lp_G2xProp}
For $1<p\le2+\epsilon_0$, with $\epsilon_0=\epsilon_0(\lambda_0,\Lambda_0,n)>0$, and for all $F\in W^{1,2}(\Rn)\cap W^{1,p}(\Rn)$,
\begin{equation}\label{Lp_G2x}
\norm{\Big(\int_0^{\infty}\abs{t^2\nabla L e^{-t^2L}F}^2\frac{dt}{t}\Big)^{1/2}}_{L^p(\Rn)}\le C_p\norm{\nabla F}_{L^p(\Rn)}.
\end{equation}
 Or equivalently,
\begin{equation*}
\norm{\Big(\int_0^{\infty}\abs{t\nabla \Dt e^{-t^2L}F}^2\frac{dt}{t}\Big)^{1/2}}_{L^p(\Rn)}\le C_p\norm{\nabla F}_{L^p(\Rn)}.
\end{equation*}
\end{prop}

\begin{re}
The upper bound $2+\epsilon_0$ for the range of admissible of $p$ might be different from the $2+\epsilon_1$ in \cite{HLMPLp}, Proposition 6.3. For convenience, we set the minimum between $\epsilon_1$ and the $\epsilon_0$ from Proposition \ref{Hodgedecp_prop} to be $\epsilon_0$ and fix the notation from now on.
\end{re}

\begin{prop}[\cite{HLMPLp} Proposition 6.4]\label{Lp_G2t prop}
For all $1<p<\infty$, and all $F\in W^{1,2}(\Rn)\cap W^{1,p}(\Rn)$,
\begin{equation}\label{Lp_G2t}
\norm{\Big(\int_0^{\infty}\abs{t^2\Dt L e^{-t^2L}F}^2\frac{dt}{t}\Big)^{1/2}}_{L^p(\Rn)}\le C_p\norm{\nabla F}_{L^p(\Rn)}.
\end{equation}
\end{prop}

\subsection{$L^p$ estimates for non-tangential maximal functions}\label{Lp nt max subsec}
\begin{defn}
The non-tangential maximal function is defined as
\begin{equation}
    N^{\alpha}(u)(x):= \sup_{t>0}\sup_{(y,t):\abs{x-y}<\alpha t}\abs{u(y,t)}.
\end{equation}
The integrated non-tangential maximal function is defined as
\begin{equation}\label{integratedNT defn}
    \wt{N}^{\alpha}(u)(x):=\sup_{t>0}\sup_{(y,t):\abs{x-y}<\alpha t}\br{\fint_{\abs{y-z}<\alpha t}\abs{u(z,t)}^2dz}^{1/2}.
\end{equation}
\end{defn}

Again, we shall simply write $L$ for the $n$-dimensional operator $L_{||}$ in this section. We consider functions such as $N^{\alpha}(\Dt e^{-t^2L}f)$, where we think of\\
$\Dt e^{-t^2L}f(x)$ as a function of $x$ and $t$.

\begin{prop}\label{NTalpha_Lp Prop}  Let $\eta>0$, $\alpha>0$. Then
\begin{equation*}
   \norm{\eta^{-1}N^{\eta\alpha}(\Dt e^{-(\eta t)^2L}f)}_{L^p}\le C_{\alpha,p}\norm{\nabla f}_{L^p} 
\end{equation*}
for all $p>1$, and $f\in W^{1,p}$. The constant $C_{\alpha,p}$ also depends on $\lambda_0$, $\Lambda_0$ and $n$, but not on $\eta$.
\end{prop}

\begin{proof}
 Fix any $x\in\Rn$, and let $(y,t)\in \Gamma_{\eta\alpha}(x)$ so that  $\abs{x-y}<\eta\alpha t$. We claim that for every  $f\in\mathscr{S}(\Rn)$ 
 \begin{equation}\label{DtPt_ptbd_M}
     \abs{\eta^{-1}\Dt e^{-(\eta t)^2L}f(y)}\le C_{\alpha}M(\nabla f)(x).
 \end{equation}
 Let $V_t(x,y)$ be the kernel associated to $\Dt e^{-t^2L}$. Then by \cite{HLMPLp} Theorem 4.8, we have 
\[\abs{V_t(x,y)}\lesssim t^{-n-1}e^{-\frac{\abs{x-y}^2}{ct^2}},\qquad \abs{\eta^{-1}V_{\eta t}(x,y)}\lesssim (\eta t)^{-n-1}e^{-\frac{\abs{x-y}^2}{c(\eta t)^2}}\]
where the implicit constant depends on $\lambda_0$, $\Lambda_0$ and $n$.
We write
\begin{align*}
   \eta^{-1}\Dt e^{-(\eta t)^2L}f(y)&= \eta^{-1}\Dt e^{-(\eta t)^2L}\Bigl(f-\fint_{B_{2\eta
    \alpha t}(x)}f\Bigr)(y)\\
    &=\ \int_{\Rn}\eta^{-1}V_{\eta t}(y,z)\Bigl(f-\fint_{B_{2\eta\alpha t}(x)}f\Bigr)(z)dz.
\end{align*}
Then the estimate for the kernel entails the bound
\begin{multline*}
     \abs{\eta^{-1}\Dt e^{-(\eta t)^2L}f(y)}\lesssim \int_{B_{2\eta \alpha t}(x)}\frac{1}{(\eta t)^{n+1}}e^{-\frac{\abs{y-z}^2}{c(\eta t)^2}}\abs{f(z)-\fint_{B_{2\eta\alpha t}(x)}f}dz\\
     +\sum_{k=1}^{\infty}\int_{2^{k+1}B_{\eta\alpha t}(x)\setminus 2^k B_{\eta\alpha t}(x)}\frac{1}{(\eta t)^{n+1}}e^{-\frac{\abs{y-z}^2}{c(\eta t)^2}}\abs{f(z)-\fint_{B_{2\eta\alpha t}(x)}f}dz\\
    =: I_1+I_2.
\end{multline*}
  
  For $I_1$, we trivially bound $e^{-\frac{\abs{y-z}^2}{c(\eta t)^2}}$ by 1, and then the Poincar\'e inequality gives
\begin{multline*}
    I_1\lesssim\frac{\alpha^n}{\eta t}\fint_{B_{2\eta\alpha t}(x)}\abs{f(z)-\fint_{B_{2\eta\alpha t}(x)}f}dz\lesssim\alpha^{n+1}\fint_{B_{2\eta\alpha t}(x)}\abs{\nabla f}\\
    \lesssim \alpha^{n+1}M(\nabla f)(x),
\end{multline*}
where the implicit constants depend only on $n$.

For $I_2$, we have
\begin{align*}
    I_2&\lesssim\sum_{k=1}^{\infty}\frac{1}{(\eta t)^{n+1}}\exp\set{-\frac{(2^k-1)^2\alpha^2}{c}}\int_{2^{k+1}B_{\eta\alpha t}(x)}\abs{f(z)-\fint_{B_{2\eta\alpha t}(x)}f}dz\\
    &\lesssim\sum_{k=1}^{\infty}\frac{2^{n(k+1)}\alpha^n}{\eta t}\exp\set{-\frac{(2^k-1)^2\alpha^2}{c}}\fint_{2^{k+1}B_{\eta\alpha t}(x)}\abs{f(z)-\fint_{B_{2\eta\alpha t}(x)}f}dz.
\end{align*}
Breaking the integrand into sum of terms containing $\fint_{2^{l+1}B_{\eta\alpha t}(x)}f-\fint_{2^{l}B_{\eta\alpha t}(x)}f$ and using the Poincar\'e inequality again, we obtain
\[
I_2\lesssim\sum_{k=1}^{\infty}\exp\set{-\frac{4k^2\alpha^2}{c}}2^{n(k+1)}\alpha^{n+1}\sum_{l=2}^{k+1}2^l M(\nabla f)(x)\lesssim_{\alpha}M(\nabla f)(x),
\]
and thus \eqref{DtPt_ptbd_M} follows. By the choice of $(y,t)$, this implies  
\[
\eta^{-1}N^{\eta\alpha}(\Dt e^{-(\eta t)^2L}f)(x)\le C_{\alpha}M(\nabla f)(x),
\]
so that
\[
\norm{\eta^{-1}N^{\eta\alpha}(\Dt e^{-(\eta t)^2L}f)}_{L^p}\le C_{\alpha,p}\norm{\nabla f}_{L^p}\qquad\forall\, p>1,\  f\in\mathscr{S}(\Rn).
\]
Then the proposition follows from a standard limiting argument. 
\end{proof}

We also have $L^p$ estimates for the integrated non-tangential maximal function:
\begin{prop}\label{integratedNT_Lp prop}
Let $\eta>0$. Then for any $p>2$, $f\in W^{1,p}(\Rn)$,
\[
\norm{\wt{N}^{\eta}(\nabla e^{-(\eta t)^2L}f)}_{L^p}\le C_p\norm{\nabla f}_{L^p},
\]
where the constant depends on $p$, $\lambda_0$, $\Lambda_0$ and $n$, but not on $\eta$.
\end{prop}

\begin{proof}
Let $f\in\mathscr{S}(\Rn)$. Define $u(x,t)=e^{-tL}f(x)$. Then $u$ satisfies the equation $\Dt u-\divg(A\nabla u)=0$ in $L^2$. Now fix $x\in\Rn$, and fix $(y,t)\in \Gamma_{\eta}(x)$. Define $B_s=B(y,(1+s)\eta t)$, the ball centered at $y$ with radius $(1+s)\eta t$.

For $0\le s<s'<\frac{1}{2}$, choose
\[
\Psi\in C_0^{\infty}(B_{\frac{s+s'}{2}}), \quad\text{with }\quad\Psi=1\text{ on }B_s, \quad\abs{\nabla \Psi}\lesssim\frac{1}{(s'-s)\eta t},
\]
and
\[
\wt{\Psi}\in C_0^{\infty}(B_{s'}), \quad\text{with }\quad\wt{\Psi}=1\text{ on }B_{\frac{s+s'}{2}}, \quad\abs{\nabla \wt{\Psi}}\lesssim\frac{1}{(s'-s)\eta t}
\]
Let $\bar u=\fint_{B(y,\frac{3}{2}\eta t)}u(x,0)dx$. Taking $(u-\bar{u})\Psi^2$ as a test function, we obtain
\[
\int_{\Rn}A(x)\nabla u(x,\tau)\cdot\nabla\br{(u(x,\tau)-\bar{u})\Psi^2}dx=-\int_{\Rn}\partial_{\tau}u(x,\tau)(u(x,\tau)-\bar{u})\Psi^2dx
\]
for any $\tau>0$.
Then by an argument similar to the proof of Proposition \ref{nablau_RH prop}, one can write
\begin{multline*}
    \int_{\Rn}A(x)\nabla u(x,\tau)\cdot\nabla\br{(u(x,\tau)-\bar{u})\Psi^2}dx\\
\le\frac{\lambda_0}{2}\int_{B_s}\abs{\nabla u(x,\tau)}^2dx-\frac{C_{\theta}}{(s'-s)^2(\eta t)^2}\int_{B_{s'}}\abs{u(x,\tau)-\bar{u}}^2dx\\
-\theta\int_{B_{s'}}\abs{\nabla u(x,\tau)}^2dx,  
\end{multline*}
and
\begin{align*}
 &\abs{\int_{\Rn}\partial_{\tau}u(x,\tau)(u(x,\tau)-\bar{u})}\\
 &\lesssim (s'-s)^2(\eta t)^2\int_{B_{\frac{s+s'}{2}}}\abs{\partial_{\tau}u(x,\tau)}^2dx+\frac{1}{(s'-s)^2(\eta t)^2}\int_{B_{\frac{s+s'}{2}}}\abs{u(x,\tau)-\bar{u}}^2. 
\end{align*}
Combining, we have
\begin{multline*}
    \int_{B_s}\abs{\nabla u(x,\tau)}^2dx
    \le\frac{C}{(s'-s)^2(\eta t)^2}\int_{B_{s'}}(u(x,\tau)-\bar{u})^2dx\\
    +C(\eta t)^2\int_{B_{s'}}\abs{\partial_{\tau}u(x,\tau)}^2dx+C\theta\int_{B_{s'}}\abs{\nabla u(x,\tau)}^2dx.
\end{multline*}
Choosing $\theta$ sufficiently small, then Lemma \ref{GiaLemma3.1} gives
\begin{multline}\label{integratedNTest_nablau}
     \fint_{B(y,\eta t)}\abs{\nabla u(x,\tau)}^2dx\\
     \lesssim\frac{1}{(\eta t)^2}\fint_{B(y,\frac{3}{2}\eta t)}\abs{u(x,\tau)-\bar{u}}^2dx+(\eta t)^2\fint_{B(y,\frac{3}{2}\eta t)}\abs{\partial_{\tau}u(x,\tau)}^2dx, 
\end{multline}
for any $\tau>0$.

Let $w(z,t)=u(z,\eta^2t^2)$. Then it suffices to show \[\norm{\wt{N}^{\eta}(\nabla_x w)}_{L^p}\le C_p\norm{\nabla f}_{L^p} \qquad\forall\, p>2.\]
To this end, let $\tau=\eta^2t^2$ in \eqref{integratedNTest_nablau}. Noticing that $\Dt w(z,t)=2\eta^2t\Dt u(z,\eta^2t^2)$, we have
\begin{multline}\label{integratedNT_nablaw}
    \fint_{B(y,\eta t)}\abs{\nabla w(z,t)}^2dz\\
    \lesssim\frac{1}{(\eta t)^2}\fint_{B(y,\frac{3}{2}\eta t)}\abs{w(z,t)-\bar{w}}^2dz+\eta^{-2}\fint_{B(y,\frac{3}{2}\eta t)}\abs{\Dt w(z,t)}^2dz
\end{multline}
where $\bar{w}=\fint_{B(y,\frac{3}{2}\eta t)}w(z,0)dz$.

The expression $\fint_{B(y,\frac{3}{2}\eta t)}\abs{\Dt w(z,t)}^2dz$ on the right-hand side of \eqref{integratedNT_nablaw} can be controlled by
$M\br{N^{\eta}(\Dt w)^2}(x)$. To estimate the first term on the right-hand side of \eqref{integratedNT_nablaw}, we write
\[
\abs{w(z,t)-\bar{w}}\le\abs{w(z,t)-w(z,0)}+\abs{w(z,0)-\bar{w}}.
\]
Using Poinca\'re inequality, we have
\begin{align*}
    \frac{1}{(\eta t)^2}\fint_{B(y,\frac{3}{2}\eta t)}\abs{w(z,0)-\bar{w}}^2dz&\lesssim\fint_{B(y,\frac{3}{2}\eta t)}\abs{\nabla w(z,0)}^2dz\\
    &\lesssim M\br{\abs{\nabla w(\cdot,0)}^2}(x)=M\br{\abs{\nabla f}^2}(x).
\end{align*}
Since 
\[
\abs{w(z,t)-w(z,0)}\le\int_0^t\abs{\partial_{\tau}w(z,\tau)}d\tau\le t\sup_{0\le\tau\le t}\abs{\partial_{\tau}w(z,\tau)},
\]
we conclude that
\begin{multline*}
    \fint_{B(y,\frac{3}{2}\eta t)}\abs{w(z,t)-w(z,0)}^2dz
    \le t^2\fint_{B(y,\frac{3}{2}\eta t)}\sup_{0\le\tau\le t}\abs{\partial_{\tau}w(z,t)^2}dz\\
    \le t^2M(N^{\eta}(\Dt w)^2)(x).
\end{multline*}
Therefore, we have obtained the estimate
\[
\fint_{B(y,\eta t)}\abs{\nabla w(z,t)}^2dz\lesssim M\br{\abs{\nabla f}^2}(x)+\eta^{-2} M\br{N^{\eta}(\Dt w)^2}(x)
\]
for any $(y,t)\in\Gamma_{\eta}(x)$. This implies, by the definition of the integrated non-tangential maximal function $\wt{N}^{\eta}$, that
\[
\wt{N}^{\eta}(\nabla_z w)(x)\lesssim\br{M\br{\abs{\nabla f}^2}(x)}^{1/2}+\eta^{-1}\br{M\br{N^{\eta}(\Dt w)^2}(x)}^{1/2},\quad\forall\, x\in\Rn. 
\]
Thus for any $p>2$,
\begin{multline*}
    \norm{\wt{N}^{\eta}\br{\nabla e^{-\eta^2t^2L}f}}_{L^p(\Rn)}\lesssim\norm{\nabla f}_{L^p}+\norm{\eta^{-1}N^{\eta}\br{\Dt e^{-\eta^2t^2L}f}}_{L^p(\Rn)}\\
    \le C_p\norm{\nabla f}_{L^p(\Rn)},
\end{multline*}
where we have used Proposition \ref{NTalpha_Lp Prop} with $\alpha=1$. 
\end{proof}

\section{Construction of F and sawtooth domains associated with F}\label{F section}
\subsection{The set F}\label{set F subsection}
We define the following maximal differential operator
\begin{equation}
    D_pf(x):=\sup_{r>0}\br{\fint_{\abs{x-y}<r}\br{\frac{\abs{f(x)-f(y)}}{\abs{x-y}}}^pdy}^{1/p}.
\end{equation}

\begin{lem}\label{DLp_lem}
\[\norm{D_{p_1}f}_{L^p(\Rn)}\le C_{p,p_1,n}\norm{\nabla f}_{L^p} \qquad\forall\ 1\le p_1<p<\infty. \]
\end{lem}

This lemma follows from a Morrey type inequality 
\begin{equation*}
    \frac{\abs{f(x)-f(y)}}{\abs{x-y}}\le M(\nabla f)(x)+M(\nabla f)(y) \qquad\forall\, x,y\in\Rn,
\end{equation*}
and the $L^p$ bound for the Hardy-Littlewood maximal function.

We introduce a few notations. Recall that we use $\mathcal{P}_t$ to denote $e^{-t^2L_{||}}$, and $\mP^*_t=e^{-t^2L_{||}^*}$. Define
\[
\Lambda_1:=\eta^{-1}N^{\eta}(\Dt\mathcal{P}^*_{\eta t}\vp)+N(\Dt\mathcal{P}^*_{t}\vp)+\wt{N}^{\eta}(\nabla_x\mathcal{P}^*_{\eta t}\vp)+\br{M(\abs{\nabla\vp}^2)}^{1/2},
\]
\[
\Lambda_2:=\eta^{-1}N^{\eta}(\Dt\mathcal{P}_{\eta t}\wt{\vp})+N(\Dt\mathcal{P}_{t}\wt{\vp})+\wt{N}^{\eta}(\nabla_x\mathcal{P}_{\eta t}\wt{\vp})+\br{M(\abs{\nabla\wt{\vp}}^2)}^{1/2},
\]
where $\vp$ and $\wt{\vp}$ are as in Proposition \ref{Hodgedecp_prop}, and the non-tangential maximal operator $N$ in the second terms on the two right hand sides in defined with respect to the cones of aperture 1.

Let $Q\subset\Rn$ and $\kappa_0\gg 1$ be given. Fix $p_1\in(1,2)$ and define the set $F$ as follows
\begin{equation}\label{F def}
    F:=\set{x\in Q: \Lambda_1(x)+\Lambda_2(x)+D_{p_1}\vp(x)+D_{p_1}\wt{\vp}(x)\le \kappa_0}.
\end{equation}

\begin{lem}\label{F estimate lem}
Let $\epsilon_0$ be as in Proposition \ref{Hodgedecp_prop}. Then 
\begin{equation}
    \abs{Q\setminus F}\lesssim \kappa_0^{-2-\epsilon_0}\abs{Q} 
\end{equation}
uniformly in $\eta$.
\end{lem}

\begin{proof}
By Chebyshev's inequality,
\[
\kappa_0^{2+\epsilon_0}\abs{Q\setminus F}\le\int_{Q\cap\set{\Lambda_1+\Lambda_2+D_{p_1}\vp+D_{p_1}\wt{\vp}>\kappa_0}}\br{\Lambda_1+\Lambda_2+D_{p_1}\vp+D_{p_1}\wt{\vp}}^{2+\epsilon_0}dx.
\]
We apply Proposition \ref{NTalpha_Lp Prop}, Propostion \ref{integratedNT_Lp prop}, and Proposition \ref{DLp_lem}, and their analogs for the adjoint operators, with $p=2+\epsilon_0$, to see that the right-hand side is bounded by 
\[
C\int_{\Rn}\abs{\nabla\vp}^{2+\epsilon_0}+M\br{\abs{\nabla\vp}^2}^{\frac{2+\epsilon_0}{2}}+\abs{\nabla\wt{\vp}}^{2+\epsilon_0}+M\br{\abs{\nabla\wt{\vp}}^2}^{\frac{2+\epsilon_0}{2}}dx,
\]
which, in turn, is bounded by 
\[
C\abs{Q}\fint_{5Q}\br{\abs{\nabla\vp}^{2+\epsilon_0}+\abs{\nabla\wt{\vp}}^{2+\epsilon_0}}.
\]
Then the lemma follows from \eqref{nabla vp_Lp}. 
\end{proof}

We can now choose $\kappa_0$, depending only on $\lambda_0$, $\Lambda_0$ and $n$, such that 
\begin{equation}\label{kappa0}
    \abs{Q\setminus F}\le\frac{1}{1000}\abs{Q}. 
\end{equation}
This completes the construction of $F$ and from now on $\kappa_0$ is fixed.

\subsection{Sawtooth domains and related estimates}\label{theta subsection}

Define $\Omega_0$ to be the sawtooth domain
\begin{equation}
    \Omega_0:=\bigcup_{x\in F}\Gamma_{\eta}(x).
\end{equation}
Define
\begin{equation}
    \theta_t:= \vp-\mathcal{P}^*_t\vp, \qquad \wt{\theta}_t:= \wvp-\mathcal{P}_t\wvp.
\end{equation}

We observe that 
\[
\theta_{\eta t}(x)=-\int_0^{\eta t}\partial_s\mP^*_{s}\vp(x), \text{ and}\quad
\wt{\theta}_{\eta t}(x)=-\int_0^{\eta t}\partial_s\mP_{s}\wvp(x).
\]
So by the definition of the set $F$,
\begin{equation}\label{theta_sawtooth0}
    \abs{\theta_{\eta t}(x)}\le \eta t\kappa_0, \qquad \abs{\wt{\theta}_{\eta t}(x)}\le \eta t\kappa_0\qquad\forall\, (x,t)\in F\times (0,\infty).
\end{equation}

We show that such estimates also hold in the truncated sawtooth domain. Note that we shall eventually choose $\eta>0$ to be small, so we can assume in the sequel that $\eta<1/2$. 

\begin{lem}\label{theta_sawtooth1 Lem} Retain the notation above. The following estimates hold:
\begin{equation*}
    \abs{\theta_{\eta t}(x)}\lesssim\eta t\kappa_0 \quad\text{and}\quad \abs{\wt{\theta}_{\eta t}(x)}\lesssim\eta t\kappa_0,\qquad\forall\, (x,t)\in\Omega_0\cap \br{2Q\times (0,4l(Q))}.
\end{equation*}
\end{lem}
\begin{proof}
We only show the estimate for $\theta_{\eta t}$, for the proof for $\wt{\theta}_{\eta t}$ is similar. Let $(x,t)\in \Omega_0\cap \br{2Q\times (0,4l(Q))}$. Then there exists $x_0\in F$ such that $\abs{x-x_0}\le\eta t$. Since $t<4l(Q)$, and $\eta<\frac{1}{2}$, we have $2B(x_0,\eta t)\subset 5Q$.
We write
\begin{align}\label{theta_sawtooth1}
    \abs{\theta_{\eta t}(x)}
    &\le \abs{\vp(x)-\vp(x_0)}+\abs{\theta_{\eta t}(x_0)}
   +\abs{\mP^*_{\eta t}\br{\vp-\br{\vp}_{2B_{\eta t}(x_0)}}(x_0)}\nonumber\\
   &\qquad+\abs{\mP^*_{\eta t}\br{\vp-\br{\vp}_{2B_{\eta t}(x_0)}}(x)}
\end{align}
where $\br{\vp}_{2B_{\eta t}(x_0)}=\fint_{2B_{\eta t}(x_0)}\vp$. Note that we have used the conservation property, and $\mP^*_{\eta t}\br{\vp}_{B_{\eta t}(x_0)}$ is a constant.

By Proposition \ref{HodgeDecompMoser_prop}, the first term on right-hand side of \eqref{theta_sawtooth1} is bounded by 
\[
C\br{\fint_{2B_{\eta t}(x_0)}\abs{\vp-\vp(x_0)}^{p_1}}^{1/{p_1}}+C\eta t(\norm{\mathbf{c}^s}_{L^{\infty}}+\norm{\mathbf{c}^a}_{\BMO}).
\]
By the definition of $D_{p_1}$ and the set $F$, this is bounded by
\[
C\eta t\br{D_{p_1}\vp(x_0)+\lambda_0+\Lambda_0}\le C\eta t(\kappa_0+\lambda_0+\Lambda_0)\le C\eta t\kappa_0,
\]
with $C=C(\lambda_0,\Lambda_0,n,p_1)$.

By \eqref{theta_sawtooth0}, the second term on the right-hand side of \eqref{theta_sawtooth1} is also bounded by $C\eta t\kappa_0$. Now we take care of the last two terms in \eqref{theta_sawtooth1}. We claim that for any $(y,s)\in\Gamma_{\eta}(x_0)$, 
\begin{equation*}
    \mP^*_{\eta s}\br{\vp-\br{\vp}_{2B_{\eta s}(x_0)}}(y)\lesssim\eta sM(\nabla\vp)(x_0)\lesssim\eta s\kappa_0.
\end{equation*}

Consider the kernel $K_{(\eta s)^2}^*$ associated to $\mP_{\eta s}^*$. Then by the Gaussian estimate for the kernel of the semigroup, 
\[
\abs{K_{(\eta s)^2}^*(y,z)}\lesssim\frac{1}{(\eta s)^n}e^{-\frac{c\abs{y-z}^2}{(\eta s)^2}}.
\]
Then, for  $(y,s)\in\Gamma_{\eta}(x_0)$, 
\begin{multline*}
    \quad\abs{\mP^*_{\eta s}\br{\vp-\br{\vp}_{2B_{\eta s}(x_0)}}(y)}\lesssim\int_{\Rn}\frac{1}{(\eta t)^n}e^{-\frac{c\abs{y-z}^2}{(\eta s)^2}}\abs{\vp(z)-(\vp)_{2B_{\eta s}(x_0)}}dz\\
    \lesssim\int_{2B_{\eta s}(x_0)}\frac{1}{(\eta s)^n}\abs{\vp(z)-(\vp)_{2B_{\eta s}(x_0)}}dz\\
    \qquad+\sum_{k=1}^{\infty}\int_{2^{k+1}B_{\eta s}(x_0)\setminus 2^kB_{\eta s}(x_0)}\frac{1}{(\eta s)^n}e^{-\frac{c\abs{y-z}^2}{(\eta s)^2}}\abs{\vp(z)-(\vp)_{2B_{\eta s}(x_0)}}dz.
\end{multline*}
Since $\abs{y-x_0}\le\eta s$, $\abs{y-z}\ge(2^k-1)\eta s$ for $z\in 2^{k+1}B_{\eta s}(x_0)\setminus 2^kB_{\eta s}(x_0)$. Therefore,
\begin{multline*}
    \abs{\mP^*_{\eta s}\br{\vp-\br{\vp}_{2B_{\eta s}(x_0)}}(y)}\lesssim(\eta s)\fint_{2B_{\eta s}(x_0)}
\abs{\nabla \vp(z)}dz\\
+\sum_{k=1}^{\infty}\int_{2^{k+1}B_{\eta s}(x_0)\setminus 2^kB_{\eta s}(x_0)}\frac{1}{(\eta s)^n}e^{-c(2^k-1)^2}\abs{\vp(z)-(\vp)_{2B_{\eta s}(x_0)}}dz\\
\lesssim(\eta s) M(\nabla\vp)(x_0)+\sum_{k=1}^{\infty}2^{k(n+1)}e^{-c(2^k-1)^2}\eta sM(\nabla \vp)(x_0)\\
\lesssim \eta s M(\nabla\vp)(x_0)\lesssim \eta s\br{M\br{\abs{\nabla\vp}^2}(x_0)}^{1/2}\lesssim \eta s\kappa_0.
\end{multline*}
This finishes the proof.  
\end{proof}

\begin{lem}\label{theta_integral lem} Retain the notation above. The following estimates hold:
\begin{equation*}
    \iint_{\Real_+^{n+1}}\abs{\theta_{\eta t}(x)}^2\frac{dxdt}{t^3}\lesssim\eta^2\abs{Q}, \quad\mbox{and}\quad \iint_{\Real_+^{n+1}}\abs{\wt{\theta}_{\eta t}(x)}^2\frac{dxdt}{t^3}\lesssim\eta^2\abs{Q},
\end{equation*}
where the implicit constants only depend on $\lambda_0$, $\Lambda_0$ and $n$.
\end{lem}

\begin{proof}
We only prove the estimate for $\wt{\theta}_{\eta t}$, for the proof for ${\theta}_{\eta t}$ is similar.
We have the following weighted Hardy's inequality:
\begin{equation}\label{Hardysineq}
    \int_0^{\infty}\br{\frac{1}{t}\int_0^t\abs{f(s)}ds}^p\frac{dt}{t}\le\int_0^{\infty}\abs{f(t)}^p\frac{dt}{t}, \quad\forall\, 1<p<\infty.
\end{equation}
A short and direct proof of \eqref{Hardysineq} is provided at the end.
Recall that 
\[
\abs{\wt{\theta}_{\eta t}}=\abs{\int_0^{\eta t}\partial_s\mP_s\wvp ds}\le\int_0^{\eta t}\abs{\partial_s\mP_s\wvp}ds,
\]
so that
\begin{multline*}
    \int_0^{\infty}\br{\frac{1}{t}\abs{\wt{\theta}_{\eta t}}}^2\frac{dt}{t}\le\int_0^{\infty}\br{\frac{1}{t}\int_0^{\eta t}\abs{\partial_s\mP_s\wvp}ds}^2\frac{dt}{t}\\
    =\eta^2\int_0^{\infty}\br{\frac{1}{t}\int_0^{t}\abs{\partial_s\mP_s\wvp}ds}^2\frac{dt}{t}.
\end{multline*}
By \eqref{Hardysineq}, the last term is bounded by $\eta^2\int_0^{\infty}\abs{\Dt\mP_t\wvp}^2\frac{dt}{t}$.
Then Proposition \ref{Lp_G1Prop} gives
\[
\int_{\Rn}\br{\int_0^{\infty}\abs{\wt{\theta}_{\eta t}}^2\frac{dt}{t^3}}^{p/2}dx
\le\eta^2\int_{\Rn}\br{\int_0^{\infty}\abs{\Dt\mP_t\wvp}^2\frac{dt}{t}}^{p/2}dx\lesssim\eta^2\int_{\Rn}\abs{\nabla\wvp}^pdx
\]
for any $p\ge2$. In particular, with $p=2$, we obtain
\[
\iint_{\Real_+^{n+1}}\abs{\wt{\theta}_{\eta t}(x)}^2\frac{dxdt}{t^3}\lesssim\eta^2\int_{\Rn}\abs{\nabla\wvp}^2dx\lesssim\eta^2\fint_{5Q}\abs{\nabla\wvp}^2dx\abs{Q}\lesssim\eta^2\abs{Q}.
\]
\underline{Proof of \eqref{Hardysineq}.} By H\"older's inequality,
\[\int_0^t\abs{f(s)}ds\le\br{\int_0^t\abs{f(s)}^pds}^{1/p}t^{1-\frac{1}{p}}.\]
    And so \[\br{\frac{1}{t}\int_0^t\abs{f(s)}ds}^p\le\frac{1}{t}\int_0^t\abs{f(s)}^pds.\]
    Integrating in $t$ and using Fubini, we obtain
\begin{multline*}
    \int_0^{\infty}\br{\frac{1}{t}\int_0^t\abs{f(s)}ds}^p\frac{dt}{t}\le\int_0^{\infty}\int_0^t\abs{f(s)}^pds\frac{dt}{t^2}\\
    =\int_0^{\infty}\abs{f(s)}^p\int_s^{\infty}\frac{1}{t^2}dtds
    =\int_0^{\infty}\abs{f(s)}^p\frac{ds}{s}.
\end{multline*} 
\end{proof}

\subsection{The cut-off function}\label{cutoff section}
In this subsection, we define the cut-off function adapted to a thinner sawtooth domain. Define 
\begin{equation}
    \Omega_1=\bigcup_{x\in F}\Gamma_{\frac{\eta}{8}}(x).
\end{equation}
Let $\Phi\in C^{\infty}(\Real)$ with $0\le\Phi\le1$, $\Phi(r)=1$ if $r\le\frac{1}{16}$, and $\Phi(r)=0$ if $r>\frac{1}{8}$.
Define
\begin{equation}\label{def cutoff}
    \Psi(x,t):=\Psi_{\eta,\epsilon}:=\Phi\br{\frac{\delta(x)}{\eta t}}\Phi\br{\frac{t}{32l(Q)}}\br{1-\Phi\br{\frac{t}{16\epsilon}}},
\end{equation}
where $\delta(x):= \dist(x,F)$.

Then $\Psi$ has following properties. First, 
$$\Psi=1 \quad \mbox{on} \quad \bigcup_{x\in F}\Gamma_{\frac{\eta}{16}}\cap\set{2\epsilon<t\le 2l(Q)}.$$
Secondly, $\supp\Psi\subset\Omega_1\cap\set{\epsilon<t<4l(Q)}$. And finally, 
$$\supp\nabla\Psi\subset E_1\cup E_2 \cup E_3, $$
where
\begin{align*}
    E_1&=\set{(x,t)\in 2Q\times (0,4l(Q)):\frac{\eta t}{16}\le\delta(x)\le\frac{\eta t}{8}},\\
    E_2&=\set{(x,t)\in 2Q\times (2l(Q),4l(Q)):\delta(x)\le\frac{\eta t}{8}},\\
    E_3&=\set{(x,t)\in 2Q\times (\epsilon,2\epsilon):\delta(x)\le\frac{\eta t}{8}}.
\end{align*}
In addition,  a direct computation shows
\begin{equation}\label{gradientPsi_ptbound}
    \abs{\nabla\Psi(x,t)}\lesssim\frac{1}{\eta t}\1_{E_1}+\frac{1}{l(Q)}\1_{E_2}+\frac{1}{\epsilon}\1_{E_3}.
\end{equation}

\begin{lem}\label{cutoff est lem} Under the assumptions above,
\begin{equation}\label{gradientPsialpha_integral}
    \int_{\Rn}\br{\int_0^{\infty}\abs{\nabla\Psi}^{\alpha}t^{\alpha-1}dt}^p dx\le C(\eta,\alpha,p,n)\abs{Q},
\end{equation}
for any $\alpha>0$, $p>0$, and
\begin{equation}\label{gradientPsi_integral}
    \underset{\supp\nabla\Psi}{\iint}\frac{dxdt}{t}\le C_n\abs{Q}.
\end{equation}

\end{lem}

\begin{proof}
Using \eqref{gradientPsi_ptbound}, we compute
\begin{multline*}
   \int_{\Rn}\br{\int_0^{\infty}\abs{\nabla\Psi}^{\alpha}t^{\alpha-1}dt}^p dx
    \lesssim_n\int_{2Q}\br{\int_{\frac{8\delta(x)}{\eta}}^{\frac{16\delta(x)}{\eta}}\br{\frac{1}{\eta t}}^{\alpha}t^{\alpha-1}dt}^pdx\\
   +\int_{2Q}\br{\int_{2l(Q)}^{4l(Q)}\br{\frac{1}{l(Q)}}^{\alpha}t^{\alpha-1}dt}^p dx
    +\int_{2Q}\br{\int_{\epsilon}^{2\epsilon}\br{\frac{1}{\epsilon}}^{\alpha}t^{\alpha-1}dt}^pdx\\
    \lesssim_n \frac{1}{\eta^{\alpha p}}\int_{2Q}\br{\int_{\frac{8\delta(x)}{\eta}}^{\frac{16\delta(x)}{\eta}}\frac{1}{t}\,dt}^p dx
    +C_{\alpha,p}\abs{2Q}\\
    \le C(\alpha,p,n)\br{1+\frac{1}{\eta^{\alpha p}}}\abs{Q}\le C(\eta,\alpha,p,n)\abs{Q}.
\end{multline*}
This shows \eqref{gradientPsialpha_integral}, and \eqref{gradientPsi_integral} can be derived similarly:
\begin{align*}
    \underset{\supp\nabla\Psi}{\iint}\frac{dxdt}{t}&\le\iint_{E_1}\frac{dxdt}{t}+\iint_{E_2}\frac{dxdt}{t}+\iint_{E_3}\frac{dxdt}{t}\\
     &\le\int_{2Q}\int_{\frac{8\delta(x)}{\eta}}^{\frac{16\delta(x)}{\eta}}\frac{dt}{t}\,dx
     +\int_{2Q}\int_{2l(Q)}^{4l(Q)}\frac{dt}{t}\,dx+\int_{2Q}\int_{\epsilon}^{2\epsilon}\frac{dt}{t}\,dx\\
     &\le C_n\abs{Q},
\end{align*}
as desired.  
\end{proof}



\section{Proof of the Carleson measure estimate}\label{proof of Carl sec}

Throughout this section, let $Q\subset\Rn$ be fixed, and construct $F\subset Q$ and the cut-off function $\Psi$ as in Section \ref{F section}. Recall that $\kappa_0$ is fixed to ensure that \eqref{kappa0} holds.

Recall that we have the matrix $A=A(x)$ whose entries are functions on $\Rn$, or, independent of $t$, and we write
$$
A=\begin{bmatrix}
\begin{BMAT}{c.c}{c.c}
A_{||} & \mathbf{b} \\
\mathbf{c} & d
\end{BMAT}
\end{bmatrix}.
$$
Write the $n\times 1$ vector $\bd b$ as  $\bd b=\bd {b}_1+\bd b_2$, with $\divg_x\bd b_2=0$. We define a new matrix $A_1$ as follows:
$$
A_1=\begin{bmatrix}
\begin{BMAT}{c.c}{c.c}
A_{||} & \mathbf{b}_1 \\
\mathbf{c}+\bd{b}_2^\intercal & d
\end{BMAT}
\end{bmatrix}
$$
and define $L_1=-\divg A_1\nabla$. Then $L_1$ and $L$ actually define the same operator. To be precise, we have the following
\begin{lem}
For any $u\in W^{1,2}(\Rpl)$ and $v\in W_0^{1,2}(\Real_+^{n+1})$
\begin{equation}\label{AandA1}
    \iint_{\Rpl}A(x)\nabla u(x,t)\cdot\nabla v(x,t)dxdt=\iint_{\Rpl}A_1(x)\nabla u(x,t)\cdot\nabla v(x,t)dxdt.
\end{equation}
In particular, a weak solution to $Lu=0$ in $\Rpl$ is also a weak solution to $L_1u=0$ in $\Rpl$.

\end{lem}

\begin{proof}
We first show \eqref{AandA1} for $u\in W^{1,2}(\Rpl)$ and $v\in C_0^2(\Rpl)$. To this end, we write
\begin{multline}\label{AandA1 eq}
    \iint_{\Rpl}A(x)\nabla u(x,t)\cdot\nabla v(x,t)dxdt\\
    =\iint A_{||}\nabla_x u\cdot\nabla_x v +\br{\bd b_1^\intercal+\bd b_2^{\intercal}}\cdot\nabla_x v\,\Dt u
    +\bd c\cdot\nabla_x u\,\Dt v+d\,\Dt u\,\Dt v\,dxdt,
\end{multline}
and
\begin{multline*}
    \iint_{\Rpl}\bd b_2^\intercal\cdot\nabla_x v\,\Dt u\, dxdt=-\iint_{\Rpl}\Dt\br{\bd b_2^{\intercal}\cdot\nabla_x v}u\,dxdt\\
   =-\iint\bd b_2^{\intercal}\cdot\nabla_x (\Dt v)\,u\,dxdt
   =-\iint\bd b_2^{\intercal}\cdot\nabla_x (\Dt v\,u)dxdt\\+\iint \bd b_2^{\intercal}\cdot\nabla_xu\,\Dt v\,dxdt
   =\iint \bd b_2^{\intercal}\cdot\nabla_x u\,\Dt v\,dxdt,
\end{multline*}
where in the second equality we have used the facts that $\bd b_2$ is $t$-independent and that $v\in C^2$, and in the last equality we have used the divergence-free property of $\bd b_2$. Then \eqref{AandA1 eq} is further equal to
\begin{align*}
    &\iint_{\Rpl} A_{||}\nabla_x u\cdot\nabla_x v +\bd b_1^\intercal\cdot\nabla_x v\,\Dt u
    +\br{\bd c+\bd b_2^\intercal}\cdot\nabla_x u\,\Dt v+d\,\Dt u\,\Dt v\,dxdt\\
    &=\iint_{\Rpl}A_1(x)\nabla u(x,t)\cdot\nabla v(x,t)dxdt.
\end{align*}
Now since $C_0^2(\Rpl)$ is dense in $W_0^{1,2}(\Rpl)$, a limiting argument shows that \eqref{AandA1} holds for all $u\in W^{1,2}(\Rpl)$, $v\in W_0^{1,2}(\Rpl)$. 
\end{proof}

Define 
\begin{equation}\label{A0}
   A_0=\begin{bmatrix}
\begin{BMAT}{c.c}{c.c}
A_{||} & \mathbf{b}-\br{\bd b^a}_{2Q} \\
\mathbf{c}-\br{\bd{c}^a}_{2Q} & d
\end{BMAT}
\end{bmatrix}, 
\end{equation}
where $\br{\bd b^a}_{2Q}=\fint_{2Q}\bd{b}^a$,
and let $L_0:=-\divg A_0\nabla$. 

Note that $\br{\bd b^a}_{2Q}^{\intercal}=-\br{\bd{c}^a}_{2Q}$ by definition of $\bd b^a$ and $\bd c^a$. Also, $\br{\bd b^a}_{2Q}$ is a constant vector so we of course have $\divg_x\br{\bd b^a}_{2Q}=0$. Hence, we can apply the lemma with $\bd b_2=\br{\bd b^a}_{2Q}$. Moreover, observe that 
$$
 A_0=\begin{bmatrix}
\begin{BMAT}{c.c}{c.c}
A_{||} & \mathbf{b}-\br{\bd b^a}_{2Q} \\
\mathbf{c}-\br{\bd{c}^a}_{2Q} & d
\end{BMAT}
\end{bmatrix} 
= 
\begin{bmatrix}
\begin{BMAT}{c.c}{c.c}
A_{||}^s & \mathbf{b}^s \\
\mathbf{c}^s & d
\end{BMAT}
\end{bmatrix} 
+
\begin{bmatrix}
\begin{BMAT}{c.c}{c.c}
A_{||}^a & \mathbf{b}^a-\br{\bd b^a}_{2Q} \\
\mathbf{c}^a-\br{\bd{c}^a}_{2Q} & 0
\end{BMAT}
\end{bmatrix}, 
$$
where $\begin{bmatrix}
\begin{BMAT}{c.c}{c.c}
A_{||}^s & \mathbf{b}^s \\
\mathbf{c}^s & d
\end{BMAT}
\end{bmatrix} $ is the symmetric part of $A_0$, which is the same as the symmetric part of $A$, and
$\begin{bmatrix}
\begin{BMAT}{c.c}{c.c}
A_{||}^a & \mathbf{b}^a-\br{\bd b^a}_{2Q} \\
\mathbf{c}^a-\br{\bd{c}^a}_{2Q} & 0
\end{BMAT}
\end{bmatrix}$ is anti-symmetric, $\BMO$, with the same $\BMO$ semi-norm as $A^a$. We summarize these observations in the following lemma.

\begin{lem}\label{L to L0 lem}
A weak solution to $Lu=0$ in $\Rpl$ is also a weak solution to $L_0u=0$ in $\Rpl$. Moreover, the operator $L_0$ has the same ellipticity constant and $\BMO$ semi-norm as $L$.
\end{lem}

Let $u$ be a bounded weak solution to $Lu=0$ in $\Rpl$ with $\norm{u}_{L^{\infty}}\le 1$. Then $u$ is also a bounded weak solution to $L_0u=0$ in $\Rpl$.
Recall that 
\begin{equation*}
    J_{\eta,\epsilon}=\iint_{\Real_+^{n+1}}A_0\nabla u\cdot \nabla u\,\Psi^2 t\,dxdt.
\end{equation*}
Then by ellipticity of $A_0$ and the support property of $\Psi$, we have
\begin{equation*}
    J_{\eta,\epsilon}\ge\lambda_0\int_{2\epsilon}^{l(Q)}\int_F\abs{\nabla u(x,t)}^2 t\, dxdt.
\end{equation*}
 
The goal of this section is to prove Lemma \ref{main lem}.
Once it is proved, we choose $\sigma$ and $\eta$ to be sufficiently small, so that
\begin{equation}\label{eq Jupb}
    J_{\eta,\epsilon}\le 2\tilde{c}\abs{Q}.
\end{equation}
Now that $\eta$ is fixed, and $\tilde{c}$ is independent of $\epsilon$, we let $\epsilon\to 0$ and thus obtain
\[
\int_0^{l(Q)}\int_F\abs{\nabla u(x,t)}^2 t\, dxdt\le 2\tilde{c},
\]
as desired.

Let us further reduce the statement to the case of smooth coefficients before we prove Lemma \ref{main lem}. We claim that we can assume that $A$ is smooth (and thus $A_0$ is smooth) in Lemma \ref{main lem}, as long as all bounds depend on $A$ only through its ellipticity constant and $\BMO$ semi-norm. If $A$ is not smooth, we take $A_\delta=\xi_\delta*A_0$, where $\xi_\delta(X)=\delta^{-n-1}\xi(\frac{X}{\delta})$ is an approximate identity. Then $A_\delta$ converges to $A_0$ locally in $L^p(\Real^{n+1})$ for all $1\le p<\infty$ as $\delta\to0$, and 
\begin{equation}\label{eq BMObd}
    \norm{A^a_\delta}_{\BMO(\Real^{n+1})}\le\norm{A_0^a}_{\BMO(\Real^{n+1})}.
\end{equation}
See e.g. \cite{qian2019parabolic} Proposition 3.3 for a proof of \eqref{eq BMObd}. Then the desired result in the non-smooth case follows from a limiting argument. 
To see this, fix any $u$ that satisfies $Lu=0$ in $\Real^{n+1}_+$ and $\norm{u}_{L^\infty(\Real^{n+1}_+)}\le 1$. Then fix a cube $Q\subset\Rn$ and define the cutoff function $\Psi=\Psi_{Q,\epsilon,\eta}$ as in \eqref{def cutoff}. Take cubes $\wt{Q}_0$ and $\wt{Q}_1$ such that 
$$\supp\Psi\subset\subset\wt{Q}_0\subset\subset\wt{Q}_1\subset\subset\Real^{n+1}_+.$$ 
Now let $u_\delta$ satisfy $L_\delta u_\delta=-\divg(A_\delta\nabla u_\delta)=0$ in $\wt{Q}_0$ and $u_\delta=u$ on $\bdy\wt{Q}_0$. Let furthermore
 \[J_\delta=J_{\eta,\epsilon,\delta}=\iint_{\Real_+^{n+1}}A_\delta\nabla u_\delta\cdot \nabla u_\delta\,\Psi^2 t\,dxdt.\]
Since $A_\delta$ is smooth, we can use the result in the smooth case and have $J_\delta\le C\abs{Q}$ by \eqref{eq Jupb}. The constant $C$ is independent of $\epsilon$, and can be independent of $\delta$ because of \eqref{eq BMObd}. Hence, it remains to show that $\abs{J_\delta-J}\to 0$ as $\delta\to 0$, where 
\[J=J_{\eta,\epsilon}=\iint_{\Real_+^{n+1}}A_0\nabla u\cdot \nabla u\,\Psi^2 t\,dxdt.\]
Notice that $A^a_\delta\nabla u_\delta\cdot \nabla u_\delta=A^a_0\nabla u\cdot \nabla u=0$ by anti-symmetry, and thus
\[
\abs{J_\delta-J}=\abs{\iint_{\Real_+^{n+1}}A^s_\delta\nabla u_\delta\cdot \nabla u_\delta\,\Psi^2 t\,dxdt-\iint_{\Real_+^{n+1}}A_0^s\nabla u\cdot \nabla u\,\Psi^2 t\,dxdt}.
\]
We write
\begin{multline*}
    A^s_\delta\nabla u_\delta\cdot \nabla u_\delta-A_0^s\nabla u\cdot \nabla u\\
    =A^s_\delta\nabla(u_\delta-u)\cdot \nabla u_\delta+A^s_\delta\nabla u\cdot\nabla(u_\delta-u)+(A_\delta^s-A_0^s)\nabla u\cdot \nabla u,
\end{multline*}
and get
\begin{multline*}
    \abs{J_\delta-J}
    \le\,\sup_{\wt Q_0}\,(\Psi^2 t)\bigg\{\abs{\iint_{\wt Q_0}A^s_\delta\nabla(u_\delta-u)\cdot \nabla u_\delta\, dxdt}\\
    +\abs{\iint_{\wt Q_0}A^s_\delta\nabla u\cdot\nabla(u_\delta-u)dxdt}+\abs{\iint_{\wt Q_0}(A_\delta^s-A_0^s)\nabla u\cdot \nabla u\,dxdt}\bigg\}\\
    \le C
    \Big\{\iint_{\wt Q_0}\abs{\nabla(u_\delta-u)}(\abs{\nabla u_\delta}+\abs{\nabla u})dxdt+\iint_{\wt Q_0}\abs{A_\delta^s-A_0^s}\abs{\nabla u}^2dxdt\Big\}\\
    =:C(I_1+I_2).
\end{multline*}
Using ellipticity of $A_\delta$, and taking $u-u_\delta$ as a test function to both $L_\delta u_\delta=0$ and $L_0 u=0$ in $\wt{Q}_0$, one can get 
\begin{multline*}
    \iint_{\wt Q_0}\abs{\nabla(u_\delta-u)}^2dxdt\le\lambda_0^{-1}\iint_{\wt Q_0}A_\delta\nabla(u_\delta-u)\cdot\nabla(u_\delta-u)dxdt\\
    =\lambda_0^{-1}\iint_{\wt Q_0}(A_0-A_\delta)\nabla u\cdot\nabla(u_\delta-u)dxdt.
\end{multline*}
Then by Cauchy-Schwarz inequality, 
\[
    \iint_{\wt Q_0}\abs{\nabla(u_\delta-u)}^2dxdt\le \lambda_0^{-2}\iint_{\wt Q_0}\abs{A_0-A_\delta}^2\abs{\nabla u}^2dxdt.
\]
Using H\"older inequality, reverse H\"older inequality for $\nabla u$, and the fact that
$\norm{A_0-A_\delta}_{L^p(Q_0)}\to 0$ as $\delta\to 0$ ($p$ will be large), we obtain 
\begin{equation}\label{udelta conv}
    \iint_{\wt Q_0}\abs{\nabla(u_\delta-u)}^2dxdt\to 0\quad \mbox{as}\quad \delta\to 0.
\end{equation}
 Notice that \[\iint_{\wt Q_0}\abs{\nabla u_\delta}^2dxdt\le \lambda_0^{-4}\iint_{\wt Q_0}\abs{\nabla u}^2dxdt,\] 
and thus $I_1\to 0$ by Cauchy-Schwarz inequality and \eqref{udelta conv}. The second term, $I_2$, converges to 0 by the dominated convergence theorem. 

This justifies the claim that we only need to prove Lemma \ref{main lem} for $A$ smooth. Notice that in this case, $\vp$, $\wt{\vp}$, $\mP_t\wt{\vp}$, $\mP^*_t\wt{\vp}$ and $u$ are all smooth by interior regularity of elliptic equations.
\vskip 0.08 in

Now we are ready for the \\
\noindent {\it Proof of Lemma \ref{main lem}.}  
In the sequel, we shall simply write $J$ for $J_{\eta,\epsilon}$. We shall not distinguish a column vector and a row vector, namely, we shall not use the sign of transposition, as it should be clear from the context. We denote by $c$ some constant depending only on $\lambda_0$, $\Lambda_0$ and $n$, and by $\tilde{c}$ a constant depending additionally on $\sigma$ and $\eta$.

Since $u$ is a weak solution to $L_0u=0$ in $\Rpl$, 
\[
\iint_{\Rpl}A_0\nabla u\cdot\nabla( u \,\Psi^2 t)\,dxdt=0,
\]
where we have chosen $u\,\Psi^2 t$ to be the test function. Therefore,
\begin{multline*}
    J=\iint_{\Rpl} A_0\nabla u \cdot \nabla u\,\Psi^2t\,dxdt\\
    = -\iint A_0\nabla u\cdot\nabla(\Psi^2)ut\,dxdt-\iint A_0\nabla u\cdot \nabla t\, u\Psi^2dxdt
    =: J_1+J_2.
\end{multline*}
The first term 
\begin{align*}
    J_1&=-\iint A_{||}\nabla_x u\cdot\nabla_x(\Psi^2)ut\,dxdt-\iint (\bd b-\br{\bd{b}^a}_{2Q})\cdot\nabla_x(\Psi^2)\Dt u\,ut\,dxdt\\
    &-\iint (\bd c-\br{\bd{c}^a}_{2Q})\cdot\nabla_x u\, \Dt(\Psi^2)ut\,dxdt-\iint d\,\Dt u\,\Dt(\Psi^2)ut\,dxdt\\
    &=: J_{11}+J_{12}+J_{13}+J_{14}.
\end{align*}

For $J_{11}$, we claim that 
\begin{equation}\label{J11equal}
    J_{11}=-\iint \br{A_{||}-(A_{||}^a)_{2Q}}\nabla_x u\cdot\nabla_x(\Psi^2)ut\,dxdt.
\end{equation}
This is because 
\[
\iint (A_{||}^a)_{2Q}\nabla_x u\cdot\nabla_x(\Psi^2)ut\,dxdt=\frac{1}{2}\iint (A_{||}^a)_{2Q}\nabla_x (u^2)\cdot\nabla_x(\Psi^2t)dxdt,
\]
and the last integral is 0 because $(A_{||}^a)_{2Q}$ is a constant anti-symmetric matrix, and $\Psi^2t$ is $C^2$. Therefore, 
\begin{align*}
    J_{11}&=-\iint A_{||}^s\nabla_x u\cdot\nabla_x(\Psi^2)ut\,dxdt-\iint \br{A_{||}^a-(A_{||}^a)_{2Q}}\nabla_x u\cdot\nabla_x(\Psi^2)ut\,dxdt\\
    &=: J_{111}+J_{112}.
\end{align*}
For $J_{111}$, we have
\begin{multline*}
    \abs{J_{111}}=2\abs{\iint_{\Rpl}A_{||}^s\nabla_x u\, \Psi\, u \nabla_x\,\Psi\, t\,dxdt}\\
  \le\frac{2}{\lambda_0}\iint_{\Rpl}\abs{\nabla_xu}\abs{\Psi}t^{1/2}\abs{\nabla_x\Psi}t^{1/2}dxdt\\
  \le \sigma\lambda_0\iint_{\Rpl}\abs{\nabla u}^2\Psi^2 t\, dxdt+\frac{c}{\sigma}\iint_{\Rpl}\abs{\nabla\Psi}^2t\,dxdt
  \le\sigma J+\tilde{c}\abs{Q},
\end{multline*}
where in the first inequality we have used $\norm{A_{||}^s}_{L^{\infty}}\le\lambda_0^{-1}$ and $\norm{u}_{L^{\infty}}\le1$, and in the last step we have used Lemma \ref{cutoff est lem}. For $J_{112}$, by H\"older's inequality,
\begin{multline}\label{J112}
     \frac{1}{2}\abs{J_{112}}=\abs{\int_{2Q}\int_0^{4l(Q)}\br{A_{||}^a-(A_{||}^a)_{2Q}}\nabla_xu\cdot\nabla_x{\Psi} (u\Psi t)dtdx}\\
    \le \br{\int_{2Q}\abs{A_{||}^a-(A_{||}^a)_{2Q}}^{\alpha'}dx}^{\frac1{\alpha'}}\br{\int_{2Q}\Big(\int_0^{4l(Q)}\abs{\nabla u}\abs{\Psi t^{1/2}}\abs{\nabla\Psi}t^{1/2}dt\Big)^{\alpha}dx}^{\frac1{\alpha}}\\
    \le c\abs{Q}^{\frac1{\alpha'}}\norm{A_{||}^a}_{\BMO}
    \Big\{\int_{2Q}\Big(\int_0^{4l(Q)}\abs{\nabla u}^2\Psi^2t\,dt\Big)^{\frac{\alpha}{2}}\Big(\int_0^{4l(Q)}\abs{\nabla\Psi}^2t\,dt\Big)^{\frac{\alpha}{2}}dx \Big\}^{\frac1{\alpha}},
\end{multline}
where $\alpha$ is any number between 1 and 2. Now we use H\"older inequality with $p=\frac{2}{\alpha}$ to bound \eqref{J112} by
\begin{align*}
   c\abs{Q}^{\frac1{\alpha'}}\br{\int_{2Q}\int_0^{4l(Q)}\abs{\nabla u}^2\Psi^2t\,dtdx}^{1/2}\br{\int_{2Q}\Big(\int_0^{4l(Q)}\abs{\nabla\Psi}^2t\,dt\Big)^{\frac{\alpha}{2-\alpha}}dx}^{\frac{2-\alpha}{2\alpha}},
\end{align*}
which by Lemma \ref{cutoff est lem} can then be bounded by
\[
\tilde{c}J^{1/2}\abs{Q}^{\frac{1}{\alpha'}+\frac{2-\alpha}{2\alpha}}=\tilde{c}J^{1/2}\abs{Q}^{1/2}.
\]
Then Young's inequality gives
\[
\abs{J_{112}}\le\sigma J+\tilde{c}\abs{Q}.
\]

Note that $J_{12}$ and $J_{13}$ can be estimated similar as \eqref{J11equal}. So both of then are bounded by $\sigma J+\tilde{c}\abs{Q}$. Since $\norm{d}_{L^{\infty}}\le\lambda_0^{-1}$, $J_{14}$ can be also bounded by $\sigma J+\tilde{c}\abs{Q}$ using Young's inequality and Lemma \ref{cutoff est lem}.

For $J_2$, we compute
\begin{align*}
   J_2&=-\iint_{\Rpl}A_0\nabla u\cdot \bd{e}_{n+1}\,u\Psi^2dxdt\\
   &=-\iint_{\Rpl}\br{\bd c-(\bd{c}^a)_{2Q}}\cdot\nabla_x u(u\Psi^2)dxdt-\iint_{\Rpl}d\,\Dt u(u\Psi^2)dxdt\\
   &=: J_{21}+J_{22}.
\end{align*}

For $J_{22}$, since $d$ is $t$-independent, integration by parts gives
\[
J_{22}=-\frac{1}{2}\iint_{\Rpl}d\,\Dt(u^2)\Psi^2dxdt=\iint_{\Rpl}d\,u^2\Psi\Dt\Psi\,dxdt.
\]
Thus $\abs{J_{22}}\le \tilde{c}\abs{Q}$ again by Lemma \ref{cutoff est lem}.
For $J_{21}$, we write
\begin{multline*}
     J_{21}=-\iint_{\Rpl}\br{\bd c-(\bd{c}^a)_{2Q}}\cdot\nabla_x\br{\frac{u^2\Psi^2}{2}}dxdt\\
     +\iint_{\Rpl}\br{\bd c-(\bd{c}^a)_{2Q}}\cdot\nabla_x\Psi (u^2\Psi) dxdt
  =: J_{211}+J_{212}.
\end{multline*}
Going further, 
\begin{multline*}
    J_{212}=\iint_{\Rpl}\bd c^s\cdot\nabla_x\Psi (u^2\Psi) dxdt+\iint_{\Rpl}\br{\bd c^a-(\bd{c}^a)_{2Q}}\cdot\nabla_x\Psi (u^2\Psi) dxdt\\=: J_{2121}+J_{2122}.
\end{multline*}
Then again by Lemma \ref{cutoff est lem}, $\abs{J_{2121}}\le \tilde{c}\abs{Q}$. For $J_{2122}$, 
\begin{align*}
    \abs{J_{2122}}&\le\int_{2Q}\abs{\bd c^a-(\bd{c}^a)_{2Q}}\br{\int_0^{4l(Q)}\abs{\nabla_x\Psi(u^2\Psi)}dt}dx\\
    &\le\br{\int_{2Q}\abs{\bd c^a-(\bd{c}^a)_{2Q}}^{\alpha'}dx}^{1/{\alpha'}}\br{\int_{2Q}\br{\int_0^{4l(Q)}\abs{\nabla\Psi}\abs{u^2\Psi}dt}^{\alpha}}^{1/{\alpha}}\\
    &\le c\abs{Q}^{1/{\alpha'}}\br{\int_{2Q}\br{\int_0^{4l(Q)}\abs{\nabla\Psi}dt}^{\alpha}dx}^{1/{\alpha}}\le \tilde{c}\abs{Q}.
\end{align*}
For $J_{211}$, we use \eqref{cHodege} to get
\[
J_{211}=\iint_{\Rpl}A^*_{||}\nabla_x\vp\cdot\nabla_x\br{\frac{u^2\Psi^2}{2}}dxdt.
\]
Recall that we defined $\theta_{\eta t}=\vp-\mP^*_{\eta t}\vp$ in Section \ref{theta subsection}. We compute
\begin{multline*}
   J_{211}=\iint_{\Rpl}A_{||}^*\nabla_x\theta_{\eta t}\cdot\nabla_x\Big(\frac{u^2\Psi^2}{2}\Big)dxdt
    +\iint_{\Rpl}A_{||}^*\nabla_x\mP^*_{\eta t}\vp\cdot\nabla_x\Big(\frac{u^2\Psi^2}{2}\Big)dxdt\\
    =\iint_{\Rpl}\br{A_{||}^*-(A_{||}^{*a})_{2Q}}\nabla_x\theta_{\eta t}\cdot\nabla_x\Big(\frac{u^2\Psi^2}{2}\Big)dxdt\\
     +\iint_{\Rpl}A_{||}^*\nabla_x\mP^*_{\eta t}\vp\cdot\nabla_x\Big(\frac{u^2\Psi^2}{2}\Big)dxdt
     =: J_{2111}+J_{2112},
\end{multline*}
where in the second equality we have used the assumption that the coefficients are smooth, which implies that  $u^2$ is smooth, and thus $(A_{||}^{*a})_{2Q}$ being a constant anti-symmetric matrix gives
\[
\iint_{\Rpl}(A_{||}^{*a})_{2Q}\nabla_x\theta_{\eta t}\cdot\nabla_x\br{\frac{u^2\Psi^2}{2}}dxdt=0.
\]
For $J_{2112}$, integration by parts with respect to $t$ gives
\begin{multline*}
     J_{2112}=-\iint_{\Rpl}\Dt\br{A_{||}^*\nabla_x\mP_{\eta t}^*\vp\cdot\nabla_x\br{\frac{u^2\Psi^2}{2}}}t\,dxdt\\
    =-\iint_{\Rpl}A_{||}^*\nabla_x\Dt\mP_{\eta t}^*\vp\cdot\nabla_x\br{\frac{u^2\Psi^2}{2}}t\,dxdt\\
    -\iint_{\Rpl}A_{||}^*\nabla_x\mP_{\eta t}^*\vp\cdot\nabla_x\Dt\br{\frac{u^2\Psi^2}{2}}t\,dxdt
    =: I_1+I_2.
\end{multline*}
By the same reasoning as for \eqref{J11equal}, we have
\begin{align*}
    I_1&=-\iint_{\Rpl}\br{A_{||}^*-(A_{||}^{*a})_{2Q}}\nabla_x\Dt\mP_{\eta t}^*\vp\cdot\nabla_x\br{\frac{u^2\Psi^2}{2}}t\,dxdt\\
    &=-\iint_{\Rpl}A_{||}^{*s}\nabla_x\Dt\mP_{\eta t}^*\vp\cdot\nabla_x\br{\frac{u^2\Psi^2}{2}}t\,dxdt\\
    &\qquad-\iint_{\Rpl}\br{A_{||}^{*a}-(A_{||}^{*a})_{2Q}}\nabla_x\Dt\mP_{\eta t}^*\vp\cdot\nabla_x\br{\frac{u^2\Psi^2}{2}}t\,dxdt\\
    &=: I_{11}+I_{12}.
\end{align*}

Then, applying Proposition \ref{Lp_G2xProp} to the operator $L^*_{||}=-\divg A^*_{||}\nabla$, with $p=2$, 
\[
\br{\iint_{\Rpl}\abs{\nabla_x\Dt\mP_{\eta t}^*\vp}^2t\,dxdt}^{1/2}\le c\norm{\nabla\vp}_{L^2(\Rn)}.
\]
So by Cauchy-Schwartz inequality and by \eqref{nabla vp_Lp}, 
\begin{align*}
    \abs{I_{11}}&\le c\br{\iint_{\Rpl}\abs{\nabla_x\Dt\mP_{\eta t}^*\vp}^2t\,dxdt}^{1/2}
    \br{\iint_{\Rpl}\abs{\nabla_x(u^2\Psi^2)}^2t\,dxdt}^{1/2}\\
    &\le c\abs{Q}^{1/2}\br{\iint_{\Rpl}\abs{\nabla_x u}^2\Psi^2\, t\,dxdt+\iint_{\Rpl}\abs{\nabla_x\Psi}^2t\,dxdt}^{1/2}.
\end{align*}
Then Lemma \ref{cutoff est lem} and Young's inequality give
\begin{align*}
    \abs{I_{11}}\le c\abs{Q}^{1/2}\br{J+\tilde{c}\abs{Q}}^{1/2}\le\sigma J+\tilde{c}\abs{Q}.
\end{align*}
For $I_{12}$, we use H\"older inequality to get
\begin{multline*}
    \abs{I_{12}}\le \frac{1}{2}\br{\int_{2Q}\abs{A_{||}^{*a}-(A_{||}^{*a})_{2Q}}^{\alpha'}dx}^{\frac1{\alpha'}}\\
\times    \Big\{\int_{\Rn}\br{\int_0^{\infty}\abs{\nabla_x\Dt\mP_{\eta t}^*\vp}\abs{\nabla_x\br{u^2\Psi^2}}t\,dt}^{\alpha}dx\Big\}^{1/{\alpha}}\\
    \le c\abs{Q}^{1/{\alpha'}}\br{\int_{\Rn}\br{\int_0^{\infty}\abs{\nabla_x\Dt\mP^*_{\eta t}\vp}^2t\,dt}^{\frac{\alpha}{2-\alpha}}dx}^{\frac{2-\alpha}{2\alpha}}\\
\times    \br{\int_{\Rn}\int_0^{\infty}\abs{\nabla_x(u^2\Psi^2)}^2t\,dtdx}^{1/2}.
\end{multline*}
Letting $\frac{\alpha}{2-\alpha}=\frac{2+\epsilon_0}{2}$, then by Proposition \ref{Lp_G2xProp}, \eqref{nabla vp_Lp} and Lemma \ref{cutoff est lem}, 
\begin{align*}
    \abs{I_{12}}\le c\abs{Q}^{1/{\alpha'}}\abs{Q}^{\frac{2-\alpha}{2\alpha}}\br{J+\tilde{c}\abs{Q}}^{1/2}
    \le \sigma J+\tilde{c}\abs{Q}.
\end{align*}

For $I_2$, by the definition of $L^*_{||}$, we can write $I_2$ as 
\[
I_2=-\frac{1}{2}\int_0^{\infty}\int_{\Rn}L_{||}^*\mP^*_{\eta t}\vp\,\Dt(u^2\Psi^2)dx\,t\,dt.
\]
By the Cauchy-Schwartz inequality,
\begin{align*}
   \abs{I_2}\le c\br{\iint_{\Rpl}\abs{L^*_{||}\mP^*_{\eta t}\vp}^2t\,dtdx}^{1/2}
    \br{\iint_{\Rpl}\abs{\Dt(u^2\Psi^2)}^2t\,dtdx}^{1/2}.
\end{align*}
By Proposition \ref{Lp_G1Prop},
\[
\br{\iint_{\Rpl}\abs{L^*_{||}\mP^*_{\eta t}\vp}^2t\,dtdx}^{1/2}\le \tilde{c}\norm{\nabla\vp}_{L^2(\Rn)}.
\]
So by \eqref{nabla vp_Lp} and Lemma \ref{cutoff est lem}, we have
\begin{align*}
    \abs{I_2}&\le \tilde{c}\norm{\nabla\vp}_{L^2(\Rn)}\br{\iint_{\Rpl}\abs{\Dt(u^2\Psi^2)}^2t\,dtdx}^{1/2}\\
    &\le \tilde{c}\abs{Q}^{1/2}(J+\tilde{c}\abs{Q})^{1/2}\le\sigma J+\tilde{c}\abs{Q}. 
\end{align*}

We now return to $J_{2111}$. Write
\begin{align*}
    J_{2111}&=\iint_{\Rpl}\br{A_{||}^*-(A_{||}^{*a})_{2Q}}\nabla_x\theta_{\eta t}\cdot\nabla_x u(u\Psi^2)dxdt\\
    &\qquad+\frac{1}{2}\iint_{\Rpl}\br{A_{||}^*-(A_{||}^{*a})_{2Q}}\nabla_x\theta_{\eta t}\cdot\nabla_x(\Psi^2)u^2dxdt\\
    &=: II_1+II_2.
\end{align*}
For $II_2$, we split it up into the integral involving $A_{||}^{*s}$ and the integral involving $A_{||}^{*a}-(A_{||}^{*a})_{2Q}$ as before. We only treat the integral involving $A_{||}^{*a}-(A_{||}^{*a})_{2Q}$ (denoted by $II_2^a$) as the estimate for the former is similar and easier. By the Cauchy-Schwarz inequality and \eqref{gradientPsi_ptbound}, we can write
\begin{multline}\label{II_2^a}
     \abs{II_2^a}\le c\Big(\int_{2Q}\abs{A_{||}^{*a}-(A_{||}^{*a})_{2Q}}^2dx\Big)^{\frac1{2}}\br{\int_{2Q}\Big(\int_0^{4l(Q)}\abs{\nabla_x\theta_{\eta t}}\abs{\nabla_x\Psi}dt\Big)^2dx}^{\frac1{2}}\\
    \le \tilde{c}\abs{Q}^{1/2}\Big\{ \br{\int_{2Q}\Big(\int_0^{4l(Q)}\abs{\nabla_x\theta_{\eta t}}\1_{E_1}\frac{dt}{t}\Big)^2dx}^{1/2}\\
    +\br{\int_{2Q}\Big(\int_0^{4l(Q)}\abs{\nabla_x\theta_{\eta t}}\1_{E_2}\frac{dt}{l(Q)}\Big)^2dx}^{1/2}\\
    +\br{\int_{2Q}\Big(\int_0^{4l(Q)}\abs{\nabla_x\theta_{\eta t}}\1_{E_3}\frac{dt}{\epsilon}\Big)^2dx}^{1/2}
    \Big\}\\
    =: \tilde{c}\abs{Q}^{1/2}\br{(II_{21}^a)^{1/2}+(II_{22}^a)^{1/2}+(II_{23}^a)^{1/2}}.
\end{multline}
Observing that 
$$\int_0^{4l(Q)}\1_{E_1}(x,t)\frac{dt}{t}\le\int_{\frac{8\delta(x)}{\eta}}^{\frac{16\delta(x)}{\eta}}\frac{dt}{t}=\ln{2}$$ 
and  using the Cauchy-Schwarz inequality, we show that
\begin{align*}
    II_{21}^a&\le\int_{2Q}\br{\int_0^{4l(Q)}\abs{\nabla_x\theta_{\eta t}}^2\1_{E_1}\frac{dt}{t}}\br{\int_0^{4l(Q)}\1_{E_1}\frac{dt}{t}}dx\\
    &\le c\sum_{k}\sum_{Q'\in\mathcal{D}^{\eta}_k}\int_{Q'}\int_{2^{-k}}^{2^{-k+1}}\abs{\nabla_x\theta_{\eta t}}^2\1_{E_1}\frac{dt}{t}dx,
\end{align*}
where $\mathcal{D}^{\eta}_k$ denotes the grid of dyadic cubes such that
\begin{equation}\label{dyadicgrid}
    \frac{1}{64}\eta 2^{-k}\le l(Q')<\frac{1}{32}\eta 2^{-k}, \quad Q'\in\mathcal{D}^{\eta}_k.
\end{equation}

Consider for any fixed $k$ and $Q'\in \mathcal{D}^{\eta}_k$, for which $Q'\times [2^{-k},2^{-k+1}]\cap E_1\neq\emptyset$. One can show that for such $Q'$, there exists some $x_0\in F$ such that 
\begin{equation}\label{Q'subset}
    2Q'\subset B(x_0,\eta 2^{-k}).
\end{equation}
This implies that for any $t\in[2^{-k},2^{-k+1}]$,
\begin{align}
    \fint_{Q'}\abs{\nabla_x\theta_{\eta t}}^2dx&\lesssim_n \fint_{B(x_0,\eta 2^{-k})}\abs{\nabla_x\mP^*_{\eta t}\vp(x)}^2dx+\fint_{B(x_0,\eta 2^{-k})}\abs{\nabla\vp(x)}^2dx\nonumber\\
    &\lesssim_n\fint_{B(x_0,\eta t)}\abs{\nabla_x\mP^*_{\eta t}\vp(x)}^2dx+\fint_{B(x_0,\eta t)}\abs{\nabla\vp(x)}^2dx\nonumber\\
    &\lesssim_n\br{\wt{N}^{\eta}(\nabla_x\mP^*_{\eta t}\vp)}^2(x_0)+M\br{\abs{\nabla_x\vp}^2}(x_0)\lesssim\kappa_0^2\label{II21(1)},
\end{align}
by definition of the integrated non-tangential maximal function \eqref{integratedNT defn} and the definition of the set $F$.

By \eqref{dyadicgrid} and the definition of $E_1$, one can show there exists some uniform constant $C>1$ such that
\begin{equation*}
    Q'\times [2^{-k},2^{-k+1}]\subset \wt{E}_1:= \set{(y,s)\in 2Q\times(0,4l(Q)):\frac{\eta s}{C}\le \delta(y)\le C\eta s},
\end{equation*}
which implies
\begin{equation}\label{II21(2)}
    \abs{Q'}\lesssim\int_{Q'}\int_{2^{-k}}^{2^{-k+1}}\1_{\wt{E}_1}(y,s)\frac{ds}{s}dy.
\end{equation}
Using \eqref{II21(1)} and \eqref{II21(2)}, we estimate $II_{21}^a$ as follows:
\begin{multline*}
    II_{21}^a\le c\sum_{k}\sum_{Q'\in\mathcal{D}^{\eta}_k}\int_{2^{-k}}^{2^{-k+1}}\fint_{Q'}\abs{\nabla_x \theta_{\eta t}}^2dx\abs{Q'}\frac{dt}{t}\\
    \le c\kappa_0^2\sum_{k}\sum_{Q'\in\mathcal{D}^{\eta}_k}\br{\int_{2^{-k}}^{2^{-k+1}}\frac{dt}{t}}\int_{Q'}\int_{2^{-k}}^{2^{-k+1}}\1_{\wt{E}_1}(y,s)\frac{ds}{s}dy\\
    \le c\iint_{\Rpl}\1_{\wt{E}_1}(y,s)\frac{dsdy}{s}
    \le c\int_{2Q}\int_{\frac{\delta(y)}{C\eta}}^{\frac{C\delta(y)}{\eta}}\frac{ds}{s}dy\le c\abs{Q}.
\end{multline*}

For $II_{22}^a$, notice that 
\begin{multline}\label{eq II22a}
    II_{22}^a=\int_{2Q}\br{\int_{2l(Q)}^{4l(Q)}\abs{\nabla_x\theta_{\eta t}}\1_{E_2}\frac{dt}{l(Q)}}^2dx\\
        \le \frac{4}{l(Q)}\int_{2l(Q)}^{4l(Q)}\int_{2Q}\abs{\nabla_x\mP^*_{\eta t}\vp(x)}^2\1_{E_2}dxdt\\
    +\frac{4}{l(Q)}\int_{2l(Q)}^{4l(Q)}\int_{2Q}\abs{\nabla_x\vp(x)}^2\1_{E_2}dxdt.
\end{multline}
By the definition of $E_2$, one has $\delta(x)\le\frac{\eta}{2}l(Q)$ for any $(x,t)\in E_2$. Denote by $\pi_{E_2}$ the projection of $E_2$ onto $\set{t=0}$, then $\pi_{E_2}$ can be covered by balls $B(x_i,2\eta l(Q))$ with $x_i\in F$, and the number $N$ of these balls can be bounded by $c_n\eta^{-n}$, where $c_n$ is a constant depending only on the dimension. So the first term on the right-hand side of \eqref{eq II22a} is bounded by 
\begin{multline*}
    c\eta^nl(Q)^{n-1}\sum_{i=1}^N\int_{2l(Q)}^{4l(Q)}\fint_{B(x_i,2\eta l(Q))}\abs{\nabla_x\mP^*_{\eta t}\vp}^2dxdt\\
    \le  \frac{c\eta^n\abs{Q}}{l(Q)}\sum_{i=1}^N\int_{2l(Q)}^{4l(Q)}\fint_{B(x_i,2\eta l(Q))}\wt{N}^{\eta}(\nabla_x\mP^*_{\eta t}\vp)^2(x_i)^2dt\le c\kappa_0^2\abs{Q},
\end{multline*}
using the definition of $\wt{N}^{\eta}$, the definition of the set $F$, and $N\le c\eta^{-n}$. For the second term on the right-hand side of \eqref{eq II22a}, notice that $\pi_{E_2}\subset B(x_0,2l(Q))$ for any $x_0\in F$. Then the second term is bounded by
\begin{multline*}
    \frac{c\abs{Q}}{l(Q)}\int_{2l(Q)}^{4l(Q)}\fint_{B(x_0,2l(Q))}\abs{\nabla\vp(x)}^2dxdt
    \le \frac{c\abs{Q}}{l(Q)}\int_{2l(Q)}^{4l(Q)}M\br{\abs{\nabla\vp}^2}(x_0)dt\\
    \le c\kappa_0^2\abs{Q},
\end{multline*}
using again the definition of the set $F$. Combining these two estimates with \eqref{eq II22a} we obtain the bound $II_{22}^a\le c\abs{Q}$.

By a similar argument, one can show that $II_{23}^a\le\tilde{c}\abs{Q}$ as well. Combining these results with \eqref{II_2^a}, we have shown that $\abs{II_{2}^a}\le\tilde{c}\abs{Q}$, and thus $\abs{II_2}\le\tilde{c}\abs{Q}$.

We now deal with $II_1$. Write
\begin{align*}
   II_1&=\iint_{\Rpl}\br{A_{||}^*-(A^{*a}_{||})_{2Q}}\nabla_x(\theta_{\eta t}u\Psi^2)\cdot\nabla_x u dxdt  \\
  &\quad-\iint_{\Rpl}\br{A_{||}^*-(A^{*a}_{||})_{2Q}}\nabla_x u\cdot\nabla_x u\,(\theta_{\eta t}\Psi^2) dxdt\\
  &\quad-\iint_{\Rpl}\br{A_{||}^*-(A^{*a}_{||})_{2Q}}\nabla_x(\Psi^2)\cdot\nabla_x u\,(u\theta_{\eta t}) dxdt\\
  &=: II_{11}+II_{12}+II_{13}.
\end{align*}
We use Lemma \ref{theta_sawtooth1 Lem} to bound $II_{12}$ and $II_{13}$. We rewrite Lemma \ref{theta_sawtooth1 Lem} in the following way 
\begin{equation}\label{theta_sawtooth1 suppPsi}
    \abs{\theta_{\eta t}(x)}\lesssim\kappa_0\eta t \quad\text{for } (x,t)\in\supp\Psi. 
\end{equation}
Note that by anti-symmetry,
\[
II_{12}=-\iint_{\Rpl}A_{||}^{*s}\nabla_x u\cdot\nabla_x u\,(\theta_{\eta t}\Psi^2) dxdt,
\]
and thus
\begin{align*}
    \abs{II_{12}}\le c\eta\iint_{\Rpl}\abs{\nabla u}^2\Psi^2 t\,dxdt\le c\eta J.
\end{align*}
For $II_{13}$, we have
\[
\abs{II_{13}}\le c\kappa_0\eta \iint_{\Rpl}\abs{A_{||}^*-(A^{*a}_{||})_{2Q}}\abs{\nabla_x(\Psi^2)}\abs{\nabla_x u}t\, dxdt,
\]
which is bounded by $\sigma J+\tilde{c}\abs{Q}$ by the same reasoning for the term $J_{11}$.

For $II_{11}$, observe first that 
\[
II_{11}=\iint_{\Rpl}A_{||}^*\nabla_x(\theta_{\eta t}u\Psi^2)\cdot\nabla_x u dxdt
=\iint_{\Rpl}A_{||}\nabla_x u \cdot\nabla_x(\theta_{\eta t}u\Psi^2)dxdt.
\]

Taking $\theta_{\eta t}u\Psi^2$ as a test function (this is admissible due to the smoothness assumption) in the equation $L_0u=0$ in $\Rpl$, one gets
\begin{align*}
    0&=\iint_{\Rpl}A_0\nabla u\cdot\nabla(\theta_{\eta t}u\Psi^2)dxdt\\
    &=\iint_{\Rpl}A_{||}\nabla_x u\cdot\nabla_x(\theta_{\eta t}u\Psi^2)
    +\iint_{\Rpl}\br{\bd b-(\bd b^a)_{2Q}}\cdot\nabla_x(\theta_{\eta t}u\Psi^2)\Dt u\\
    &\quad +\iint_{\Rpl}\br{\bd c-(\bd c^a)_{2Q}}\cdot\nabla_x u\,\Dt(\theta_{\eta t}u\Psi^2)
    +\iint_{\Rpl}d\,\Dt u\,\Dt(\theta_{\eta t}u\Psi^2).
\end{align*}
So we have
\begin{align*}
     II_{11}&=-\iint_{\Rpl}\br{\bd b-(\bd b^a)_{2Q}}\cdot\nabla_x(\theta_{\eta t}u\Psi^2)\Dt u\\
    &\quad-\iint_{\Rpl}\br{\bd c-(\bd c^a)_{2Q}}\cdot\nabla_x u\,\Dt(\theta_{\eta t}u\Psi^2)
    -\iint_{\Rpl}d\,\Dt u\,\Dt(\theta_{\eta t}u\Psi^2)\\
    &=: II_{111}+II_{112}+II_{113}.
\end{align*}

We treat $II_{113}$ first. Write
\begin{multline*}
    II_{113}=-\iint_{\Rpl}d\,\Dt u\,\Dt\theta_{\eta t}(u\Psi^2)-\iint_{\Rpl}d\,\Dt u\,\Dt u(\theta_{\eta t}\Psi^2)\\
    -\iint_{\Rpl}d\,\Dt u\,\Dt(\Psi^2)\theta_{\eta t}u
    =: II_{1131}+II_{1132}+II_{1133}.
\end{multline*}
Note that $\Dt\theta_{\eta t}=-\Dt\mP^*_{\eta t}\vp$. So $II_{1131}=\iint_{\Rpl}d\,\Dt u\,\Dt\mP^*_{\eta t}\vp(u\Psi^2)$. We first use Cauchy-Schwartz and then Proposition \ref{Lp_G1Prop} to get
\begin{align*}
    \abs{II_{1131}}&\le c\br{\iint_{\Rpl}\abs{\Dt u}^2\Psi^2t\,dxdt}^{1/2}\br{\iint_{\Rpl}\abs{\Dt\mP^*_{\eta t}\vp}^2\frac{dt}{t}dx}^{1/2}\\
    &\le \tilde cJ^{1/2}\norm{\nabla\vp}_{L^2(\Rn)}\le \sigma J+\tilde{c}\abs{Q}.
\end{align*}
For $II_{1132}$, we use \eqref{theta_sawtooth1 suppPsi} to get
\[
\abs{II_{1132}}\le c\kappa_0\eta\iint_{\Rpl}\abs{\nabla u}^2\Psi^2t\,dxdt\le c\eta J.
\]
By \eqref{theta_sawtooth1 suppPsi}, Young's inequality and Lemma \ref{cutoff est lem}, 
\begin{align*}
    \abs{II_{1133}}\le c\kappa_0\eta\iint_{\Rpl}\abs{\Dt u}\abs{\Dt\Psi}\Psi tdxdt\le \sigma J+\tilde{c}\abs{Q}.
\end{align*}

We now treat $II_{112}$. Write
\begin{align*}
    II_{112}&=-\iint_{\Rpl}\br{\bd c-(\bd c^a)_{2Q}}\cdot\nabla_x u\,\Dt u\,(\theta_{\eta t}\Psi^2)dxdt\\
    &\quad+\iint_{\Rpl}\br{\bd c-(\bd c^a)_{2Q}}\cdot\nabla_x u\,\Dt \mP^*_{\eta t}\vp(u\Psi^2)dxdt\\
    &\quad-2\iint_{\Rpl}\br{\bd c-(\bd c^a)_{2Q}}\cdot\nabla_x u\,\Dt\Psi(\theta_{\eta t}u\Psi)dxdt\\
    &=: II_{1121}+II_{1122}+II_{1123}.
\end{align*}
For $II_{1122}$, we only focus on the anti-symmetric part, namely, the integral involving $\bd c^a-(\bd c^a)_{2Q}$ (denoted by $I_{1122}^a$), for the integral involving $\bd c^s$ is easier to estimate. We have
\begin{multline*}
    \abs{II_{1122}^a}\le\br{\int_{2Q}\abs{\bd c-(\bd c^a)_{2Q}}^{\alpha'}}^{\frac1{\alpha'}}
    \br{\int_{2Q}\Big(\int_0^{4l(Q)}\abs{\nabla_x u}\abs{\Dt\mP^*_{\eta t}\vp}\Psi^2dt\Big)^{\alpha}dx}^{\frac1{\alpha}}\\
    \le c\abs{Q}^{1/{\alpha'}}\br{\int_{\Rn}\br{\int_0^{\infty}\abs{\nabla u}\Psi^2t\,dt}^{\alpha/2}\br{\int_0^{\infty}\abs{\Dt\mP^*_{\eta t}\vp}^2\frac{dt}{t}}^{\alpha/2}dx}^{1/{\alpha}}\\
    \le \abs{Q}^{1/{\alpha'}}J^{1/2}\br{\int_{\Rn}\br{\int_0^{\infty}\abs{\Dt\mP^*_{\eta t}\vp}^2\frac{dt}{t}}^{\frac{\alpha}{2-\alpha}}dx}^{\frac{2-\alpha}{2\alpha}}.
\end{multline*}
Choosing $\alpha$ so that $\frac{\alpha}{2-\alpha}=\frac{2+\epsilon_0}{2}$ and applying Proposition \ref{Lp_G1Prop} with $p=\frac{2\alpha}{2-\alpha}=2+\epsilon_0$, as well as \eqref{nabla vp_Lp}, we get
\begin{align*}
    \abs{II_{1122}^a}\le c\eta\abs{Q}^{1/{\alpha'}} J^{1/2}\norm{\nabla\vp}_{L^{\frac{2\alpha}{2-\alpha}}(\Rn)}\le c\eta J^{1/2}\abs{Q}^{1/2}\le\sigma J+\tilde{c}\abs{Q}.
\end{align*}
Using the bound \eqref{theta_sawtooth1 suppPsi}, $II_{1123}$ can be estimated like $II_{13}$, and hence bounded by $\sigma J+\tilde{c}\abs{Q}$. 

For $II_{1121}$, we write
\begin{align}
    II_{1121}&=-\iint_{\Rpl}\bd c^s\cdot\nabla_x u\,\Dt u\,\theta_{\eta t}\Psi^2
   -\iint_{\Rpl}\br{\bd c^a-(\bd c^a)_{2Q}}\cdot\nabla_x u\,\Dt u\,\theta_{\eta t}\Psi^2\nonumber\\
   &=-\iint_{\Rpl}\bd c^s\cdot\nabla_x u\,\Dt u\,\theta_{\eta t}\Psi^2
   +\iint_{\Rpl}\br{\bd b^a-(\bd b^a)_{2Q}}\cdot\nabla_x u\,\Dt u\,\theta_{\eta t}\Psi^2.\label{II 1121}
\end{align}
The first term in \eqref{II 1121} can be estimated as $II_{1132}$. We leave the second term aside for now.

We write $II_{111}$ as follows
\begin{align*}
    II_{111}&=-\iint_{\Rpl}\br{\bd b-(\bd b^a)_{2Q}}\cdot\nabla_x\theta_{\eta t}(u\Psi^2\Dt u)dxdt\\
    &\quad-\iint_{\Rpl}\br{\bd b-(\bd b^a)_{2Q}}\cdot\nabla_xu\,(\theta_{\eta t}\Psi^2\Dt u)dxdt\\
    &\quad-2\iint_{\Rpl}\br{\bd b-(\bd b^a)_{2Q}}\cdot\nabla_x\Psi(\theta_{\eta t}u\Psi\Dt u)dxdt\\
    &=: II_{1111}+II_{1112}+II_{1113}.
\end{align*}
The term $\abs{II_{1113}}$ can be estimated like $II_{1123}$, and hence bounded by $\sigma J+\tilde{c}\abs{Q}$. For $II_{1112}$, we write
\begin{equation}\label{II 1112}
    II_{1112}=-\iint_{\Rpl}\bd b^s\cdot\nabla_xu\,(\theta_{\eta t}\Psi^2\Dt u)
    -\iint_{\Rpl}\br{\bd b^a-(\bd b^a)_{2Q}}\cdot\nabla_xu\,(\theta_{\eta t}\Psi^2\Dt u)
\end{equation}
The first term can be estimated as the first term in \eqref{II 1121}. And the second term in \eqref{II 1112} cancels the second term in \eqref{II 1121}. 

It remains to estimate $II_{1111}$. Integration by parts in $t$ gives
\begin{multline*}
    2II_{1111}=\iint_{\Rpl}\br{\bd b-(\bd b^a)_{2Q}}\cdot\Dt(\nabla_x\theta_{\eta t})u^2\Psi^2\\
    +\iint_{\Rpl}\br{\bd b-(\bd b^a)_{2Q}}\cdot\nabla_x\theta_{\eta t}\Dt(\Psi^2)u^2\\
    =\iint_{\Rpl}\br{\bd b-(\bd b^a)_{2Q}}\cdot\nabla_x(\Dt\mP^*_{\eta t}\vp\Psi^2 u^2)\\
    -\iint_{\Rpl}\br{\bd b-(\bd b^a)_{2Q}}\cdot\nabla_x(\Psi^2u^2)\Dt\mP^*_{\eta t}\vp\\
    +\iint_{\Rpl}\br{\bd b-(\bd b^a)_{2Q}}\cdot\nabla_x\theta_{\eta t}\Dt(\Psi^2)u^2
    =: III_1+III_2+III_3.
\end{multline*}
For $III_2$, we write
\begin{multline*}
    III_2=-2\iint_{\Rpl}\br{\bd b-(\bd b^a)_{2Q}}\cdot\nabla_x u\Dt\mP^*_{\eta t}\vp(u\Psi^2)\\
    -2\iint_{\Rpl}\br{\bd b-(\bd b^a)_{2Q}}\cdot\nabla_x \Psi\Dt\mP^*_{\eta t}\vp(u^2\Psi).
\end{multline*}
The first term on the right-hand side can be estimated as $II_{1122}$. The second term can be estimated using $\norm{u}_{L^\infty}\le1$, H\"older's inequality, Lemma \ref{cutoff est lem} and \eqref{Lp_G1Dt}. Together, one obtains $\abs{III_2}\le\sigma J+\tilde{c}\abs{Q}$. Finally, $III_3$ can be estimated as $II_2$, and thus $\abs{III_3}\le\tilde{c}\abs{Q}$.

For $III_1$, note that it is similar to $J_{211}$ except that it has an extra $\Dt\mP^*_{\eta t}\vp$. It turns out that this term will do our favor. We proceed like $J_{211}$ by recalling that $\divg_x\br{\bd b-(\bd b^a)_{2Q}}=\divg_x A_{||}\nabla_x\wvp=-L_{||}\wvp$ (see \eqref{bHodge}). So we have
\begin{align*}
    III_1=\iint_{\Rpl}A_{||}\nabla_x\wvp\cdot\nabla_x(\Dt\mP^*_{\eta t}\vp\Psi^2 u^2).
\end{align*}
Writing $\wvp=\wt{\theta}_{\eta t}+\mP_{\eta t}\wvp$, we get
\begin{multline*}
    III_1
    =\iint_{\Rpl}A_{||}\nabla_x\wt{\theta}_{\eta t}\cdot\nabla_x\br{\Dt\mP^*_{\eta t}\vp(\Psi^2 u^2)}\\
    +\iint_{\Rpl}A_{||}\nabla_x\mP_{\eta t}\wvp\cdot\nabla_x\br{\Dt\mP^*_{\eta t}\vp(\Psi^2 u^2)}\\
    =\iint_{\Rpl}\br{A_{||}-(A_{||}^a)_{2Q}}\nabla_x\wt{\theta}_{\eta t}\cdot\nabla_x\br{\Dt\mP^*_{\eta t}\vp(\Psi^2 u^2)}\\
    +\iint_{\Rpl}L_{||}\mP_{\eta t}\wvp\,\Dt\mP^*_{\eta t}\vp(\Psi^2 u^2)
    =: III_{11}+III_{12},
\end{multline*}
where in the second equality we have used the smoothness assumption to obtain
\[
\iint_{\Rpl}(A_{||}^a)_{2Q}\nabla_x\wt{\theta}_{\eta t}\cdot\nabla_x\br{\Dt\mP^*_{\eta t}\vp(\Psi^2 u^2)}=0.
\]

For $III_{12}$, the Cauchy-Schwartz inequality gives
\begin{align*}
    \abs{III_{12}}\le c\br{\iint_{\Rpl}t\abs{L_{||}\mP_{\eta t}\wvp}^2dxdt}^{1/2}\br{\iint_{\Rpl}\abs{\Dt\mP^*_{\eta t}\vp}^2\frac{dxdt}{t}}^{1/2}.
\end{align*}
So by Proposition \ref{Lp_G1Prop}, $\abs{III_{12}}\le\tilde{c}\abs{Q}$.

For $III_{11}$, we write
\begin{align*}
    III_{11}&=\iint_{\Rpl}\br{A_{||}-(A_{||}^a)_{2Q}}\nabla_x\wt{\theta}_{\eta t}\cdot\nabla_x(u^2)\Dt\mP^*_{\eta t}\vp\,\Psi^2\\
    &\quad+\iint_{\Rpl}\br{A_{||}-(A_{||}^a)_{2Q}}\nabla_x\wt{\theta}_{\eta t}\cdot\nabla_x(\Psi^2) \Dt\mP^*_{\eta t}\vp\, u^2\\
    &\quad+\iint_{\Rpl}\br{A_{||}-(A_{||}^a)_{2Q}}\nabla_x\wt{\theta}_{\eta t}\cdot\nabla_x \Dt\mP^*_{\eta t}\vp (\Psi^2u^2)\\
    &=: III_{111}+III_{112}+III_{113}.
\end{align*}

Since $N^{\eta}\br{\Dt\mP^*_{\eta t}\vp}(x)\le c\kappa_0\eta$ for any $x\in F$ by the construction of $F$, $\abs{\Dt\mP^*_{\eta t}\vp}\le c\kappa_0\eta$ on the support of $\Psi$. Therefore, $III_{112}$ can be estimated like the term $II_2$ and thus $\abs{III_{112}}\le\tilde{c}\abs{Q}$. 

For $III_{113}$, note that Proposition \ref{Lp_G2t prop} implies
\begin{equation}\label{III 113}
    \iint_{\Rpl}\abs{t^2L_{||}^*\Dt\mP^*_{\eta t}\vp}^2\frac{dxdt}{t}\le c\eta^{-2}\abs{Q}.
\end{equation}
We write
\begin{multline*}
    III_{113}=\iint_{\Rpl}\nabla_x(\wt{\theta}_{\eta t} u^2\Psi^2)\cdot A^*_{||}\nabla_x \Dt\mP^*_{\eta t}\vp\\
    -\iint_{\Rpl}\wt{\theta}_{\eta t}\nabla_x(u^2\Psi^2)\cdot\br{A^*_{||}-(A_{||}^{a*})_{2Q}}\nabla_x \Dt\mP^*_{\eta t}\vp\\
    =\iint_{\Rpl}\wt{\theta}_{\eta t} u^2\Psi^2 L^*_{||}\nabla_x \Dt\mP^*_{\eta t}\vp\\
    -\iint_{\Rpl}\wt{\theta}_{\eta t}\nabla_x(u^2\Psi^2)\cdot\br{A^*_{||}-(A_{||}^{a*})_{2Q}}\nabla_x \Dt\mP^*_{\eta t}\vp
    =: III_{1131}+III_{1132}.
\end{multline*}
By the Cauchy-Schwartz inequality, Lemma \ref{theta_integral lem} and \eqref{III 113},
\begin{align*}
    \abs{III_{1131}}&\le c\br{\iint_{\Rpl}\abs{\wt{\theta}_{\eta t}}^2\frac{dxdt}{t^3}}^{1/2}\br{\iint_{\Rpl}\abs{t^2L_{||}^*\Dt\mP^*_{\eta t}\vp}^2\frac{dxdt}{t}}^{1/2}\\
    &\le c\abs{Q}.
\end{align*}
By \eqref{theta_sawtooth1 suppPsi}, $\abs{III_{1132}}$ is bounded by  
\[
c\kappa_0\eta\iint_{\Rpl}\abs{A^*_{||}-(A_{||}^{a*})_{2Q}}\abs{\nabla_x(u^2\Psi^2)}\abs{\nabla_x \Dt\mP^*_{\eta t}\vp}t\,dxdt,
\]
which is bounded by $\sigma J+\tilde{c}\abs{Q}$ using the same method of estimating $I_{12}$.

Now it remains to estimate $III_{111}$. Note that the integration is over the support of $\Psi$ instead of support of $\nabla\Psi$, so we cannot use the same method as estimating $II_2$. Like before, we only deal with the term involving $A_{||}^a-(A_{||}^a)_{2Q}$, as the term with the symmetric matrix $A_{||}^s$ is easier to estimate. We have
\begin{multline}\label{III 11}
    \abs{III^a_{11}}=\abs{\iint_{\Rpl}\br{A^a_{||}-(A_{||}^a)_{2Q}}\nabla_x\wt{\theta}_{\eta t}\cdot\nabla_x(u^2)\Dt\mP^*_{\eta t}\vp\,\Psi^2}\\
    \le c\abs{Q}^{\frac1{\alpha'}}\br{\int_{2Q}\br{\int_0^{4l(Q)}\abs{\nabla_x\wt{\theta}_{\eta t}}\abs{\nabla_xu}\Psi^2\abs{\Dt\mP^*_{\eta t}\vp}dt}^{\alpha}dx}^{\frac1{\alpha}}\\
    \le c\abs{Q}^{\frac1{\alpha'}}\br{\sigma J+\tilde c\int_{\Rn}\br{\int_0^{\infty}\abs{\nabla_x\wt{\theta}_{\eta t}}^2\abs{\Dt\mP^*_{\eta t}\vp}^2\1_{\supp\Psi}\frac{dt}{t}}^{\frac{\alpha}{2-\alpha}}dx}^{\frac1{\alpha}}.
\end{multline}
We write
\begin{align*}
    &\int_{\Rn}\br{\int_0^{\infty}\abs{\nabla_x\wt{\theta}_{\eta t}}^2\abs{\Dt\mP^*_{\eta t}\vp}^2\1_{\supp\Psi}\frac{dt}{t}}^{\frac{\alpha}{2-\alpha}}dx\\
    &\qquad=\sup_{\substack{\xi\in\mathscr{S}(\Rn)\\\norm{\xi}^{\frac{\alpha}{2\alpha-2}}\le1}}
    \abs{\iint_{\Rpl}\abs{\nabla_x\wt{\theta}_{\eta t}}^2\abs{\Dt\mP^*_{\eta t}\vp}^2\xi(x)\1_{\supp\Psi}\frac{dxdt}{t}}^{\frac{\alpha}{2-\alpha}}.
\end{align*}
As before, let $\mathcal{D}^{\eta}_k$ be the grid of dyadic cubes such that \eqref{dyadicgrid} holds. Then 
\begin{align}
    &\iint_{\Rpl}\abs{\nabla_x\wt{\theta}_{\eta t}}^2\abs{\Dt\mP^*_{\eta t}\vp}^2\xi(x)\1_{\supp\Psi}\frac{dxdt}{t}\nonumber\\
    &\qquad=\sum_k\sum_{Q'\in\mathcal{D}^{\eta}_k}\int_{Q'}\int_{2^{-k}}^{2^{-k+1}}\abs{\nabla_x\wt{\theta}_{\eta t}}^2\abs{\Dt\mP^*_{\eta t}\vp}^2\xi(x)\1_{\supp\Psi}\frac{dtdx}{t}.\label{III 11'}
\end{align}

By Corollary \ref{sup dtw cor}, we bound \eqref{III 11'} by
\begin{multline}\label{III 11''}
    c\eta\sum_k\sum_{Q'\in\mathcal{D}^{\eta}_k}\br{\int_{2Q'}\int_{2^{-k-1}}^{2^{-k+1}}\abs{\Dt\mP^*_{\eta t}\vp}^2\frac{dydt}{t}}\\
    \times\int_{2^{-k}}^{2^{-k+1}}\frac{1}{\abs{Q'}}\int_{Q'}\abs{\nabla_x\wt{\theta}_{\eta t}}^2\abs{\xi(x)}\1_{\supp\Psi}\frac{dtdx}{t}. 
\end{multline}
We now estimate the integral  in the second line of \eqref{III 11''}. 

Let $r=1+\epsilon$ with $\epsilon>0$ sufficiently small. We use H\"older's inequality, then definition of $\wt{\theta}_{\eta t}$, and Corollary \ref{nablaw RHp cor} as well as the reverse H\"older estimates for $\nabla\wvp$, to get
\begin{align*}
   &\int_{2^{-k}}^{2^{-k+1}}\frac{1}{\abs{Q'}}\int_{Q'}\abs{\nabla_x\wt{\theta}_{\eta t}}^2\abs{\xi(x)}\1_{\supp\Psi}\frac{dtdx}{t}\\
   &\le\int_{2^{-k}}^{2^{-k+1}}\br{\fint_{Q'}\abs{\nabla_x\wt{\theta}_{\eta t}}^{2r'}dx}^{1/{r'}}\br{\fint_{Q'}\abs{\xi(x)}^rdx}^{1/r}\1_{\supp\Psi}\frac{dt}{t}\\
   &\le\br{\fint_{Q'}\abs{\xi(x)}^rdx}^{1/r}\\
   &\qquad\times\int_{2^{-k}}^{2^{-k+1}}\Big\{\br{\fint_{Q'}\abs{\nabla_x\mP_{\eta t}\wvp}^{2r'}dx}^{1/{r'}}+\br{\fint_{Q'}\abs{\nabla\wvp}^{2r'}dx}^{1/{r'}}\Big\}\1_{\supp\Psi}\frac{dt}{t}\\
   &\le c\br{\fint_{Q'}\abs{\xi(x)}^rdx}^{1/r}
   \int_{\frac{1}{2^k}}^{\frac1{2^{k-1}}}\Big\{\fint_{2Q'}\abs{\nabla_x\mP_{\eta t}\wvp}^{2}dx+\eta^{-2}\br{\fint_{2Q'}\abs{\Dt\mP_{\eta t}\wvp}^{2r'}}^{\frac1{r'}}\\
   &\qquad\qquad+\fint_{2Q'}\abs{\nabla\wvp}^2\Big\}\1_{\supp\Psi}\frac{dt}{t}\\
   &\le c\br{\fint_{Q'}\abs{\xi(x)}^rdx}^{1/r}
   \int_{\frac{1}{2^k}}^{\frac1{2^{k-1}}}\Big\{\fint_{B(x_0,\eta 2^{-k})}\abs{\nabla_x\mP_{\eta t}\wvp}^{2}dx\\
   &\qquad+\eta^{-2}\br{\fint_{B(x_0,\eta 2^{-k})}\abs{\Dt\mP_{\eta t}\wvp}^{2r'}dx}^{1/{r'}}
   +\fint_{B(x_0,\eta 2^{-k})}\abs{\nabla\wvp}^2dx\Big\}\frac{\1_{\supp\Psi}dt}{t},
\end{align*}
where in the last inequality we have used \eqref{Q'subset}, with $x_0\in F$. Therefore, we can bound this by
\begin{multline*}
     c\br{\fint_{Q'}\abs{\xi(x)}^rdx}^{1/r}
   \int_{2^{-k}}^{2^{-k+1}}\wt{N}^{\eta}(\nabla_x\mP_{\eta t}\wvp)(x_0)^2
   +\eta^{-2}N^{\eta}(\Dt\mP_{\eta t}\wvp)(x_0)^2\\
   +M\br{\abs{\nabla\wvp}^2}(x_0)\,\frac{dt}{t}
   \le c\kappa_0^2\br{\fint_{Q'}\abs{\xi(x)}^rdx}^{1/r}.
\end{multline*}
So \eqref{III 11''} is bounded by
\begin{align*}
    & c\eta\sum_k\sum_{Q'\in\mathcal{D}^{\eta}_k}\br{\int_{2Q'}\int_{2^{-k-1}}^{2^{-k+1}}\abs{\Dt\mP^*_{\eta t}\vp}^2\frac{dydt}{t}}\br{\fint_{Q'}\abs{\xi(x)}^rdx}^{1/r}\\
    &\le c\eta\sum_k\sum_{Q'\in\mathcal{D}^{\eta}_k}\int_{2Q'}\br{M(\abs{\xi}^r)}^{1/r}(y)\int_{2^{-k-1}}^{2^{-k+1}}\abs{\Dt\mP^*_{\eta t}\vp(y)}^2\frac{dt}{t}dy\\
    &\le c\eta\int_{\Rn}M(\abs{\xi}^r)(y)^{1/r}\int_{0}^{\infty}\abs{\Dt\mP^*_{\eta t}\vp(y)}^2\frac{dt}{t}dy\\
    &\le c\eta\br{\int_{\Rn}M(\abs{\xi}^r)(y)^{q/r}dy}^{1/q}
    \br{\int_{\Rn}\br{\int_0^{\infty}\abs{\Dt\mP^*_{\eta t}\vp(y)}^2\frac{dt}{t}}^{q'}}^{1/{q'}}.
\end{align*}
Choosing $q=\frac{\alpha}{2\alpha-2}$, the above is bounded by
\begin{multline*}
     c\eta\br{\int_{\Rn}\abs{\xi}^{\frac{\alpha}{2\alpha-2}}}^{\frac{2\alpha-2}{\alpha}}
    \br{\int_{\Rn}\br{\int_0^{\infty}\abs{\Dt\mP^*_{\eta t}\vp(y)}^2\frac{dt}{t}}^{\frac{\alpha}{2-\alpha}}dy}^{\frac{2-\alpha}{\alpha}}\\
    \le\tilde{c}\norm{\xi}_{L^{\frac{\alpha}{2\alpha-2}}}\abs{Q}^{\frac{2-\alpha}{\alpha}},
\end{multline*}
where in the last step we have used Proposition \ref{Lp_G1Prop}.
Combining these estimates with \eqref{III 11}, we obtain
\[
\abs{III_{111}^a}\le c\abs{Q}^{1/{\alpha'}}(\sigma J+\tilde{c}\abs{Q})^{1/{\alpha}}\le \frac{\sigma}{2}J+\tilde{c}\abs{Q}.
\]
This finishes the proof of Lemma \ref{main lem}.

\section{Proof of Uniqueness and Theorem \ref{Fatou thm}}\label{uniqueness sec}
In this section, we prove the uniqueness part in the statement of Theorem \ref{main thm}. One can prove the uniqueness of $L^p$ Dirichlet problem in bounded domains as in \cite{kenig1994harmonic} Theorem 1.7.7. But that argument can not be modified to work for unbounded domains. We present here a different and simpler proof that works in a rather general setting.  

Recall that we have proved that for any cube $Q_0\subset\Rn$, $\omega^{X_{Q_0}}\in A_{\infty}(Q_0)$, which implies that there is some $q\in (1,\infty)$ such that the Radon-Nikodym derivative  $k(X_{Q_0},\cdot)$ satifies the reverse H\"older inequality \eqref{RHq}.  
We now show that we have the following non-tangential maximal function estimate:
\begin{lem}\label{uniquness mainlem}
Let $p\ge q'$, where $q$ is the exponent in the reverse H\"older inequality \eqref{RHq}. If $f\in L^p(\Rn,d\mu)$ and $u(X)=\int_{\Rn}f(y)k(X,y)d\mu(y)$, then 
\begin{equation}\label{ntmaxLq}
    \norm{Nu}_{L^p(\Rn,d\mu)}\lesssim\norm{f}_{L^p(\Rn,d\mu)}.
\end{equation} Moreover, $u$ converges non-tangentially $\mu$- a.e. to $f$.
\end{lem}
\begin{proof}
We first note that \eqref{ntmaxLq} may be obtained as in the proof of  Lemma 5.32 in \cite{hofmann2018carleson}. Indeed, the argument in \cite{hofmann2018carleson} relies only on H\"older continuity of solutions, Harnack principle, and comparison
principle. The coefficients (in $\BMO$) do not affect the argument since the equation is not used explicitly. 
It therefore suffices to show that $u$ converges non-tangentially $\mu$-a.e. to $f$.

 For any $\epsilon>0$, choose $f_{\epsilon}\in C_0(\Rn)$ such that $\norm{f-f_\epsilon}_{L^p(\Rn,\mu)}<\epsilon$. Define $u_\epsilon(X)=\int_{\Rn}f_\epsilon(y)k(X,y)d\mu(y)$. Then $u_\epsilon\in C(\overline{\Rpl})$ and $u_\epsilon=f_\epsilon$ on $\Rn$. 
  We note that the latter fact may be gleaned from the analogous fact on bounded domains (see\cite{li2019boundary}), the construction at the beginning of Section 2 (applied with $u=u_\epsilon$), and an equicontinuity argument using \cite{li2019boundary} Lemma 3.9, and Lemma 4.5.
 So \[\lim_{\Gamma(x)\ni(y,t)\to(x,0)}u_\epsilon(y,t)=f_\epsilon(x)\qquad\forall\,x\in \Rn.
\]
Since we have the non-tangential convergence for a dense class, the non-tangential convergence of $u$ follows from \eqref{ntmaxLq} and a standard argument. In fact, we have

\[\limsup _{\Gamma(x) \ni(y, t) \rightarrow(x, 0)}|u(y, t)-f(x)| \leq\left|N\left(u-u_{\epsilon}\right)(x)\right|+\left|\left(f-f_{\epsilon}\right)(x)\right| \qquad\forall\,x\in\Rn.\]
For any $\lambda>0$, we apply Chebyshev's inequality and \eqref{ntmaxLq} to get
\begin{align*}
   &\mu\left(\left\{x \in \mathbb{R}^{n} : \limsup _{\Gamma(x) \ni(y, t) \rightarrow(x, 0)}|u(y, t)-f(x)|>\lambda\right\}\right)\\
    &\le \mu\left(\left\{x \in \mathbb{R}^{n} : N\left(u-u_{\epsilon}\right)(x)>\lambda / 2\right\}\right)+\mu\left(\left\{x \in \mathbb{R}^{n} :\left|\left(f-f_{\epsilon}\right)(x)\right|>\lambda / 2\right\}\right)\\
    &\lesssim \lambda^{-p}\br{\norm{N(u-u_\epsilon)}^p_{L^p(\Rn,d\mu)}+\norm{f-f_\epsilon}_{L^p(\Rn,d\mu)}^p}\\
    &\lesssim\lambda^{-p}\norm{f-f_\epsilon}_{L^p(\Rn,d\mu)}^p\lesssim\epsilon\lambda^{-p}.
\end{align*}
Since $\epsilon>0$ is arbitrary, it shows that $\lim _{\Gamma(x) \ni(y, t) \rightarrow(x, 0)} u(y, t)=f(x)$ for $\mu$- a.e. $x\in\Rn$. 
\end{proof}

The $L^p$ boundedness of the non-tangential maximal function implies certain decay properties. To be precise, we have the following
\begin{lem}\label{decay lem}
Let $u(x,t)$ be a function in $\Rpl$. If there exists some constant $C$ such that $\norm{Nu}_{L^p(\Rn)}<C$ for some $p>0$, then $u$ satisfies the following properties:
\begin{enumerate}
    \item\label{decay_t} \(\abs{u(x,t)}<C't^{-\frac{n}{p}}\)  for all $(x,t)\in\Rpl$, where the constant $C'$ only depends on $n$ and $C$.
    \item\label{decay_x} For any $\epsilon>0$, any $\delta>0$, there exists some $R_0=R_0(u,\epsilon,\delta)>1$ such that for all $\abs{x}\ge R_0$ and $t\ge \delta$, we have $\abs{u(x,t)}<\epsilon$.
\end{enumerate}
\end{lem}
\begin{proof}
To see \eqref{decay_t}, we observe for any fixed $(x,t)\in \Rpl$, for all $y\in \Delta(x,t)$, $(x,t)\in \Gamma(y)$. So we have
\begin{align*}
    \abs{u(x,t)}^p\le \frac{1}{\abs{\Delta(x,t)}}\int_{\Delta(x,t)}Nu(y)^pd\mu(y)\le C_n C^pt^{-n}.
\end{align*}
We prove \eqref{decay_x} by contradiction. If this is not true, then there exist some $\epsilon>0$ and $\delta>0$ such that for any $k\in\mathbb N$, we can find $\abs{x_k}\ge k$ and $t_k\ge\delta$, for which $\abs{u(x_k,t_k)}\ge\epsilon$. Since $t_k\ge\delta$, $(x_k,t_k)\in\Gamma(y)$ for all $y\in\Delta(x_k,\delta)$. This implies that 
\[
Nu(y)\ge u(x_k,t_k)\ge\epsilon \qquad\forall\, y\in\Delta(x_k,\delta).
\]
Choose a subsequence $x_{k_j}$ so that the collection of surface balls $\{\Delta(x_{k_j},\delta)\}$ is pairwise disjoint. Then
\[
C^p>\int_{\Rn}\abs{Nu(y)}^pd\mu(y)\ge\sum_{j=0}^{\infty}\int_{\Delta(x_{k_j},\delta)}\epsilon^pd\mu(y)=C_n\sum_{j=0}^{\infty}\epsilon^p\delta^n=\infty,
\]
which yields a contradiction. 
\end{proof}

We now prove the uniqueness of the $L^p$ Dirichlet problem.

\noindent {\it Proof of uniqueness.}  
Fix $q$ so that $k(X,\cdot)\in L^q(\Rn)$ for all $X\in \Real^{n+1}_+$ as in \eqref{RHq}, and let $p=\frac{q}{q-1}$. We show that if $u$ is a solution of $(D)_p$, that is,
\begin{equation*}
    \begin{cases}
    Lu=0 \quad\text{in } \Rpl,\\
    u\to f\in L^p(\Rn,d\mu) \text{ non-tangentially } \mu \text{-a.e. on } \Rn,\\
    Nu\in L^p(\Rn, d\mu),
  \end{cases}
\end{equation*} 
then 
\begin{equation}\label{u form}
    u(X)=\int_{\Rn}g(y)k(X,y)d\mu(y) \qquad\text{for some } g\in L^p(\Rn,d\mu).
\end{equation}
Then by Lemma \ref{uniquness mainlem}, $u$ converges non-tangentially $\mu$-a.e. to $g$. This implies that 
\(
u(X)=\int_{\Rn}f(y)k(X,y)d\mu(y)\), which proves that the solution is unique. We now show \eqref{u form}.

For any $m\in\mathbb N$, set $f_m(x):= u(x,\frac{1}{m})$. Note that by the interior estimates for  weak solutions, $f_m$ is continuous on $\Rn$. Moreover,
\begin{equation}\label{eq fm_Lp}
    \norm{f_m}_{L^p(\Rn)}\le\sup_{t>0}\norm{u(\cdot,t)}_{L^p(\Rn)}\le\norm{Nu}_{L^p}<\infty.
\end{equation}
Since $Nu\in L^p(\Rn)$, we can apply Lemma \ref{decay lem} (2) and get 
\begin{equation}\label{eq fm_limit}
    \lim_{\abs{x}\to\infty}f_m(x)=0, \qquad \norm{f_m}_{L^\infty(\Rn)}<\infty.
\end{equation}
We define \[u_m(x,t):=\int_{\Rn}f_m(y)k\br{(x,t),y}d\mu(y), \text{ and }\,\delta_m(x,t):= u(x,t+\frac{1}{m})-u_m(x,t).\]
 Since $f_m$ is continuous on $\Rn$ and satisfies \eqref{eq fm_limit}, from the definition of elliptic measures it follows that
\begin{equation}\label{eq um_Linfty}
    \norm{u_m}_{L^\infty(\Real^{n+1}_+)}\le\norm{f_m}_{L^\infty(\Rn)}.
\end{equation}
Moreover, we claim that $u_m$ is a solution to the continuous Dirichlet problem, with
data $f_m$; in particular, $u_m(x,0)=f_m(x)$ for all $x\in \Rn$. To see this, for $R>0$ and large, let
$\Phi_R$ be a smooth
cut-off function defined on $\Rn$, identically $1$ in $\Delta(0,R)$, supported in $\Delta(0,2R)$,
with $0\leq \Phi_R \leq 1$. Set $f_{m,R} := f_m \Phi_R$, and let $u_{m,R}$ be the elliptic measure
solution with data $f_{m,R}$. Then $u_{m,R}(\cdot,0) = f_{m,R}$ continuously, since the data
belongs to $C_0(\Rn)$. In particular, $u_{m,R}(x,0) =f_m(x)$ for all $|x|<R$.
Given $\epsilon>0$, we note that by Lemma \ref{decay lem} (2), $|f_m(x)-f_{m,R}(x)| \leq \epsilon$, for
all $x\in \Rn$, provided that $R$ is large enough, hence also $|u_{m,R}(x,t) - u_m(x,t)| \leq \epsilon$,
since elliptic measure has total mass $1$.  The claim now follows. 
This means that
\begin{equation}\label{deltam_btm}
    \delta_m(x,0)=0 \qquad\text{for all } x\in\Rn.
\end{equation}
Notice that $\delta_m$ is a solution to $Lv=0$ in $\Rpl$, which vanishes continuously on $\{t\equiv 0\}$.  We claim that $\delta_m \equiv 0$ in $\Rpl$.  To prove this claim, we observe that by the maximum principle, it suffices to show that
\[
\lim_{|x|+t\to \infty}\big|u\big(x,t+\frac1m\big)\big| + |u_m(x,t)|=0.
\]
For $u\big(x,t+\frac1m\big)$, this follows immediately from Lemma \ref{decay lem} and our assumption that $Nu\in L^p(\Rn)$. To see that decay to $0$ holds for $u_m$,
we define $f_{m,R}, u_{m,R}$ as above.
Given $\epsilon>0$, fix $R$ so that $\norm{f_m-f_{m,R}}_{L^\infty(\Rn)} <\epsilon$,  hence also $\norm{u_m -u_{m,R}}_{L^\infty(\Rpl)}<\epsilon$.
By H\"older continuity at the boundary, we may choose $\delta>0$ small enough
that for $|x|>3R$, and $t<\delta$,
we have
\[ |u_{m,R}(x,t)| \lesssim \delta^\alpha ||f_m||_\infty <\epsilon,\]
and thus also $|u_m(x,t)| <2\epsilon$.  Moreover, with this value of $\delta$ now fixed, it
follows immediately from \eqref{eq fm_Lp}, the definition of $u_m$ and \eqref{ntmaxLq}, and Lemma \ref{decay lem}, that
\[
\lim_{|x|+t\to \infty}|u_m(x,t)|\1_{[\delta,\infty)}(t)=0.
\] We conclude
that $\delta_m\equiv 0$.  In turn, the latter is equivalent to
\begin{equation}\label{um form}
    u(x,t+\frac{1}{m})=\int_{\Rn}f_m(y)k((x,t),y)d\mu(y), \qquad\forall\, m\in\mathbb N.
\end{equation}
Since $\sup_m\norm{f_m}_{L^p(\Rn)}\le\norm{Nu}_{L^p}<\infty$, there is some $g\in L^p(\Rn,d\mu)$ and $\{f_{m'}\}$ such that $f_{m'}$ converges to $g$ weakly. Note that $k(X,\cdot)\in L^q(\Rn,d\mu)$ (see \eqref{k in Lq}), so by letting $m'$ go to infinity in \eqref{um form} we obtain \eqref{u form}.
 
\vskip 0.08 in
From the proof of uniqueness, one can see that we have actually proved the stronger result, Theorem \ref{Fatou thm}. In fact, we did not use $u\to f\in L^p(\Rn,d\mu)$ non-tangentially $\mu$-a.e. on $\Rn$ to obtain \eqref{u form}. Once we express $u$ as in \eqref{u form}, we apply Lemma \ref{uniquness mainlem} to conclude that the non-tangential limit of $u$ exists $\mu$- a.e. and is in $L^p(\Rn,d\mu)$.

\appendix

\section{Appendix: Weak solution of parabolic equations}\label{weak sol Appen}

\begin{lem}\label{Evanslem}
Suppose $u,v\in L^2\br{(0,T),W^{1,2}(\Rn)}$ with $\Dt u, \Dt v \in L^2\br{(0,T),\wt{W}^{-1,2}(\Rn)}$. Then
\begin{enumerate}[(i)]
    \item $u\in C\br{[0,T], L^2(\Rn)}$;
    \item The mapping $t\mapsto\norm{u(\cdot,t)}_{L^2(\Rn)}$ is absolutely continuous, with
    \[\frac{d}{dt}\norm{u(\cdot,t)}_{L^2(\Rn)}^2=2\Re\act{\Dt u(\cdot,t),u(\cdot,t)} \quad\text{for a.e. }t\in[0,T].\]
    As a consequence, 
    \[
    \frac{d}{dt}\br{u(\cdot,t),v(\cdot,t)}_{L^2(\Rn)}=\act{\Dt u(\cdot,t),v(\cdot,t)}+\overline{\langle{\Dt v(\cdot,t),u(\cdot,t)\rangle}}_{\wt{W}^{-1,2},W^{1,2}}\,\text{a.e.}
    \]
\end{enumerate}
\end{lem}
For the proof see, e.g., \cite{evans1998partial}, Section 5.9.2, Theorem 3.

Suppose that $A=A(x)=A^s(x)+A^a(x)$ is a real, $n\times n$ matrix, with $A^s$ being symmetric, elliptic with constant $\lambda_0>0$, $\norm{A^s}_{L^{\infty}(\Rn)}\le\lambda_0^{-1}$, and $A^a$ being anti-symmetric and $\norm{A^a}_{\BMO(\Rn)}\le\Lambda_0$. 

\begin{prop}\label{wellpose_IVP}
For any $u_0\in L^2(\Rn)$, the initial value problem
\begin{equation}\label{Appen_IVP}
   \begin{cases}
\Dt u-\divg(A\nabla u)=0 \quad\text{in }\Rn\times(0,\infty),\\
u(x,0)=u_0(x),
\end{cases} 
\end{equation}
has a unique weak solution $u(x,t)=e^{-tL}(u_0)(x)$. Here, $\divg=\divg_x$ and $\nabla=\nabla_x$.
\end{prop}

\begin{proof}
\underline{Existence.}

Since the domain of $L$ (denoted by $D(L)$) is dense in $W^{1,2}(\Rn)$, and thus dense in $L^2(\Rn)$, we can find a sequence $\set{u_{0,\epsilon}}\subset D(L)$ such that $u_{0,\epsilon}$ converges to $u_0$ in $L^2(\Rn)$. Denote $u_\epsilon(x,t):=e^{-tL}(u_{0,\epsilon})(x)$. Then by semigroup theory,
\begin{equation}\label{Appen_sol}
   \Dt u_\epsilon+Lu_\epsilon=0 \quad\text{in } L^2(\Rn) \quad\forall\, t\ge0.
\end{equation}
 For any $0<\tau<T$, and any $\vp\in L^2\br{(0,T),W^{1,2}(\Rn)}$, with $\Dt\vp\in L^2\br{(0,T),\wt{W}^{-1,2}(\Rn)}$, \eqref{Appen_sol} implies
\begin{equation}\label{eq Appen1}
 \int_{\tau}^{T}\br{\Dt u_\epsilon,\vp}_{L^2}dt+\int_{\tau}^T\br{Lu_\epsilon,\vp}_{L^2}dt=0.   
\end{equation}
Since $\Dt u_\epsilon\in L^2_{\loc}\br{(0,\infty),L^2(\Rn)}$ (see \cite{HLMPLp} Theorem 4.9), and by Lemma \ref{Evanslem} (ii), \eqref{eq Appen1} can be written as 
\begin{multline}\label{Appen_ueps}
    \int_{\Rn}u_\epsilon(x,T)\overline{\vp(x,T)}dx+\int_{\tau}^T\int_{\Rn}A\nabla u_\epsilon\cdot\overline{\nabla\vp}dxdt\\
    =\int_{\Rn}u_\epsilon(x,\tau)\overline{\vp(x,\tau)}dx+\int_{\tau}^T\overline{\langle \Dt\vp,u_\epsilon\rangle}_{\wt{W}^{-1,2},W^{1,2}}.
\end{multline}
Notice that $u_\epsilon\to u$ in $C((\tau,T),W^{1,2}(\Rn))$ (see \cite{HLMPLp} Theorem 4.9), and so letting $\epsilon\to 0^+$ we get
\begin{multline}\label{Appen_utest}
    \int_{\Rn}u(x,T)\overline{\vp(x,T)}dx+\int_{\tau}^T\int_{\Rn}A\nabla u\cdot\overline{\nabla\vp}dxdt\\
    =\int_{\Rn}u(x,\tau)\overline{\vp(x,\tau)}dx+\int_{\tau}^T\overline{\langle \Dt\vp,u\rangle}_{\wt{W}^{-1,2},W^{1,2}}.
\end{multline}
Letting $\vp=u_\epsilon$ in \eqref{Appen_ueps} and applying Lemma \ref{Evanslem} (ii) again, one obtains
\[
    \int_{\Rn}\abs{u_\epsilon(x,T)}^2dx+2\Re\int_{\tau}^T\int_{\Rn}A\nabla u_\epsilon\cdot\overline{\nabla u_\epsilon}dxdt
    =\int_{\Rn}\abs{u_\epsilon(x,\tau)}^2dx.
\]
By ellipticity and the definition of $u_\epsilon$, we have
\begin{multline*}
    2\lambda_0\int_{\tau}^T\int_{\Rn}\abs{\nabla u_\epsilon}^2dxdt\le\norm{e^{-\tau L}(u_{0,\epsilon})}_{L^2(\Rn)}^2\\
    \le2\norm{e^{-\tau L}(u_{0,\epsilon}-u_0)}_{L^2(\Rn)}^2+2\norm{e^{-\tau L}(u_{0})}_{L^2(\Rn)}^2.
\end{multline*}
Letting $\epsilon\to 0^+$, $\tau\to 0^+$, $T\to\infty$, we obtain $\int_0^{\infty}\int_{\Rn}\abs{\nabla u}^2dxdt\le\lambda_0^{-1}\norm{u_0}_{L^2}^2<\infty$. This enables us to take limit as $\tau$ go to $0^+$ on both sides of \eqref{Appen_utest} and get
\begin{multline*}
    \int_{\Rn}u(x,T)\overline{\vp(x,T)}dx+\int_0^T\int_{\Rn}A\nabla u\cdot\overline{\nabla\vp}dxdt\\
    =\int_{\Rn}u(x,0)\overline{\vp(x,0)}dx+\int_0^T\overline{\langle \Dt\vp,u\rangle}_{\wt{W}^{-1,2},W^{1,2}},
\end{multline*}
i.e. $u(x,t)$ is a weak solution of \eqref{Appen_IVP}. 

\vskip 0.08 in

\underline{Uniqueness.}

Let $v$ be a weak solution of \eqref{Appen_IVP}. We first show that 
$\Dt v\in L^2\br{(0,T),\wt{W}^{-1,2}(\Rn)}$ for any $T\in(0,\infty)$.
Define a semilinear functional $F$ on $L^2\br{[0,T],W^{1,2}(\Rn)}$ as follows:
for any $\vp\in L^2\br{[0,T],W^{1,2}(\Rn)}$, let 
\[
\langle F,\vp \rangle:=\int_0^T\int_{\Rn}A\nabla v\cdot\overline{\nabla\vp}dxdt.
\]
Obviously,
\begin{equation*}
    \abs{\langle F,\vp\rangle}\le C\norm{\nabla v}_{L^2\br{[0,T],L^2(\Rn)}}\norm{\nabla\vp}_{L^2\br{[0,T],L^2(\Rn)}}.
\end{equation*}
Then by Riesz representation theorem, there exists $w(x,t)\in L^2\br{[0,T],W^{1,2}(\Rn)}$ such that
\begin{align*}
    \langle F,\vp \rangle&=\int_0^T\int_{\Rn}(\nabla w\cdot\nabla\overline{\vp}+w\overline{\vp})dxdt\\
    &=\int_0^T\act{-\Delta w(\cdot,t)+w(\cdot,t),\vp}dt,
\end{align*}
and
\[
\norm{-\Delta w+w}_{L^2\br{[0,T],\wt{W}^{-1,2}(\Rn)}}\le\norm{w}_{L^2([0,T],W^{1,2}(\Rn))}\le C\norm{\nabla v}_{L^2\br{[0,T],L^2(\Rn)}}.
\]

Choose $\vp(x,t)=\Psi(x)\overline{\eta}(t)$ as a test function in \eqref{Appen_IVP}, where $\Psi\in W^{1,2}(\Rn)$, $\eta\in C_0^1\br{(0,T)}$. Then since $v$ is a weak solution, we have
\begin{align*}
 \int_0^T\br{v(\cdot,t),\Psi}_{L^2}\eta'(t)dt&=\int_0^T\int_{\Rn}A\nabla v\cdot\overline{\nabla\Psi}\eta(t)dxdt\\
&=\int_0^T\act{-\Delta w(\cdot,t)+w(\cdot,t),\Psi}\eta(t)dt.   
\end{align*}
Since $\Psi\in W^{1,2}(\Rn)$ is arbitrary,
\[
\int_0^Tv(x,t)\eta'(t)dt=\int_0^T(-\Delta w+w)\eta(t)dt\qquad\text{in }\wt{W}^{-1,2}(\Rn),
\]
which gives $\Dt v=\Delta w-w\in L^2\br{(0,T),\wt{W}^{-1,2}(\Rn)}$. Therefore, we can take $\vp=v$ as a test function in \eqref{Appen_IVP} and get
\[
\int_{\Rn}\abs{v(x,T)}^2+\int_0^T\int_{\Rn}A\nabla v\cdot\nabla\overline{v}dxdt=\int_0^T\overline{\langle \Dt v,v\rangle}_{\wt{W}^{-1,2},W^{1,2}}+\int_{\Rn}\abs{v(x,0)}^2dx.
\]
Using this and Lemma \ref{Evanslem} (ii), we have
\[
\int_{\Rn}\abs{v(x,T)}^2+2\Re\int_0^T\int_{\Rn}A\nabla v\cdot\nabla\overline{v}dxdt=\int_{\Rn}\abs{v(x,0)}^2dx.
\]
So we get
\begin{align*}
   \int_{\Rn}\abs{v(x,T)}^2+2\lambda_0\int_0^T\int_{\Rn}\abs{\nabla v}^2dxdt
   \le\int_{\Rn}\abs{v(x,0)}^2dx, 
\end{align*}
which implies that if $v(x,0)=0$ then $v\equiv 0$. 
\end{proof}

\begin{re}\label{Dt sol Re}
Let $u(x,t)$ be the weak solution to \eqref{Appen_IVP}. Since the coefficients are independent of $t$, a standard argument shows that $\Dt u$ is a weak solution to $\Dt v-\divg(A\nabla v)=0$ in $\Rn\times(0,\infty)$. That is, for any $T>0$, any $\vp\in L^2\br{[0,T],W^{1,2}(\Rn)}$ with $\Dt\vp\in L^2\br{[0,T],\wt{W}^{-1,2}(\Rn)}$ and $\vp=0$ when $0\le t\le\eps$ for some $0<\eps<T$,
\[
\int_{\Rn}\Dt u(x,T)\overline{\vp(x,T)}dx+\int_0^T\int_{\Rn}A\nabla(\Dt u)\cdot\overline{\nabla\vp}dxdt=\int_0^T\overline{\langle \Dt \vp,\Dt u\rangle}_{\wt{W}^{-1,2},W^{1,2}}dt.
\]
Moreover, since $\Dt^lu\in L^2_{\loc}\br{(0,\infty),L^2(\Rn)}$ and $\Dt^l\nabla u\in L^2_{\loc}\br{(0,\infty),L^2(\Rn)}$ for any $l\in\mathbb N$, 
one can show that for any $l\in\mathbb N$, $\Dt^l u$ is a weak solution to $\Dt v-\divg(A\nabla v)=0$ in $\Rn\times(0,\infty)$.
\end{re}

\section*{Acknowledgement}
We would like to thank the referee for a careful review and for insightful comments and questions on the manuscript that have allowed us to improve and clarify the exposition, and to correct some omissions. In particular, we thank the referee for noting that our proof of uniqueness actually yields a stronger Fatou-type theorem, which, following the referee's suggestion, we have now formulated as our Theorem \ref{Fatou thm}.

%

\bibliographystyle{plain}
\bibliography{reference}

\begin{thebibliography}{10}

\bibitem{Auscher2017parabolic}
Pascal Auscher, Moritz Egert, and Kaj Nystr\"om.
\newblock The {D}irichlet problem for second order parabolic operators in
  divergence form.
\newblock {\em Journal de l'Ecole Polytechnique}, 5:407--441, 2018.

\bibitem{auscher2002solution}
Pascal Auscher, Steve Hofmann, Michael Lacey, Alan McIntosh, and Philippe
  Tchamitchian.
\newblock The solution of the {K}ato square root problem for second order
  elliptic operators on $\mathbb{R}^n$.
\newblock {\em Annals of mathematics}, 156(2):633--654, 2002.

\bibitem{auscher2001extrapolation}
Pascal Auscher, Steve Hofmann, John~L Lewis, and Philippe Tchamitchian.
\newblock Extrapolation of {C}arleson measures and the analyticity of {K}ato's
  square-root operators.
\newblock {\em Acta mathematica}, 187(2):161--190, 2001.

\bibitem{caffarelli1981boundary}
L.~Caffarelli, E.~Fabes, S.~Mortola, and S.~Salsa.
\newblock Boundary behavior of nonnegative solutions of elliptic operators in
  divergence form.
\newblock {\em Indiana University Mathematics Journal}, 30(4):621--640, 1981.

\bibitem{caffarelli1981completely}
Luis~A Caffarelli, Eugene~B Fabes, and Carlos~E Kenig.
\newblock Completely singular elliptic-harmonic measures.
\newblock {\em Indiana University Mathematics Journal}, 30(6):917--924, 1981.

\bibitem{coifman1993compensated}
R.~Coifman, P.L. Lions, Y.~Meyer, and S.~Semmes.
\newblock Compensated compactness and {H}ardy spaces.
\newblock {\em J. Math. Pures Appl.}, 72(9):247--286, 1993.

\bibitem{escauriaza2018kato}
Luis Escauriaza and Steve Hofmann.
\newblock Kato square root problem with unbounded leading coefficients.
\newblock {\em Proceedings of the American Mathematical Society},
  146(12):5295--5310, 2018.

\bibitem{evans1998partial}
Lawrence~C. Evans.
\newblock {\em Partial Differential Equations}, volume~19.
\newblock American Mathematical Society, 1998.

\bibitem{giaquinta1983multiple}
Mariano Giaquinta.
\newblock {\em Multiple Integrals in the Calculus of Variations and Nonlinear
  Elliptic Systems}.
\newblock Number 105. Princeton University Press, 1983.

\bibitem{hofmann2015square}
Steve Hofmann, Carlos Kenig, Svitlana Mayboroda, and Jill Pipher.
\newblock Square function/non-tangential maximal function estimates and the
  {D}irichlet problem for non-symmetric elliptic operators.
\newblock {\em Journal of the American Mathematical Society}, 28(2):483--529,
  2015.

\bibitem{hofmann2002solution}
Steve Hofmann, Michael Lacey, and Alan McIntosh.
\newblock The solution of the {K}ato problem for divergence form elliptic
  operators with {G}aussian heat kernel bounds.
\newblock {\em Annals of Mathematics}, 156:623--631, 2002.

\bibitem{hofmann2018carleson}
Steve Hofmann, Phi Le, and Andrew~J Morris.
\newblock Carleson measure estimates and the {D}irichlet problem for degenerate
  elliptic equations.
\newblock {\em Analysis \& PDE}, 12(8):2095--2146, 2019.

\bibitem{HLMPLp}
Steve Hofmann, Linhan Li, Svitlana Mayboroda, and Jill Pipher.
\newblock {$L^p$} theory for the square roots and square functions of elliptic
  operators having a {BMO} anti-symmetric part.
\newblock {\em arXiv preprint arXiv:1908.01030}, 2019.

\bibitem{jerison1981dirichlet}
David~S Jerison and Carlos~E Kenig.
\newblock The {D}irichlet problem in non-smooth domains.
\newblock {\em Annals of mathematics}, 113(2):367--382, 1981.

\bibitem{kenig2016square}
Carlos Kenig, B~Kirchheim, J~Pipher, and T~Toro.
\newblock Square functions and the {$A_{\infty}$} property of elliptic
  measures.
\newblock {\em The Journal of Geometric Analysis}, 26(3):2383--2410, 2016.

\bibitem{kenig2000new}
Carlos Kenig, Herbert Koch, Jill Pipher, and Tatiana Toro.
\newblock A new approach to absolute continuity of elliptic measure, with
  applications to non-symmetric equations.
\newblock {\em Advances in Mathematics}, 153(2):231--298, 2000.

\bibitem{kenig1994harmonic}
Carlos~E Kenig.
\newblock {\em Harmonic analysis techniques for second order elliptic boundary
  value problems}, volume~83.
\newblock American Mathematical Soc., 1994.

\bibitem{li2019boundary}
Linhan Li and Jill Pipher.
\newblock Boundary behavior of solutions of elliptic operators in divergence
  form with a {BMO} anti-symmetric part.
\newblock {\em Communications in Partial Differential Equations},
  44(2):156--204, 2019.

\bibitem{lieberman1996second}
Gary~M Lieberman.
\newblock {\em Second order parabolic differential equations}.
\newblock World scientific, 1996.

\bibitem{modica1980construction}
Luciano Modica and Stefano Mortola.
\newblock Construction of a singular elliptic-harmonic measure.
\newblock {\em manuscripta mathematica}, 33(1):81--98, 1980.

\bibitem{qian2019parabolic}
Zhongmin Qian and Guangyu Xi.
\newblock Parabolic equations with singular divergence-free drift vector
  fields.
\newblock {\em Journal of the London Mathematical Society}, 100(1):17--40,
  2019.

\bibitem{seregin2012divergence}
Gregory Seregin, Luis Silvestre, Vladim{\'\i}r {\v{S}}ver{\'a}k, and Andrej
  Zlato{\v{s}}.
\newblock On divergence-free drifts.
\newblock {\em Journal of Differential Equations}, 252(1):505--540, 2012.

\end{thebibliography}

\Addresses

\end{document}